\documentclass[aos]{imsart}
\usepackage{amsmath,amssymb,amsthm}
\usepackage[numbers]{natbib}
\usepackage{graphicx,epsfig,epstopdf,array,multirow}

\usepackage{xr}

\startlocaldefs
\newtheorem{theorem}{Theorem}[section]
\newtheorem{corollary}{Corollary}[section]
\newtheorem{lemma}{Lemma}[section]

\theoremstyle{definition}

\newtheorem{remark}{Comment}[section]
\numberwithin{remark}{section}

\newcommand{\eps}{\varepsilon}

\renewcommand{\(}{\left(}
\renewcommand{\)}{\right)}
\renewcommand{\hat}{\widehat}

\newcommand{\Ep}{{\mathrm{E}}}
\newcommand{\barEp}{\bar \Ep}
\newcommand{\En}{{\mathbb{E}_n}}

\renewcommand{\Pr}{{\mathrm{P}}}

\newcommand{\RR}{\mathbb{R}}

\def\x{{x}}

\renewcommand{\hat}{\widehat}
\renewcommand{\leq}{\leqslant}
\renewcommand{\geq}{\geqslant}
\renewcommand{\qed}{\hfill{\tiny \ensuremath{\blacksquare} }}%

\DeclareMathOperator{\pr}{pr}

\endlocaldefs

\begin{document}

\begin{frontmatter}

\title{Gaussian Approximations and Multiplier Bootstrap for Maxima of Sums of High-Dimensional Random Vectors\thanksref{T1}}
\runtitle{Gaussian Approximations and multiplier bootstrap}
\thankstext{T1}{Date:  June, 2012. Revised June, 2013. V. Chernozhukov and D. Chetverikov  are supported by a National Science Foundation grant. K. Kato is supported by the Grant-in-Aid for Young Scientists (B) (25780152), the Japan Society for the Promotion of Science.}


\begin{aug}
\author{\fnms{Victor} \snm{Chernozhukov}\thanksref{m1}\ead[label=e1]{vchern@mit.edu}},
\author{\fnms{Denis} \snm{Chetverikov}\thanksref{m2}\ead[label=e2]{chetverikov@econ.ucla.edu}}
\and
\author{\fnms{Kengo} \snm{Kato}\thanksref{m3}
\ead[label=e3]{kkato@e.u-tokyo.ac.jp}}

\runauthor{Chernozhukov Chetverikov Kato}

\affiliation{MIT\thanksmark{m1}, UCLA\thanksmark{m2}, and University of Tokyo\thanksmark{m3}}

\address{Department of Economics and\\
Operations Research Center, MIT \\
50 Memorial Drive \\
Cambridge, MA 02142, USA.\\
\printead{e1}}

\address{Department of Economics, UCLA\\
Bunche Hall, 8283 \\
315 Portola Plaza \\
Los Angeles, CA 90095, USA.\\
\printead{e2}}

\address{Graduate School of Economics \\
University of Tokyo \\
7-3-1 Hongo, Bunkyo-ku\\
Tokyo 113-0033, Japan. \\
\printead{e3}}
\end{aug}

\begin{abstract}
{We derive a Gaussian approximation result for the maximum of a sum of high dimensional random vectors.
Specifically, we establish conditions under which   the distribution of the maximum is approximated by that of the maximum of a sum of the Gaussian random vectors with the same covariance matrices as the original vectors. This result  applies  when the dimension of random vectors ($p$) is large compared to the sample size ($n$);  in fact, $p$ can be much larger than $n$, without restricting correlations of the coordinates of these vectors. We also show that the distribution of the maximum of a sum of the  random vectors with unknown covariance matrices can be consistently estimated by the distribution of the maximum of a sum of the conditional Gaussian random vectors obtained by multiplying the original vectors with i.i.d. Gaussian multipliers. This is the Gaussian multiplier (or wild) bootstrap procedure. Here too, $p$ can be large or even much larger than $n$.
These distributional approximations, either Gaussian or conditional Gaussian, yield a high-quality approximation to the distribution of the original maximum,
often with approximation error decreasing polynomially in the sample size, and hence are of interest in many applications.
We demonstrate how our Gaussian approximations and the multiplier bootstrap can be used for modern high dimensional estimation,  multiple hypothesis testing, and adaptive specification testing. All these results contain non-asymptotic bounds on approximation errors.}
\end{abstract}

\begin{keyword}[class=AMS]
\kwd{62E17}
\kwd{62F40}
\end{keyword}

\begin{keyword}
\kwd{Dantzig selector}
\kwd{Slepian}
\kwd{Stein method}
\kwd{maximum of vector sums}
\kwd{high dimensionality}
\kwd{anti-concentration}
\end{keyword}

\end{frontmatter}

\section{Introduction}
Let $x_{1},\dots,x_{n}$ be independent random vectors in $\RR^{p}$, with each $x_i$ having coordinates denoted by $x_{ij}$, that is, $x_{i}= (x_{i1},\dots,x_{ip})'$.
Suppose that each $x_{i}$ is centered, namely $\Ep[x_i]=0$, and has a finite covariance matrix $\Ep[x_i x_i']$.
Consider the rescaled sum:
\begin{equation}\label{eq: define X}
X := (X_{1},\dots,X_{p})' := \frac{1}{\sqrt{n}} \sum_{i=1}^n x_{i}.
\end{equation}
Our goal is to obtain a distributional approximation for the statistic $T_0$  defined as the maximum coordinate of vector $X$:
\begin{equation*}
T_0:= \max_{1\leq j \leq p} X_{j}.
\end{equation*}
The distribution of $T_0$ is of interest in many applications.
When $p$ is fixed, this distribution can be approximated by
the classical Central Limit Theorem (CLT) applied to $X$. However, in modern applications (cf. \cite{BV11}),
$p$ is often comparable or even larger than $n$, and the classical CLT does not apply in such cases.
This paper provides a tractable approximation to the distribution of $T_0$ when $p$ can be large and possibly much larger than $n$.

The \textit{first} main result of the paper is the Gaussian approximation result (GAR), which bounds the Kolmogorov distance between
the distributions of $T_0$  and its Gaussian analog $Z_0$.
Specifically, let $y_{1},\dots,y_{n}$ be independent centered Gaussian random vectors in $\RR^{p}$ such that each $y_{i}$ has the same covariance matrix as $x_{i}$: $y_{i} \sim N(0,\Ep [ x_{i} x_{i}' ])$. Consider the rescaled sum of these vectors:
\begin{equation}\label{eq: define Y}
Y := (Y_{1},\dots,Y_{p})' :=  \frac{1}{\sqrt{n}} \sum_{i=1}^n y_{i}.
 \end{equation}
 Vector $Y$ is the Gaussian analog of $X$ in the sense of sharing the same mean and covariance matrix, namely
$\Ep[X] = \Ep[Y] = 0$ and $\Ep[XX'] = \Ep[YY'] =  n^{-1}\sum_{i=1}^n \Ep[x_i x_i'].$
We then define the Gaussian analog $Z_0$ of $T_0$ as the maximum coordinate of vector $Y$:
\begin{equation}\label{eq: define Z}
Z_{0} := \max_{1 \leq j \leq p} Y_{j}.
\end{equation}
We show that, under suitable moment assumptions,  as $n \to \infty$ and possibly $p=p_{n} \to \infty$,
\begin{equation}
\rho:= \sup_{t \in \RR} \left| \Pr( T_0 \leq t ) - \Pr ( Z_0 \leq t ) \right| \leq C n^{- c} \to 0, \label{eq: main result}
 \end{equation}
where constants $c > 0$  and $C > 0$ are independent of $n$.

Importantly, in  (\ref{eq: main result}), $p$ can be large in comparison to $n$ and be as large as $e^{o(n^{c})}$ for some $c>0$.  For example, if $x_{ij}$ are uniformly bounded  (namely, $|x_{ij}| \leq C_{1}$ for some constant $C_{1} > 0$ for all $i$ and $j$) the Kolmogorov distance $\rho$ converges to zero at a polynomial rate whenever $(\log p)^7/n \to 0$ at a polynomial rate.  We obtain similar results when $x_{ij}$ are sub-exponential and even non-sub-exponential under suitable moment assumptions.  Figure \ref{fig: two deviations} illustrates the result (\ref{eq: main result}) in a non-sub-exponential example, which is motivated by the analysis of the Dantzig selector of \cite{CandesTao2007} in non-Gaussian settings (see Section \ref{sec: Dantzig}).
\begin{figure}
\label{fig: two deviations}
\includegraphics[width=4in]{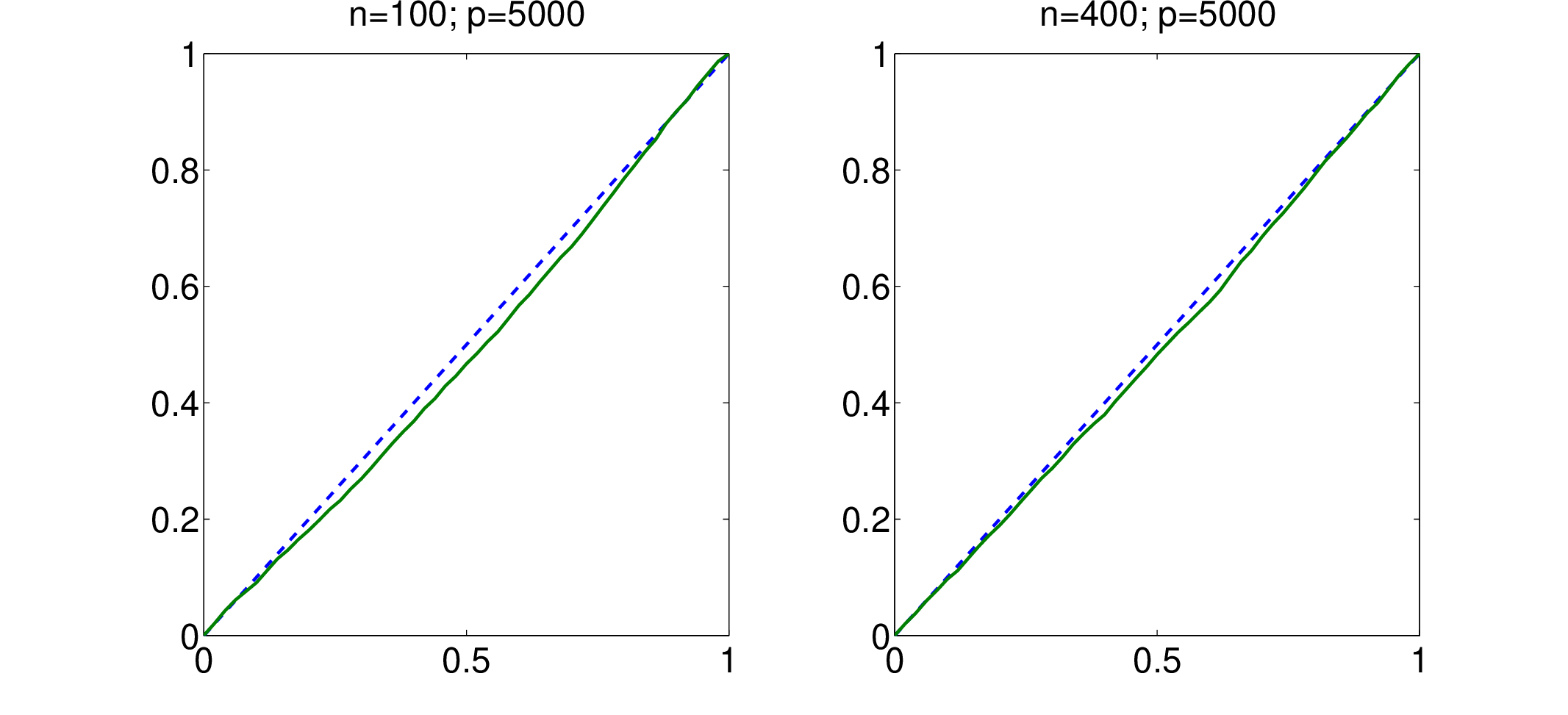}
\caption{\footnotesize
P-P plots comparing distributions of $T_0$ and $Z_0$ in the example motivated by the problem of selecting the penalty level of the Dantzig selector. Here $x_{ij}$ are generated as
$x_{ij}=z_{ij}\eps_{i}$ with $\eps_{i} \sim t(4),$ (a $t$-distribution with four degrees of freedom), and $z_{ij}$ are non-stochastic (simulated once using $U[0,1]$ distribution independently across $i$ and $j$).  The dashed line is 45$^\circ$. The distributions of $T_0$ and $Z_0$ are close, as (qualitatively) predicted by the GAR derived in the paper. The quality of the Gaussian approximation is particularly good
for the tail probabilities, which is most relevant for practical applications.}
\end{figure}

The proof of the Gaussian approximation result (\ref{eq: main result}) builds on a number of technical tools such as Slepian's  smart path interpolation (which is
related to the solution of  Stein's partial differential equation; see Appendix \ref{sec: Slepian-Stein note} of the Supplementary Material (SM; \cite{CCK13})), Stein's leave-one-out method, approximation of maxima by the smooth potentials (related to ``free energy" in spin glasses) and using some fine or subtle properties of such approximation, and exponential inequalities for self-normalized sums. See, for example,
\cite{Slepian1962, Stein1981, Dudley1999, ChenGoldsteinShao2011, Talagrand2003, Chatterjee2005b, Rollin2011, Shao2009, Panchenko2013} for introduction and prior uses of some of these tools.
The proof also critically relies on the anti-concentration and comparison bounds of maxima of Gaussian vectors  derived in \cite{ChernozhukovChetverikovKato2012c} and restated in this paper as Lemmas \ref{lem: anticoncentration} and \ref{lemma: distances Gaussian to Gaussian}.

Our new Gaussian approximation theorem has the following innovative features.
First, we provide a general result that establishes that  maxima of sums of random vectors can be approximated in distribution by the maxima of sums of Gaussian random vectors when $p \gg n$ and especially when $p$ is of order $e^{o(n^{c})}$ for some $c > 0$. The existing techniques can also lead to results of the form (\ref{eq: main result}) when $p=p_{n} \to \infty$, but under much stronger conditions on $p$ requiring $p^{c}/n \to 0$; see Example 17 (Section 10) in \cite{Pollard2002}.   Some high-dimensional cases where $p$ can be of order $e^{o(n^{c})}$ can also be handled via Hungarian couplings, extreme value theory or other methods, though special structure is required (for a detailed review, see Section \ref{sec: literature review} of the SM \cite{CCK13}).   Second, our Gaussian approximation theorem covers cases where $T_{0}$ does not have  a limit distribution as $n \to \infty$ and $p=p_{n} \to \infty$. In some cases, after a suitable normalization, $T_{0}$ could have an extreme value distribution as a limit distribution, but the approximation to an extreme value distribution requires some restrictions on the dependency structure among the coordinates in $x_{i}$. Our result does not limit the dependency structure. We also emphasize that our theorem specifically covers cases where the process $\{\sum_{i=1}^nx_{ij}/\sqrt{n},1\leq j\leq p\}$ is not asymptotically Donsker (i.e., can't be embedded into a path of an empirical process that is Donsker).
 Otherwise, our result would follow from the classical functional central limit theorems for empirical processes, as in \cite{Dudley1999}. Third, the quality of approximation in (\ref{eq: main result}) is of polynomial order in $n$, which is better than the logarithmic in $n$ quality that we could obtain in
some (though not all) applications using the approximation of the distribution of $T_{0}$ by an extreme value distribution (see \cite{Leadbetter1983}).

Note that the result (\ref{eq: main result}) is immediately useful for inference with statistic $T_0$,
even though $\Pr (Z_{0} \leq t )$ needs not converge itself to a well-behaved distribution function.  Indeed,  if the covariance matrix
$n^{-1} \sum_{i=1}^n \Ep[x_{i} x_{i}']$ is known, then $c_{Z_0}(1-\alpha): = (1-\alpha)$-quantile of $Z_0$,  can be computed numerically, and we have
\begin{equation}
\label{eq: inference}
|\Pr ( T_0 \leq c_{Z_0}(1-\alpha) )  - (1-\alpha)|  \leq C n^{-c} \to 0.
 \end{equation}
 %

The \textit{second} main result of the paper establishes
validity of the multiplier (or Wild) bootstrap for estimating
quantiles of $Z_0$ when the covariance matrix $n^{-1} \sum_{i=1}^n \Ep[x_{i} x_{i}']$
 is unknown.   Specifically, we define
the Gaussian-symmetrized version $W_0$ of $T_0$  by multiplying $x_{i}$ with i.i.d. standard Gaussian
random variables $e_{1},\dots,e_{n}$:
\begin{equation}
W_0:=\max_{1\leq j\leq p}\frac{1}{\sqrt{n}}\sum_{i=1}^{n}x_{ij}e_{i}. \label{eq: main quantity}
\end{equation}
We show that the conditional quantiles of $W_0$ given data $( x_{i} )_{i=1}^{n}$ are able to
consistently estimate the quantiles of $Z_0$ and hence those of $T_0$ (where the notion
of consistency used is the one that guarantees asymptotically valid inference).  Here the primary factor driving the bootstrap estimation error is the maximum difference between the empirical and population covariance matrices:
\begin{equation*}
\Delta:=\max_{1 \leq j, k \leq p} \left |  \frac{1}{n} \sum_{i=1}^n  (x_{ij} x_{ik}  - \Ep[x_{ij} x_{ik}] ) \right |,
\end{equation*}
which can converge to zero even when  $p$ is much larger than $n$.  For example, when $x_{ij}$ are uniformly bounded, the multiplier
bootstrap is valid for inference if $(\log p)^7/n \to 0$.  Earlier related results on bootstrap in the ``$p \to \infty$ but $p/n \to 0$'' regime
were obtained in \cite{Mammen1993}; interesting results on  inference on the mean vector of high-dimensional random vectors when $p \gg n$ based on concentration inequalities and symmetrization are obtained in \cite{ArlotBlanchardRoquain2010a,ArlotBlanchardRoquain2010b},
albeit the approach and results are quite different from those given here. In particular, in \cite{ArlotBlanchardRoquain2010a}, either Gaussianity or symmetry in distribution is imposed on the data.

The key motivating example of our analysis is the analysis of construction of one-sided or two-sided uniform
confidence band for high-dimensional means under non-Gaussian assumptions.    This requires estimation of a high quantile of the maximum of sample means.       We give two concrete applications.  One application deals with  high-dimensional sparse regression model. In this model, \cite{CandesTao2007} and \cite{BickelRitovTsybakov2009} assume Gaussian errors to analyze the Dantzig selector, where the high-dimensional means enter the constraint in the problem.  Our results show that Gaussianity is not necessary and the sharp, Gaussian-like, conclusions hold approximately, with just the fourth moment of the
regression errors being bounded.  Moreover, our
approximation allows to take into account correlations
 among the regressors. This leads to  a better choice of the penalty level
 and tighter bounds on performance than those that had been available previously.
In another example we apply our results in the multiple hypothesis testing via
 the step-down method of \cite{RomanoWolf05}.  In the SM \cite{CCK13} we also provide an application to adaptive specification testing. In either case the number of hypotheses to be tested or  the number of moment restrictions to be tested can be much larger than the sample size. Lastly, in a companion work (\cite{ChernozhukovChetverikovKato2012b}), we derive the strong coupling for suprema of general empirical processes based on the methods developed here and maximal inequalities. These results represent a useful complement to the results based on the Hungarian coupling developed by \cite{KomlosMajorTusnady1975,BretagnolleMassart1989, Koltchinskii1994,Rio1994} for the entire empirical process and have applications to inference in nonparametric problems such as construction of uniform confidence bands and testing qualitative hypotheses (see, e.g., \cite{GineNickl2010}, \cite{Spokoiny}, and \cite{Chetverikov2012}).

\subsection{Organization of the paper}In Section \ref{sec: Gaus vs NonGaus}, we give the results on Gaussian approximation, and in Section \ref{sec: multiplier bootstrap} on the multiplier bootstrap. In Sections \ref{sec: Dantzig} and \ref{sub: MHT}, we develop applications to the Dantzig selector and multiple testing.  Appendices \ref{sec: auxiliary lemmas}-\ref{sec: multiplier bootstrap proofs} contain proofs for each of these sections, with Appendix \ref{sec: auxiliary lemmas} stating auxiliary tools and lemmas. Due to the space limitation, we put additional results and proofs into the SM \cite{CCK13}. In particular, Appendix \ref{sub: AST} of the SM provides additional application to adaptive specification testing. Results of Monte Carlo simulations are presented in Appendix \ref{sec: monte carlo} of the SM.


\subsection{Notation}
In what follows, unless otherwise stated, we will assume that $p \geq 3$. In making asymptotic statements, we assume that
$n \to \infty$ with understanding that $p$ depends on $n$ and possibly $p \to \infty$ as $n \to \infty$.
Constants $c,C,c_{1},C_{1},c_{2},C_{2},\dots$ are understood to be independent of $n$.
Throughout the paper, $\En[\cdot]$ denotes the
average over index $1 \leq i \leq n$, that is, it simply abbreviates the notation $n^{-1} \sum_{i=1}^n[\cdot]$. For example, $\En[x_{ij}^{2}]$ $=$ $n^{-1} \sum_{i=1}^{n}x_{ij}^{2}$. In addition, $\barEp[\cdot]=\En[\Ep[\cdot]]$. For example, $\barEp[x_{ij}^2]$ $=$ $n^{-1}\sum_{i=1}^n\Ep[x_{ij}^2]$. For $z\in\RR^p$, $z^\prime$ denotes the transpose of $z$.
For a function $f:\RR \to \RR$, we write $\partial^{k} f (x) = \partial^{k} f(x) / \partial x^{k}$ for nonnegative integer $k$;  for a function $f:\RR^p \to \RR$, we write $\partial_{j} f (x) = \partial f(x)/ \partial x_{j}$ for $j=1,\dots,p$, where $x = (x_{1},\dots,x_{p})'$.
We denote by $C^k(\RR)$ the class of $k$ times continuously differentiable functions from $\RR$ to itself, and denote by $C_{b}^{k}(\RR)$ the class of all functions $f \in C^{k}(\RR)$ such that $\sup_{z \in \RR} | \partial^{j} f(z) | < \infty$ for $j=0,\dots,k$.
We write $a \lesssim b$ if $a$ is smaller than or equal to $b$ up to a universal positive constant. For $a,b \in \RR$, we write $a \vee b = \max \{ a,b \}$.
For two sets $A$ and $B$, $A \ominus B$ denotes their symmetric difference, that is, $A \ominus B = (A \backslash B) \cup (B \backslash A)$.

\section{Gaussian Approximations for Maxima of Non-Gaussian Sums}
\label{sec: Gaus vs NonGaus}

The purpose of this section is to compare and bound the difference
between the expectations and distribution functions of the non-Gaussian to  Gaussian maxima:
\begin{equation*}
T_{0} := \max_{1\leq j \leq p} X_{j}   \  \text{ and } \  Z_{0}:=\max_{1\leq j \leq p} Y_{j},
\end{equation*}
where vector $X$ is defined in equation (\ref{eq: define X}) and
$Y$ in equation (\ref{eq: define Y}).  Here and in what follows, without loss of generality, we will assume that $(x_{i})_{i=1}^{n}$ and $(y_{i})_{i=1}^{n}$ are independent.
In order to derive the main result of this section, we shall employ Slepian interpolation, Stein's leave-one-out method,  a truncation method combined with self-normalization, as well as some fine properties of the smooth max function (such as ``stability").
(The relative complexity of the approach is justified in Comment \ref{comment: warmup} below.)

The following bounds on moments will be used in stating the bounds in Gaussian approximations:
\begin{equation}\label{define M and S}
\quad M_k:=\max_{1\leq j\leq p}(\barEp[|x_{ij}|^k])^{1/k}.
\end{equation}

The problem of comparing distributions of maxima is of intrinsic difficulty since the maximum function $z=(z_{1},\dots,z_{p})' \mapsto \max_{1 \leq j \leq p} z_{j}$ is non-differentiable.
To circumvent the problem, we use a smooth approximation of the maximum function. For $z=(z_{1},\dots,z_{p})' \in \RR^{p}$, consider the function:
\begin{equation*}
F_{\beta}(z):=\beta^{-1}\log\left(\sum_{j=1}^{p}\exp(\beta z_{j})\right),
\end{equation*}
where $\beta > 0$ is the smoothing parameter that controls the level of approximation (we call this function the ``smooth max function'').
An elementary calculation shows that for all $z \in \RR^{p}$,
\begin{equation}\label{eq: smooth max property}
0 \leq  F_{\beta}(z)- \max_{1 \leq j \leq p} z_{j} \leq \beta^{-1} \log p.
\end{equation}
This smooth max function arises in the definition of ``free energy" in spin glasses; see, for example, \cite{Talagrand2003}.
Some important properties of this function, such as stability, are derived in the Appendix.

Given a threshold
level $u>0$,  we define a truncated version of $x_{ij}$ by
\begin{equation}
\tilde x_{ij} = x_{ij} 1 \left\{ |x_{ij}| \leq u (\barEp[x_{ij}^2])^{1/2} \right \} - \Ep\left [x_{ij} 1  \left \{|x_{ij}| \leq u (\barEp[x_{ij}^2])^{1/2} \right \}  \right ]. \label{def: truncated x}
\end{equation}
Let $\varphi_x(u)$ be the infimum, which is attained, over all numbers $\varphi \geq 0$ such that
\begin{equation}
\barEp\left[x_{ij}^21\left\{|x_{ij}|> u(\barEp[x_{ij}^2])^{1/2}\right\}\right] \leq \varphi^{2} \barEp[x_{ij}^{2}].  \label{def: truncation varphi}
\end{equation}
Note that the function $\varphi_{x}(u)$ is right-continuous; it measures  the impact of truncation
on second moments.  Define  $u_x(\gamma)$ as the infimum over all numbers $u \geq 0$ such that
\begin{equation*}
\Pr\left ( |x_{ij}| \leq u (\barEp[x_{ij}^2])^{1/2}, 1 \leq i \leq n, 1 \leq j \leq p \right) \geq 1-\gamma.
\end{equation*}
Also define $\varphi_y(u)$ and $u_y(\gamma)$ by the corresponding quantities
for the analogue Gaussian case, namely with  $(x_i)_{i=1}^n$ replaced by $(y_i)_{i=1}^n$ in the above definitions.
Throughout the paper we use the following quantities:
\begin{equation*}
\varphi(u): = \varphi_x(u) \vee \varphi_y(u), \ \ u(\gamma) := u_x(\gamma) \vee u_y(\gamma).
\end{equation*}
Also, in what follows, for a smooth function $g: \RR \to \RR$, write
\begin{equation*}
G_{k} : = \sup_{z \in \RR} |\partial^{k} g(z)|, \ k \geq 0.
\end{equation*}
The following theorem is the main building block toward deriving a result of the form (\ref{eq: main result}).

\begin{theorem}[Comparison of Gaussian to Non-Gaussian Maxima] \label{theorem:comparison non-Gaussian}
Let $\beta>0, u > 0$ and $\gamma \in (0,1)$  be such that $2\sqrt{2}u M_2 \beta/\sqrt{n} \leq 1$ and $u \geq u(\gamma)$.
Then for every $g \in C_{b}^3(\RR)$,  $|\Ep[g(F_{\beta}(X))- g(F_{\beta}(Y))]| \lesssim D_{n}(g,\beta,u,\gamma)$, so that
\begin{align*}
&|\Ep[g(T_{0})- g(Z_{0})]| \lesssim D_{n}(g,\beta,u,\gamma) + \beta^{-1} G_1 \log p,
\end{align*}
where
\begin{align*}
D_{n}(g,\beta,u,\gamma) &:= n^{-1/2} (G_3  +  G_2 \beta + G_1  \beta^2) M_3^3 + (G_2 +  \beta G_1) M_2^2 \varphi(u)\\
&\quad +  G_1  M_2  \varphi(u) \sqrt{\log (p/\gamma)} +  G_0 \gamma.
\end{align*}
\end{theorem}

We will also invoke the following lemma, which is proved in \cite{ChernozhukovChetverikovKato2012c}.

 \begin{lemma}[Anti-Concentration]\label{lem: anticoncentration}
(a) Let $Y_{1},\dots,Y_{p}$ be jointly Gaussian random variables with $\Ep[Y_{j}]=0$ and $\sigma_{j}^{2} := \Ep [Y_{j}^{2} ]  > 0$ for all $1 \leq j \leq p$, and  let $a_{p} := \Ep [ \max_{1 \leq j \leq p} (Y_{j}/\sigma_{j}) ]$.
Let $\underline{\sigma} = \min_{1 \leq j \leq p} \sigma_{j}$ and $\bar{\sigma} = \max_{1 \leq j \leq p} \sigma_{j}$. Then for every $\varsigma > 0$,
\begin{equation*}
\sup_{z \in \RR} \Pr \left( | \max_{1 \leq j \leq p} Y_{j} - z|  \leq   \varsigma \right) \leq C \varsigma \{ a_p + \sqrt{1 \vee \log (\underline{\sigma}/\varsigma)} \},
\end{equation*}
where $C>0$ is  a constant depending only on $\underline{\sigma}$ and $\bar{\sigma}$.
When $\sigma_{j}$ are all equal,  $\log (\underline{\sigma}/\varsigma)$ on the right side can be replaced by $1$.  (b) Furthermore,
the worst case bound is obtained by bounding $a_p$ by $\sqrt{2 \log p}$.
\end{lemma}

By Theorem \ref{theorem:comparison non-Gaussian} and Lemma \ref{lem: anticoncentration}, we can obtain  a bound on the Kolmogorov distance, $\rho$, between the distribution functions of $T_{0}$ and $Z_{0}$, which is the main theorem of this section.

\begin{theorem}[\textbf{Main Result 1: Gaussian Approximation}]
\label{cor: Gaussian to nonGaussian KS 2}
Suppose that there are some constants $0 < c_{1} < C_{1}$  such that $c_{1} \leq \barEp[x_{ij}^2] \leq C_{1}$ for all $1\leq j \leq p$. Then for every $\gamma \in (0,1)$,
\begin{equation*}
\rho \leq
C \left \{  n^{-1/8} (M_{3}^{3/4} \vee M_{4}^{1/2} ) (\log (pn/\gamma))^{7/8} + n^{-1/2} (\log (pn/\gamma))^{3/2} u(\gamma) + \gamma \right \},
\end{equation*}
 where $C>0$ is a constant that depends on $c_1$ and $C_1$ only.
\end{theorem}

\begin{remark}[Removing lower bounds on the variance]\label{comment:relaxed}
The condition that $\barEp[x_{ij}^2]\geq c_1$ for {\em all} $1\leq j\leq p$ can not be removed in general.
 However, this condition becomes redundant, if there is at least a nontrivial fraction of components $x_{ij}$'s of vector $x_{i}$ with variance bounded away from zero and all pairwise correlations bounded away from 1: for some $J\subset\{1,\dots,p\}$,
 $$
 |J| \geq \nu p,    \ \  \barEp[x_{ij}^2]\geq c_1,  \ \ \frac{|\barEp[x_{ij}x_{ik}] |}{ \sqrt{\barEp[x_{ij}^2] }\sqrt{\barEp[x_{ik}^2]}} \leq 1- \nu', \ \  \forall  (k, j) \in J \times J: k \neq j,
 $$
where $\nu>0$ and $\nu'>0$ are some constants independent of $n$ or $p$.   Section \ref{sec: low variance} of the SM \cite{CCK13} contains formal results  under this condition. \qed
\end{remark}

In applications, it is useful to have explicit bounds on the upper function $u(\gamma)$. To this end,
let  $h: [0,\infty) \to [0,\infty)$  be a {\em Young-Orlicz modulus}, that is, a convex and strictly increasing function with $h(0) = 0$. Denote by $h^{-1}$ the inverse function of $h$. Standard examples include the power function
 $h(v) =v^q$ with inverse  $h^{-1}(\gamma) = \gamma^{1/q}$  and the exponential function $h(v) = \exp(v) -1$ with  inverse
 $h^{-1}(\gamma) = \log(\gamma+1)$.  These functions describe how many moments the random variables have; for example,
a random variable $\xi$ has finite $q$th moment if $\Ep[|\xi|^q] < \infty$, and is sub-exponential if $\Ep[\exp(|\xi|/C)] < \infty$ for some $C > 0$. We refer to  \cite{VW96}, Chapter 2.2, for further details.

\begin{lemma}[Bounds on the upper function $u(\gamma)$]\label{lem: bound on u}
Let $h: [0,\infty) \to [0,\infty)$ be a Young-Orlicz modulus, and  let $B > 0$ and $D > 0$ be constants such that $(\Ep [ x_{ij}^{2} ])^{1/2} \leq B$ for all $1 \leq i \leq n, 1 \leq j \leq p$, and $\bar \Ep [ h( \max_{1 \leq j \leq p} | x_{ij} |/ D )] \leq 1$. Then under the condition of Theorem \ref{cor: Gaussian to nonGaussian KS 2},
\begin{equation*}
u (\gamma) \leq C \max \{ D h^{-1}(n/\gamma), B \sqrt{\log  (pn/\gamma)} \},
\end{equation*}
where $C>0$ is a constant that depends on $c_1$ and $C_1$ only.
\end{lemma}

In applications, parameters $B$ and $D$ (with $M_3$ and $M_4$ as well) are allowed to increase with $n$.
The size of these parameters and the choice of the Young-Orlicz modulus are case-specific.

\subsection{Examples}\label{sub: examples of applications GAR}

The purpose of this subsection is to obtain bounds on $\rho$ for various leading examples frequently encountered in applications.
We are concerned with simple conditions under which $\rho$ decays polynomially in $n$.

Let $c_{1} > 0$ and $C_{1} > 0$ be some constants, and  let $B_{n} \geq 1$ be a sequence of constants.
We allow for the case where $B_{n} \to \infty$ as $n \to \infty$.
We shall first consider applications where one of the following  conditions is satisfied   {\em uniformly in} $1 \leq i \leq n$ and $1 \leq j \leq p$:



\begin{itemize}
\item[(E.1)] $c_{1} \leq \barEp[x^2_{ij}] \leq C_{1}$ and $\displaystyle \max_{k =1,2}\barEp[|x_{ij}|^{2+k}/B^k_n] + \Ep[\exp(|x_{ij}|/B_n)] \leq 4$;
\item[(E.2)] $c_{1} \leq \barEp[x^2_{ij}] \leq C_{1}$ and  $\displaystyle  \max_{k =1,2}\barEp[|x_{ij}|^{2+k}/B^k_n] +\Ep[ (\max_{1\leq j \leq p} |x_{ij}| / B_{n})^4] \leq 4$.
\end{itemize}

\begin{remark}
As a rather special case, Condition (E.1) covers vectors $x_i$ made up from sub-exponential random variables, that is, $$\barEp [ x_{ij}^{2} ] \geq c_{1} \text{ and }\Ep[\exp ( |x_{ij}|/C_{1}) ]  \leq 2$$
  (set $B_n = C_1$), which in turn includes, as a special case, vectors $x_i$ made up from sub-Gaussian random variables.  Condition (E.1) also covers the case when  $|x_{ij}| \leq B_n$ for all $i$ and $j$, where $B_n$ may increase with $n$. Condition (E.2) is weaker than (E.1) in that it restricts only the growth of the fourth moments but stronger than (E.1) in that it restricts the growth of $\max_{1\leq j\leq p}|x_{ij}|$. \qed
\end{remark}

We shall also consider regression applications where one of the following conditions is satisfied {\em uniformly in} $1 \leq i \leq n$ and $1 \leq j \leq p$:
\begin{itemize}
\item[(E.3)]  \ $x_{ij}= z_{ij} \eps_{ij}$, where $z_{ij}$ are non-stochastic with $|z_{ij}| \leq B_n$, $\En[z_{ij}^2]=1$, and $\Ep[\eps_{ij}]=0$, $\Ep[\eps^2_{ij}] \geq c_{1}$, and $\Ep[ \exp (|\eps_{ij}|/C_{1}) ] \leq  2$; or
\item[(E.4)]  \ $x_{ij}= z_{ij} \eps_{ij}$, where $z_{ij}$ are non-stochastic with $|z_{ij}| \leq B_n$, $\En[z_{ij}^2]=1$, and  $\Ep[\eps_{ij}]=0$, $\Ep[\eps^2_{ij}] \geq c_{1}$, and
 $\Ep[ \max_{1 \leq j \leq p} \eps_{ij}^4] \leq C_1$.
\end{itemize}
\begin{remark} Conditions (E.3) and (E.4) cover examples that arise in high-dimensional regression, for example, \cite{CandesTao2007}, which we shall revisit later in the paper.
Typically, $\eps_{ij}$'s are independent of $j$ (i.e., $\eps_{ij} = \eps_{i})$ and hence $\Ep[ \max_{1 \leq j \leq p} \eps_{ij}^4] \leq C_1$ in condition (E.4) reduces to $\Ep[ \eps_{i}^{4} ] \leq C_1$.
Interestingly, these examples are also connected to  spin glasses, see, for example, \cite{Talagrand2003} and \cite{Panchenko2013} ($z_{ij}$ can be interpreted
as generalized products of ``spins" and $\eps_{i}$ as their random ``interactions"). Note that conditions (E.3) and (E.4) are special cases of conditions (E.1) and (E.2) but we state (E.3) and (E.4) explicitly because these conditions are useful in applications.
\qed \\
\end{remark}

\begin{corollary}[\textbf{Gaussian Approximation in Leading Examples}]\label{cor: central limit theorem}
Suppose that there exist constants $c_2>0$ and $C_2>0$ such that one of the following conditions is satisfied:
(i) (E.1) or (E.3) holds and $B_n^2 (\log (pn))^7/n\leq C_2 n^{-c_2}$ or
(ii) (E.2) or (E.4) holds and $B_n^4 (\log (pn))^7/n\leq C_2 n^{-c_2}$.
Then there exist constants $c > 0$ and $C>0$ depending only on $c_{1}, C_1, c_{2}$, and $C_{2}$ such that $$\rho  \leq Cn^{-c}.$$
\end{corollary}


\begin{remark}
This corollary follows relatively directly from Theorem \ref{cor: Gaussian to nonGaussian KS 2} with help of Lemma \ref{lem: bound on u}.
Moreover,  from Lemma \ref{lem: bound on u},  it is routine to find other conditions that lead to the conclusion of Corollary \ref{cor: central limit theorem}.  \qed
\end{remark}


\begin{remark}[The benefits from the overall proof strategy]
\label{comment: warmup}
We note in Section \ref{sec: elementary GAR} of the SM \cite{CCK13}, that it is possible to derive the following
result by a much simpler proof:

\begin{lemma}[A Simple GAR] Suppose that there are some constants $c_{1} > 0$ and  $C_{1} > 0$ such that $c_{1}  \leq \barEp[x_{ij}^2] \leq C_{1}$ for all $1\leq j \leq p$. Then there exists a constant $C>0$ depending only on $c_{1}$ and $C_{1}$ such that
\begin{equation}\label{eq: warmup}
\sup_{t\in\RR}\left|\Pr ( T_{0} \leq t ) - \Pr (Z_{0} \leq t) \right|   \leq C (n^{-1} (\log (pn))^7)^{1/8} (\barEp [S^3_i])^{1/4},
\end{equation}
 where $S_i :=  \max_{1\leq j\leq p} ( |x_{ij}| +|y_{ij}|)$.
 \end{lemma} This simple (though apparently new, at this level of generality) result follows from the classical Lindeberg's argument previously given in Chatterjee \cite{Chatterjee2005a} (in the special context of a spin-glass setting like (E.4) with $\epsilon_{ij} = \epsilon_i$)  in combination with Lemma \ref{lem: anticoncentration} and standard kernel smoothing of indicator functions.   In the SM \cite{CCK13}, we provide the proof using Slepian-Stein methods, which a reader wishing to see a simple exposition (before reading a much more involved proof of the main results) may find helpful.  The bound here is only useful in some limited cases, for example, in (E.3) or (E.4)  when $B_n^6 (\log (pn))^7/n \to 0$. When $B_n^6 (\log (pn))^7/n \to \infty$, the simple methods fail, requiring a more delicate argument.  Note
 that in applications $B_n$ typically grows at a fractional power of $n$, see, for example, \cite{ChernozhukovChetverikovKato2012b} and \cite{Chetverikov2011}, and so the limitation is rather major, and was the principal motivation for our whole paper. \qed
\end{remark}

\section{Gaussian Multiplier Bootstrap}
\label{sec: multiplier bootstrap}

\subsection{A Gaussian-to-Gaussian Comparison Lemma}\label{sec: Gaussian vs Gaussian}
The proofs of the main results in this section rely on the following lemma. Let $V$ and $Y$ be centered Gaussian random vectors in $\RR^{p}$
with covariance matrices $\Sigma^{V}$ and
$\Sigma^{Y}$, respectively. The following lemma compares the distribution functions
of $\max_{1 \leq j \leq p}V_{j} \text{and} \max_{1 \leq j \leq p} Y_{j}$
in terms of $p$ and
\begin{equation*}
\Delta_0:=\max_{1\leq j,k\leq p}\left|\Sigma^V_{jk}-\Sigma^Y_{jk}\right|.
\end{equation*}

\begin{lemma}[Comparison of Distributions of Gaussian Maxima]\label{lemma: distances Gaussian to Gaussian}
Suppose that there are some constants $0 < c_{1} < C_{1}$ such that  $c_{1} \leq \Sigma^Y_{jj}\leq C_{1}$ for  all $1\leq j \leq p$.
Then there exists a constant $C>0$ depending only on $c_{1}$ and $C_{1}$ such that
\begin{equation*}
\sup_{t\in\RR}\left|\Pr\left(\max_{1 \leq j \leq p} V_{j}\leq t\right)-\Pr\left(\max_{1 \leq j \leq p} Y_{j}\leq t\right)\right| \leq  C\Delta_{0}^{1/3}(1 \vee \log (p/\Delta_0))^{2/3}.
\end{equation*}
\end{lemma}
\begin{remark} The result  is derived in \cite{ChernozhukovChetverikovKato2012c}, and  extends that of \cite{Chatterjee2005b} who gave an explicit error in Sudakov-Fernique comparison of expectations of maxima of Gaussian random vectors.  \qed
\end{remark}

\subsection{Results on Gaussian Multiplier Bootstrap}

Suppose that we have a dataset $(x_{i})_{i=1}^{n}$ consisting of $n$ independent centered random  vectors  $x_{i}$ in $\RR^{p}$.
In this section, we are interested in approximating quantiles of
\begin{equation}
T_0= \max_{1\leq j\leq p}\frac{1}{\sqrt{n}}\sum_{i=1}^{n}x_{ij} \label{average}
\end{equation}
using the multiplier bootstrap method. Specifically, let $(e_{i})_{i=1}^{n}$ be a  sequence of i.i.d. $N(0,1)$ variables independent of $(x_{i})_{i=1}^{n}$, and let
\begin{equation}
W_0= \max_{1\leq j\leq p}\frac{1}{\sqrt{n}}\sum_{i=1}^{n}x_{ij}e_{i}. \label{average-multiplier}
\end{equation}
Then we define the multiplier bootstrap estimator of the $\alpha$-quantile of $T_0$ as the conditional $\alpha$-quantile of $W_0$ given $(x_{i})_{i=1}^{n}$, that is,
\begin{equation*}
c_{W_0}(\alpha):=\inf\{ t \in \RR: \Pr_e(W_0\leq t) \geq \alpha \},
\end{equation*}
where $\Pr_e$ is the probability measure induced by the multiplier variables $(e_{i})_{i=1}^n$ holding $(x_{i})_{i=1}^{n}$ fixed (i.e., $\Pr_e(W_0\leq t) = \Pr (W_0\leq t \mid (x_{i})_{i=1}^{n})$).
The multiplier bootstrap theorem below provides a non-asymptotic  bound on the bootstrap estimation error.

Before presenting the theorem, we first give a simple useful lemma that is helpful in the proof of the theorem and in power analysis in applications. Define
\begin{equation*}
c_{Z_0}(\alpha):=\inf\{t\in\RR:\Pr(Z_0\leq t)\geq \alpha\},
\end{equation*}
where $Z_0=\max_{1\leq j\leq p}\sum_{i=1}^ny_{ij}/\sqrt{n}$ and $(y_{i})_{i=1}^n$ is a sequence of independent $N(0,\Ep[x_{i}x_{i}^{\prime}])$ vectors. Recall that $\Delta=\max_{1 \leq j, k \leq p} \left | \En[x_{ij}x_{ik}]-\barEp[x_{ij} x_{ik}] \right |$.

\begin{lemma}[Comparison of Quantiles, I]\label{lem: quantile conditional to unconditional}
Suppose that there are some constants $0 < c_{1} < C_{1}$ such that $c_{1} \leq \barEp[x_{ij}^2]\leq C_{1}$ for all $1\leq j\leq p$.
 Then for every  $\alpha \in (0,1)$,
\begin{align*}
&\Pr\big(c_{W_0}(\alpha)\leq c_{Z_0}(\alpha+\pi(\vartheta))\big) \geq 1- \Pr(\Delta> \vartheta), \\
&\Pr\big(c_{Z_0}(\alpha)\leq c_{W_0}(\alpha+\pi(\vartheta))\big) \geq 1- \Pr(\Delta> \vartheta),
\end{align*}
where, for $C_{2} > 0$ denoting a constant depending only on $c_{1}$ and $C_{1}$,
 $$\pi(\vartheta):=C_2\vartheta^{1/3}(1 \vee \log(p/\vartheta))^{2/3}.$$
\end{lemma}
Recall that $\rho:=\sup_{t\in\RR}\left|\Pr(T_0\leq t)-\Pr(Z_0\leq t)\right|.$
We are now in position to state the first main theorem of this section.
\begin{theorem}[\textbf{Main Result 2: Validity of Multiplier Bootstrap for High-Dimensional Means}]\label{thm: multiplier bootrstrap I}
Suppose that for some constants $0 < c_{1} < C_{1}$,  we have $c_{1} \leq \barEp[x_{ij}^2]\leq C_{1}$ for all $1\leq j\leq p$.
Then for every  $\vartheta>0$,
\begin{equation*}
\rho_{\ominus}:=\sup_{\alpha\in(0,1)}\Pr(\{T_0\leq c_{W_0}(\alpha)\}\ominus\{T_0\leq c_{Z_0}(\alpha)\})\leq 2(\rho+\pi(\vartheta)+\Pr(\Delta>\vartheta)),
\end{equation*}
where $\pi(\cdot)$ is defined in Lemma \ref{lem: quantile conditional to unconditional}. In addition,
\[
\sup_{\alpha\in(0,1)} \left | \Pr(T_0\leq c_{W_0}(\alpha))-\alpha \right|\leq \rho_{\ominus}+\rho.
\]
\end{theorem}

Theorem \ref{thm: multiplier bootrstrap I} provides a useful result for the case where the statistics
are maxima of exact averages. There are many applications, however, where the relevant statistics
arise as maxima of approximate averages.  The following result shows that
the theorem continues to apply if the approximation error of the relevant statistic
by a maximum of an exact average can be suitably controlled.  Specifically, suppose
that a statistic of interest, say $T=T(x_{1}\dots,x_{n})$ which may not be of the form (\ref{average}), can be approximated by $T_0$ of the form (\ref{average}), and that
the multiplier bootstrap is performed on a statistic  $W=W(x_{1},\dots,x_{n},e_{1},\dots,e_{n})$, which may be different from (\ref{average-multiplier}) but still can be approximated by $W_0$ of the form (\ref{average-multiplier}).

We require the approximation to hold in the following sense: there exist $\zeta_1 \geq 0$ and $\zeta_2 \geq 0$, depending on $n$ (and typically $\zeta_{1} \to 0, \zeta_{2} \to 0$ as $n \to \infty$),
such that
\begin{align}
&\Pr(|T-T_0|>\zeta_1)<\zeta_2, \label{eq: statistic approximation} \\
&\Pr(\Pr_e(|W-W_0|>\zeta_1)> \zeta_2)< \zeta_2. \label{eq: conditional quantiles}
\end{align}
We use the $\alpha$-quantile of $W=W(x_{1},\dots,x_{n},e_{1},\dots,e_{n})$, computed conditional on $(x_{i})_{i=1}^{n}$:
\begin{equation*}
c_{W}(\alpha):=\inf\{t\in\RR:\Pr_e(W\leq t)\geq \alpha\},
\end{equation*}
as an estimate of the $\alpha$-quantile of $T$.

\begin{lemma}[Comparison of Quantiles, II]\label{lem: quantile approximated to exact}
Suppose that condition (\ref{eq: conditional quantiles}) is satisfied. Then for every $\alpha\in(0,1)$,
\begin{align*}
&\Pr(c_{W}(\alpha)\leq c_{W_0}(\alpha+\zeta_2)+\zeta_1) \geq 1-\zeta_2,\\
&\Pr(c_{W_0}(\alpha)\leq c_{W}(\alpha+\zeta_2)+\zeta_1) \geq 1-\zeta_2.
\end{align*}
\end{lemma}

The next result provides a bound on the bootstrap estimation error.

\begin{theorem}[\textbf{Main Result 3: Validity of Multiplier Bootstrap for Approximate High-Dimensional Means}]
\label{thm: multiplier bootrstrap II}
Suppose that, for some constants $0 < c_{1} < C_{1}$,  we have $c_{1} \leq \barEp[x_{ij}^2]\leq C_{1}$ for all $1\leq j\leq p$.
Moreover, suppose that  (\ref{eq: statistic approximation}) and (\ref{eq: conditional quantiles}) hold.
Then for every $\vartheta>0$,
\begin{align*}
\rho_\ominus&:=\sup_{\alpha\in(0,1)}\Pr(\{T\leq c_{W}(\alpha)\}\ominus\{T_0\leq c_{Z_0}(\alpha)\})\\
&\leq  2(\rho +  \pi(\vartheta) +  \Pr(\Delta>\vartheta)) + C_{3}\zeta_1\sqrt{1 \vee \log(p/\zeta_1)} +5\zeta_2,
\end{align*}
where $\pi(\cdot)$ is defined in Lemma \ref{lem: quantile conditional to unconditional}, and $C_{3}>0$ depends only on $c_{1}$ and $C_{1}$. In addition,
$
\sup_{\alpha\in(0,1)}\left|\Pr(T\leq c_W(\alpha))-\alpha\right|\leq \rho_{\ominus}+\rho.$
\end{theorem}


\begin{remark}[On Empirical and other bootstraps]
In this paper, we focus on the Gaussian multiplier bootstrap (which is a form of wild bootstrap).
This is because  other exchangeable bootstrap methods
are asymptotically equivalent to this bootstrap.  For example, consider the empirical (or Efron's) bootstrap which approximates the distribution of $T_{0}$ by the conditional distribution of $T_{0}^{*} = \max_{1 \leq j \leq p} \sum_{i=1}^{n}(x_{ij}^{*}-\En[x_{ij}])/\sqrt{n}$ where $x_{1},\dots,x_{n}^{*}$ are i.i.d. draws from the empirical distribution of $x_{1},\dots,x_{n}$.
We show in Section \ref{sec: nonparametric bootstrap} of the SM \cite{CCK13},  that the empirical bootstrap is asymptotically equivalent
to the Gaussian multiplier bootstrap, by virtue of Theorem \ref{cor: Gaussian to nonGaussian KS 2} (applied conditionally on the data).  The validity of the empirical bootstrap then follows from the validity of the Gaussian multiplier method.    The result is demonstrated
under a simplified condition.  A detailed analysis of more sophisticated conditions, and the validity of more general exchangeably weighted bootstraps (see \cite{PW93}) in the current setting, will be pursued in future work. \qed
\end{remark}

\subsection{Examples  Revisited}

Here we revisit the examples in Section \ref{sub: examples of applications GAR} and see how the multiplier bootstrap works for these leading examples.
Let, as before, $c_{2} > 0$ and $C_{2} > 0$ be some constants, and let $B_{n} \geq 1$ be a sequence of constants.
Recall conditions (E.1)-(E.4) in Section \ref{sub: examples of applications GAR}.
The next corollary shows that the multiplier bootstrap is valid with a polynomial rate of accuracy for the significance level under weak conditions.

\begin{corollary}[\textbf{Multiplier Bootstrap in Leading Examples}]
\label{cor: multiplier bootstrap examples}
Suppose that conditions (\ref{eq: statistic approximation}) and (\ref{eq: conditional quantiles}) hold with $\zeta_1 \sqrt{\log p} + \zeta_2 \leq C_{2} n^{-c_{2}}$.
Moreover, suppose that  one of the following conditions is satisfied:
(i) (E.1) or (E.3) holds and $B_n^2 (\log (pn))^7/n\leq C_2 n^{-c_2}$ or
(ii) (E.2) or (E.4) holds and $B_n^4 (\log (pn))^7/n\leq C_2 n^{-c_2}$.
Then there exist constants $c > 0$ and $C > 0$ depending only on $c_{1},C_1,c_{2}$, and $C_{2}$ such that
\begin{equation*}
\rho_{\ominus}=\sup_{\alpha\in(0,1)}\Pr(\{T\leq c_{W}(\alpha)\}\ominus\{T_0\leq c_{Z_0}(\alpha)\}) \leq Cn^{-c}.
\end{equation*}
In addition, $\sup_{\alpha \in (0,1)} | \Pr(T\leq c_W(\alpha)) - \alpha |\leq \rho_{\ominus}+\rho\leq Cn^{-c}$.
\end{corollary}



\section{Application: Dantzig Selector in the Non-Gaussian Model}
\label{sec: Dantzig}

The purpose of this section is to demonstrate the case with which the GAR and the multiplier bootstrap theorem given in Corollaries \ref{cor: central limit theorem} and \ref{cor: multiplier bootstrap examples} can be applied in important problems,
dealing with a high-dimensional inference and estimation. We consider
the Dantzig selector previously studied in the path-breaking works of \cite{CandesTao2007}, \cite{BickelRitovTsybakov2009}, \cite{YeZhang2010} in the  Gaussian setting
and of \cite{Koltchinskii2009} in a sub-exponential setting. Here we consider the non-Gaussian case, where the errors have only four bounded moments, and derive the performance bounds that are approximately as sharp as in the Gaussian model. We consider both homoscedastic and heteroscedastic models.

\subsection{Homoscedastic case}
Let $(z_{i},y_{i})_{i=1}^{n}$ be a sample of independent observations where $z_{i} \in \RR^{p}$ is a non-stochastic vector of regressors. We consider the model
\begin{equation*}
y_{i}=z_{i}^{\prime}\beta+\eps_{i}, \ \ \Ep[\eps_{i}]=0, \ i=1,\dots,n, \  \En[z_{ij}^2]=1,  \ j=1,\dots,p,
\end{equation*}
where $y_{i}$ is a random scalar dependent variable, and the regressors are normalized in such a way that $\En[z_{ij}^2]=1$.
 Here we consider the
homoscedastic case:
\begin{equation*}
\Ep[\eps_{i}^2] = \sigma^2, \ i=1,\dots,n,
\end{equation*}
where $\sigma^{2}$ is assumed to be known (for simplicity).  We allow $p$ to be substantially larger than $n$.  It is well known that a condition that gives a good performance for the Dantzig selector is that $\beta$ is sparse, namely
$\|\beta\|_0 \leq s \ll n$ (although this assumption will not be invoked below explicitly).

The aim is to estimate the vector $\beta$ in some semi-norms of interest:
$\|\cdot\|_I$, where the label $I$ is the name of a norm of interest. For example, given an estimator $\hat \beta$ the prediction semi-norm for $\delta = \hat \beta - \beta$ is
\begin{equation*}
\| \delta \|_{\pr} := \sqrt{\En[ (z_{i}'\delta)^2 ]},
\end{equation*}
or the $j$th component seminorm for $\delta$ is
$\| \delta \|_{\text{jc}} := |\delta_{j}|,$ and so on.

The Dantzig selector is the estimator defined by
\begin{equation}
\label{eq: dantzig estimator}
\hat\beta \in \arg\min_{b\in\RR^p}\Vert b\|_{\ell_1} \ \text{subject to} \ \sqrt{ n} \max_{1\leq j\leq p}|\En[z_{ij}(y_{i}-z_{i}^\prime b)]|\leq{\lambda},
\end{equation}
where $\| \beta \|_{\ell_{1}} = \sum_{j=1}^{p} | \beta_{j} |$ is the $\ell_{1}$-norm.
An ideal choice of the penalty level $\lambda$  is meant to ensure that
\begin{equation*}
T_0 := \sqrt{n} \max_{1\leq j\leq p}|\En[z_{ij}\eps_{i}]|\leq{\lambda}
\end{equation*}
with a prescribed confidence level $1-\alpha$ (where $\alpha$ is a number close
to zero.) Hence we would like to set penalty level ${\lambda}$ equal to
\begin{equation*}
c_{T_0} (1-\alpha) := \text{$(1-\alpha)$-quantile of $T_{0}$},
\end{equation*}
(note that $z_{i}$ are treated as fixed).  Indeed, this penalty
would take into account the correlation amongst the regressors, thereby adapting the performance
of the estimator to the design condition.

We can approximate this quantity using the
Gaussian approximations derived in Section 2.  Specifically,  let
\begin{equation*}
Z_0: = \sigma \sqrt{n} \max_{1\leq j\leq p}|\En[z_{ij} e_{i}]|,
\end{equation*}
where $e_{i}$ are i.i.d. $N(0,1)$ random variables independent of the data. We then estimate $c_{T_0} (1-\alpha)$ by
\begin{equation*}
c_{Z_0}(1-\alpha):= \text{ $(1-\alpha)$-quantile of $Z_{0}$}.
\end{equation*}
Note that we can calculate $c_{Z_0}(1-\alpha)$ numerically with any specified precision by the simulation.
(In a Gaussian model, design-adaptive penalty level $c_{Z_0}(1-\alpha)$ was proposed in \cite{BelloniChernozhukov2011}, but its extension to non-Gaussian cases was not available up to now).

An alternative choice of the penalty level is given by
\begin{equation*}
c_{0}(1-\alpha) :=  \sigma \Phi^{-1}(1- \alpha/(2p)),
\end{equation*}
which is the canonical choice; see \cite{CandesTao2007} and \cite{BickelRitovTsybakov2009}.
Note that canonical choice $c_{0}(1-\alpha)$ disregards the correlation amongst the regressors, and is
therefore more conservative than $c_{Z_0}(1-\alpha)$.  Indeed, by the union bound, we see that
\begin{equation*}
c_{Z_0}(1-\alpha) \leq c_{0}(1-\alpha).
\end{equation*}

Our first result below shows that the \textit{either} of the two penalty choices,  $\lambda= c_{Z_0}(1-\alpha)$ or $\lambda = c_0(1-\alpha)$, are
approximately valid under non-Gaussian noise--under the mild moment assumption $\Ep[\eps_{i}^4] \leq \text{const}.$ replacing the canonical Gaussian noise assumption. To derive this result we apply our GAR to $T_0$ to establish that the difference between distribution functions of $T_0$ and $Z_0$ approaches zero at polynomial speed.
Indeed $T_0$ can be represented as a maximum of averages, $
T_0 = \max_{1\leq k\leq 2p} n^{-1/2} \sum_{i=1}^ n \tilde z_{ik}\eps_{i}$, for
$\tilde z_{i} = (z_{i}', -z_{i}')'$ where $z_i^\prime$ denotes the transpose of $z_i$.

To derive the bound on estimation error $\| \delta \|_I$ in a seminorm of interest, we employ the following identifiability factor:
\begin{equation*}
\kappa_I(\beta):=\inf_{\delta \in \RR^p} \left \{ \max_{1\leq j\leq p}\frac{|\En[z_{ij}(z_{i}^\prime\delta)]|}{\| \delta \|_I }:
\delta\in \mathcal{R}(\beta),  \| \delta\|_I \neq 0 \right \},
\end{equation*}
where $\mathcal{R}(\beta):= \{ \delta \in \RR^p: \Vert\beta+\delta\|_{\ell_1}\leq\Vert\beta\|_{\ell_1}\}$ is the restricted
set; $\kappa_I(\beta)$ is defined as $\infty$ if $\mathcal{R}(\beta) = \{0\}$ (this happens if $\beta =0$).    The factors summarize the impact of sparsity of true parameter value $\beta$ and the design on the identifiability of $\beta$ with respect to the norm $\| \cdot \|_I$.

\begin{remark}[A comment on the identifiability factor $\kappa_I(\beta)$]
The identifiability factors $\kappa_I(\beta)$ depend on the true parameter value $\beta$. These factors represent a modest generalization of the cone invertibility factors and sensitivity characteristics defined in \cite{YeZhang2010} and \cite{GautierTsybakov2011}, which are known to be quite general.
The difference is the use of a norm of interest $\|\cdot\|_I$ instead of the $\ell_q$ norms and the use of smaller (non-conic) restricted set $\mathcal{R}(\beta)$ in the definition.
It is useful to note for later comparisons that in the case of prediction norm $\|\cdot \|_{I} = \| \cdot \|_{\pr}$ and under the exact sparsity assumption $\|\beta\|_0 \leq s$, we have
\begin{equation}
\kappa_{\pr}(\beta) \geq 2^{-1} s^{-1/2} \kappa(s,1), \label{relate to RE}
\end{equation}
where $\kappa(s,1)$ is the restricted eigenvalue defined in \cite{BickelRitovTsybakov2009}.  \qed
 \end{remark}

Next we state bounds on the estimation error for the Dantzig selector $\hat \beta^{(0)}$ with
canonical penalty level $\lambda = \lambda^{(0)} := c_0(1-\alpha)$ and the Dantzig selector $\hat \beta^{(1)}$
with design-adaptive penalty level $\lambda= \lambda^{(1)} := c_{Z_0}(1-\alpha).$

\begin{theorem}[Performance of Dantzig Selector in Non-Gaussian Model]\label{thm: dantzig estimator}
Suppose that there are some constants $c_{1} > 0, C_{1} > 0$ and $\sigma^2 > 0$, and a sequence $B_{n} \geq 1$ of constants such that for all $1 \leq i \leq n$ and $1 \leq j \leq p$:
(i) $|z_{ij}|\leq B_n$; (ii) $\En[z_{ij}^2]= 1$; (iii) $ \Ep[\eps_{i}^2] =\sigma^2$; (iv) $\Ep[\eps_{i}^{4}]\leq C_{1}$; and (v) $B_n^4(\log (pn))^7/n\leq C_{1} n^{-c_{1}}$.
Then there exist constants $c > 0$ and $C > 0$ depending only on $c_{1},C_{1}$ and $\sigma^{2}$ such that, with probability at least $ 1- \alpha - C n^{-c}$, for either $k=0$ or $1$,
\begin{equation*}
\Vert\hat\beta^{(k)} -\beta \Vert_I \leq \frac{2 \lambda^{(k)}}{\sqrt{n} \kappa_I(\beta)}.
\end{equation*}
\end{theorem}

The most important feature of this result is that it provides Gaussian-like conclusions (as
explained below) in a model with non-Gaussian noise, having only four bounded moments.  However, the probabilistic guarantee
is not $1-\alpha$ as, for example, in \cite{BickelRitovTsybakov2009}, but rather $1- \alpha - C n^{-c}$, which reflects the cost of non-Gaussianity
(along with more stringent side conditions). In what follows
we discuss details of this result.  Note that the bound above holds for any semi-norm of interest $\| \cdot \|_I$.

\begin{remark}[Improved Performance from Design-Adaptive Penalty Level] The use of the design-adaptive penalty level
implies a better performance guarantee for $\hat \beta^{(1)}$ over $\hat \beta^{(0)}$. Indeed,
we have
\begin{equation*}
\frac{2 c_{Z_0}(1-\alpha)}{\sqrt{n} \kappa_I(\beta)} \leq \frac{2 c_0(1-\alpha)}{\sqrt{n} \kappa_I(\beta)}.
\end{equation*}
For example, in some designs, we can have $\sqrt{n} \max_{1 \leq j \leq p}| \En[z_{ij} e_{i}]| = O_{\Pr}(1)$,
so that $c_{Z_0}(1-\alpha) = O(1)$, whereas $c_0(1-\alpha) \propto \sqrt{\log p}$.  Thus,
the performance guarantee provided by $\hat \beta^{(1)}$ can be much better than
that of $\hat \beta^{(0)}$.  \qed
\end{remark}

\begin{remark}[Relation to the previous results under Gaussianity] To compare to the previous results obtained for
the Gaussian settings, let us focus on the prediction norm and on estimator $\hat \beta^{(1)}$ with penalty level $\lambda= c_{Z_0}(1-\alpha)$.
Suppose that the true value $\beta$ is sparse, namely $\|\beta\|_0 \leq s$. In this case, with probability at least $1- \alpha - C n^{-c}$,
\begin{equation}\label{eq: boundD}
\Vert\hat\beta^{(1)} -\beta \Vert_{\pr} \leq \frac{2 c_{Z_0}(1-\alpha)}{\sqrt{n} \kappa_{\pr}(\beta)} \leq \frac{4 \sqrt{s}  c_0(1-\alpha)}{\sqrt{n} \kappa(s,1)} \leq \frac{4 \sqrt{s}  \sqrt{2 \log ( \alpha/(2p))}}{\sqrt{n} \kappa(s,1)},
 \end{equation}
where the last bound is the same as in \cite{BickelRitovTsybakov2009}, Theorem 7.1, obtained for the Gaussian case.
We recover the same (or tighter) upper bound without making the Gaussianity assumption on the errors. However,  the probabilistic guarantee is
not $1-\alpha$ as in  \cite{BickelRitovTsybakov2009}, but rather $1- \alpha - Cn^{-c}$, which together with side conditions is the cost of non-Gaussianity.  \qed
\end{remark}

\begin{remark}[Other refinements]  Unrelated to the main theme of this paper, we can see from (\ref{eq: boundD}) that
 there is some tightening of the performance bound due to the use of the identifiability factor $\kappa_{\pr}(\beta)$ in place of the restricted eigenvalue $\kappa(s,1)$; for example, if $p=2$ and $s=1$ and the two regressors are identical, then $\kappa_{\pr}(\beta)>0$, whereas $\kappa(1,1) = 0$. There is also some tightening due to the use of $c_{Z_0}(1-\alpha)$ instead of $c_{0}(1-\alpha)$ as penalty level, as mentioned above. \qed
\end{remark}

\subsection{Heteroscedastic case} We consider the same model as above, except
now the assumption on the error becomes
\begin{equation*}
\sigma_{i}^2 := \Ep[\eps_{i}^2]  \leq  \sigma^2, \ \ i=1,\dots,n,
\end{equation*}
that is, $\sigma^{2}$ is the upper bound on the conditional variance, and we assume that this bound is known (for simplicity).  As before, ideally we would like to set penalty level ${\lambda}$ equal to
\begin{equation*}
c_{T_0} (1-\alpha) := \text{$(1-\alpha)$-quantile of $T_{0}$},
\end{equation*}
(where $T_0$ is defined above, and we note that $z_{i}$ are treated as fixed).
The GAR applies as before, namely the difference of the distribution functions of $T_0$ and its Gaussian analogue $Z_0$
converges to zero.  In this case, the Gaussian analogue can be represented as
\begin{equation*}
Z_0 :=  \sqrt{n} \max_{1\leq j\leq p}|\En[z_{ij} \sigma_{i} e_{i}]|.
\end{equation*}
Unlike in the homoscedastic case, the covariance structure is no longer known, since
$\sigma_{i}$ are unknown and we can no longer calculate the quantiles of $Z_0$.  However,
we can estimate them using the following multiplier bootstrap procedure.

First,  we estimate the residuals $\hat \eps_{i} = y_{i} - z_{i}'\hat \beta^{(0)}$
obtained from a preliminary Dantzig selector $\hat \beta^{(0)}$ with the conservative penalty level ${\lambda} = \lambda^{(0)} := c_0(1-1/n) :=  \sigma \Phi^{-1}(1- 1/(2pn))$,
where $\sigma^{2}$ is  the upper bound on the error variance assumed to be known.
Let $(e_{i})_{i=1}^{n}$ be a sequence of i.i.d. standard Gaussian  random variables, and let
\begin{equation*}
W: = \sqrt{n} \max_{1\leq j\leq p}|\En[z_{ij}\hat\eps_{i}e_{i}]|.
\end{equation*}
Then we estimate $c_{Z_0} (1-\alpha)$ by
\begin{equation*}
c_{W}(1-\alpha):=  \text{$(1-\alpha)$-quantile of $W$},
\end{equation*}
defined conditional on data $(z_{i},y_{i})_{i=1}^{n}$. Note that $c_W(1-\alpha)$ can be calculated numerically with any specified precision by the simulation. Then we apply program (\ref{eq: dantzig estimator}) with ${\lambda} = \lambda^{(1)}= c_W(1-\alpha)$ to obtain $\hat\beta^{(1)}$.

\begin{theorem}[Performance of Dantzig in Non-Gaussian Model with Bootstrap Penalty Level]\label{thm: dantzig estimator 2}
Suppose that there are some constants $c_{1} > 0, C_{1} > 0,\underline{\sigma}^2 > 0$ and $\sigma^2 > 0$, and a sequence $B_{n} \geq 1$ of constants such that for all $1 \leq i \leq n$ and $1 \leq j \leq p$:
(i) $|z_{ij}|\leq B_n$;
(ii) $\En[z_{ij}^2]= 1$;
(iii) $\underline{\sigma}^2 \leq \Ep[\eps_{i}^2] \leq \sigma^2$;
(iv) $\Ep[\eps_{i}^{4}]\leq C_{1}$;
(v) $B_n^4(\log (pn))^7/n\leq C_{1} n^{-c_{1}}$;
and (vi)  $ (\log p) B_n c_0(1-1/n)/(\sqrt{n} \kappa_{\pr}(\beta) ) \leq C_{1} n^{-c_{1}}$.
Then there exist constants $c > 0$ and $C > 0$ depending only on  $c_{1}, C_{1},\underline{\sigma}^2$ and $\sigma^2$ such that,
with probability at least $ 1- \alpha -  \nu_{n}$ where $\nu_{n} = Cn^{-c}$, we have
 \begin{equation}
\label{eq: bound1}
\Vert\hat\beta^{(1)} -\beta \Vert_I \leq \frac{2 \lambda^{(1)}}{\sqrt{n} \kappa_{I}(\beta)}.
 \end{equation}
Moreover, with probability at least $1-  \nu_{n}$,
\begin{equation*}
\lambda^{(1)} = c_W(1-\alpha) \leq  c_{Z_0}(1-\alpha + \nu_n),
\end{equation*}
where $c_{Z_0}(1-a)  := \text{ $(1-a)$-quantile of $Z_{0}$}$; where $c_{Z_0}(1-a) \leq c_0(1-a)$.
\end{theorem}

\begin{remark}[A Portmanteu Signicance Test] The result above contains a practical test of joint significance of all regressors, that is, a test of the hypothesis that $\beta_0 =0$, with the exact asymptotic size $\alpha$.

\begin{corollary}  Under conditions of the either of preceding two theorems, the test, that rejects the null hypothesis $\beta_0= 0$ if $\hat \beta^{(1)} \neq 0$, has size equal to $\alpha + Cn^{-c}$.
\end{corollary}
To see this note that under the null hypothesis of $\beta_0 =0$, $\beta_0$ satisfies the constraint  in (\ref{eq: dantzig estimator}) with probability  $(1-\alpha - Cn^{-c})$, by construction of $\lambda$; hence $\|\hat \beta^{(1)}\| \leq \| \beta_0\| = 0$ with exactly this probability.  Appendix \ref{sub: AST} of the SM \cite{CCK13} generalizes this to a more general test, which tests $\beta_0 =0$ in the regression model $y_i= d_i'\gamma_0 + x_i'\beta_0 + \varepsilon_i$,  where $d_i$'s are a small set of variables, whose coefficients are not known and need to be estimated.  The test orthogonalizes each $x_{ij}$ with respect to $d_i$ by partialling out linearly the effect of $d_i$ on $x_{ij}$.  The result similar to that in the corollary continues to hold.  \qed
\end{remark}

\begin{remark}[Confidence Bands]
Following  Gautier and Tsybakov \cite{GautierTsybakov2011},  the bounds given in the preceding theorems can be used for Scheffe-type (simultaneous) inference on all components of $\beta_0$.
\begin{corollary} Under the conditions of either of the two preceding theorems,  a $(1-\alpha - Cn^{-c})$-confidence rectangle for $\beta_0$ is given by  the region $\times_{j=1}^p I_j$, where
$I_j = [ \hat \beta^{(1)}_{j} \pm 2 \lambda^{(1)}/(\sqrt{n} \kappa_{\text{jc}}(\beta)].
$\end{corollary}
We note that  $\kappa_{\text{jc}}(\beta) = 1$ if $\En[z_{ij}z_{ik}] =0$ for all $k \neq j$.  Therefore,
in the orthogonal model of Donoho and Johnstone,  where $\En[z_{ij}z_{ik}] =0$ for all pairs $j \neq k$, we have that
 $\kappa_{\text{jc}}(\beta) = 1$ for all $1 \leq j \leq p$,  so that $I_j = [ \hat \beta^{(1)}_{j} \pm 2 \lambda^{(1)}/\sqrt{n}]$,  which gives a practical  simultaneous $(1-\alpha - Cn^{-c})$ confidence rectangle for $\beta$.    In non-orthogonal designs,
 we can rely on \cite{GautierTsybakov2011}'s tractable linear programming   algorithms for computing lower bounds on $\kappa_I(\beta)$ for various norms $I$ of interest; see also \cite{JuditskyNemirovski2011}.
 \qed \end{remark}

\begin{remark}[Generalization of Dantzig Selector] There are many interesting applications where the results
given above apply. There are, for example, interesting works by \cite{AlquierHebiri2011} and \cite{FrickMarnitzMunk2012} that consider
related estimators that minimize a convex penalty subject to the multiresolution screening constraints.
In the context of the regression problem studied above, such estimators may be defined as:
\begin{equation*}
\hat\beta \in \arg\min_{b\in\RR^p} J(b) \text{ subject to } \sqrt{ n} \max_{1\leq j\leq p}|\En[z_{ij}(y_{i}-z_{i}^\prime b)]|\leq{\lambda},
\end{equation*}
where $J$ is a convex penalty, and the constraint is used for multiresolution screening.  For example, the Lasso estimator is nested by the above formulation by using $J(b) = \|b \|_{\pr}$, and the previous Dantzig selector by using  $J(b) = \|b \|_{\ell_1}$; the estimators can be interpreted as a point in confidence set for $\beta$, which lies closest to zero under $J$-discrepancy (see references cited above for both of these points). Our results on choosing $\lambda$ apply to this class of estimators, and the previous analysis also applies by redefining the identifiability factor $\kappa_I(\beta)$ relative to the new restricted set  $\mathcal{R}(\beta):= \{ \delta \in \RR^p:  J (\beta+\delta)\leq  J(\beta)\}$; where
$\kappa_I(\beta)$ is defined as $\infty$ if $\mathcal{R}(\beta) = \{0\}$. \qed
\end{remark}

\section{Application: Multiple Hypothesis Testing via the Stepdown Method}
\label{sub: MHT}

In this section, we study the problem of multiple hypothesis testing
in the framework of multiple means or, more generally, approximate means. The latter possibility allows us
 to cover the case of testing multiple coefficients in multiple regressions, which is often required in empirical studies; see, for example, \cite{Anderson08}. We combine a general stepdown
procedure described in \cite{RomanoWolf05} with the multiplier bootstrap developed in this paper.
In contrast with \cite{RomanoWolf05}, our results do not require
weak convergence arguments, and, thus, can be applied to models with an increasing
number of means. Notably, the number of means can be large in comparison with the sample size.

Let $\beta:=(\beta_1,\dots,\beta_p)^\prime\in\RR^p$ be a vector of parameters of interest. We are interested in simultaneously testing the set of null hypotheses $H_j:\beta_j\leq\beta_{0j}$ against the alternatives $H_j^\prime:\beta_j> \beta_{0j}$ for $j=1,\dots,p$ where $\beta_0:=(\beta_{01},\dots,\beta_{0p})^\prime\in\RR^p$. Suppose that the estimator $\hat{\beta}:=(\hat{\beta}_1,\dots,\hat{\beta}_p)^\prime\in\RR^p$ is available that has an approximately linear form:
\begin{equation}\label{linearize}
\sqrt{n}(\hat{\beta}-\beta) = \frac{1}{\sqrt{n}}\sum_{i=1}^nx_{i}+r_n,
\end{equation}
where $x_1,\dots,x_n$ are independent zero-mean random vectors in $\RR^p$, the influence functions, and  $r_n:=(r_{n1},\dots,r_{np})^\prime\in\RR^p$ are linearization errors that are small in the sense required by condition (M) below. Vectors $x_1,\dots,x_n$ need not be directly observable. Instead, some estimators $\hat{x}_1,\dots,\hat{x}_n$ of influence functions $x_1,\dots,x_n$ are available, which will be used in the bootstrap simulations.

We refer to this framework as testing multiple approximate means.  This framework covers the case of testing multiple means with $r_{n}=0$.  More generally, this framework also covers the case of multiple linear and non-linear m-regressions; see, for example, \cite{HeShao00} for explicit conditions giving rise to linearizaton (\ref{linearize}). The detailed exposition of how the case of multiple linear regressions fits into this framework can be found in \cite{ChernozhukovChetverikovKato2012a}. Note also that this framework implicitly covers the case of testing equalities ($H_j:\beta_j=\beta_{0j}$) because equalities can be rewritten as pairs of inequalities.

We are interested in a procedure with the strong control of the family-wise error rate.
In other words, we seek a procedure that would reject at least one true null hypothesis
with probability not greater than $\alpha+o(1)$ uniformly over a large class of data-generating processes and, in particular, uniformly over the
set of true null hypotheses. More formally, let $\Omega$ be a set
of all data generating processes, and $\omega$ be the true process.
Each null hypothesis $H_{j}$ is equivalent to $\omega\in\Omega_{j}$
for some subset $\Omega_{j}$ of $\Omega$.
Let $\mathcal{W}:=\{1,\dots,p\}$ and for $w\subset \mathcal{W}$
denote $\Omega^{w}:=(\cap_{j\in w}\Omega_{j})\cap(\cap_{j\notin w}\Omega_{j}^c)$ where $\Omega_{j}^c:=\Omega\backslash\Omega_{j}$. The strong control of the family-wise error rate means
\begin{equation}
\sup_{w \subset \mathcal{W}}\sup_{\omega\in\Omega^{w}}
\Pr_{\omega}\{\text{reject at least one hypothesis among \ensuremath{H_{j}}, \ensuremath{j\in w}}\}
\leq \alpha+o(1) \label{eq: strong control}
\end{equation}
where $\Pr_{\omega}$ denotes the probability distribution under the data-generating process $\omega$.
This setting is clearly of interest in many empirical studies.

For $j=1,\dots,p$, denote $t_{j}:=\sqrt{n}(\hat{\beta}_j-\beta_{0j})$. The stepdown procedure of \cite{RomanoWolf05} is described as follows. For
a subset $w\subset \mathcal{W}$, let $c_{1-\alpha,w}$ be some estimator of
the $(1-\alpha)$-quantile of $\max_{j\in w}t_{j}$. On the first step, let $w(1)=\mathcal{W}$. Reject
all hypotheses $H_{j}$ satisfying $t_{j}>c_{1-\alpha,w(1)}$.
If no null hypothesis is rejected, then stop. If some $H_{j}$ are
rejected, let $w(2)$ be the set of all null hypotheses that
were not rejected on the first step. On step $l\geq2$, let $w(l)\subset \mathcal{W}$
be the subset of null hypotheses that were not rejected up to step
$l$. Reject all hypotheses $H_{j}$, $j\in w(l)$, satisfying $t_{j}>c_{1-\alpha,w(l)}$.
If no null hypothesis is rejected, then stop. If some $H_{j}$ are
rejected, let $w(l+1)$ be the subset of all null hypotheses
among $j\in w(l)$ that were not rejected. Proceed in this way until
the algorithm stops.

Romano and Wolf \cite{RomanoWolf05} proved the following result. Suppose that $c_{1-\alpha,w}$
satisfy
\begin{align}
&c_{1-\alpha,w^{\prime}}\leq c_{1-\alpha,w^{\prime\prime}} \quad \text{whenever $w^{\prime}\subset w^{\prime\prime}$},\label{eq: critical value property 1}\\
&\sup_{w\subset \mathcal{W}}\sup_{\omega\in\Omega^{w}}\Pr_\omega \left ( \max_{j\in w}t_{j}>c_{1-\alpha,w} \right ) \leq \alpha+o(1),\label{eq: conditions MHT}
\end{align}
then inequality (\ref{eq: strong control}) holds if the stepdown procedure is used. Indeed, let $w$
be the set of true null hypotheses. Suppose that the procedure rejects
at least one of these hypotheses. Let $l$ be the step when the procedure
rejected a true null hypothesis for the first time, and let $H_{j_0}$ be
this hypothesis. Clearly, we have $w(l)\supset w$. So,
\begin{equation*}
\max_{j\in w}t_{j}\geq t_{j_0}>c_{1-\alpha,w(l)}\geq c_{1-\alpha,w}.
\end{equation*}
Combining this chain of inequalities with (\ref{eq: conditions MHT})
yields (\ref{eq: strong control}).

To obtain suitable $c_{1-\alpha,w}$ that satisfy inequalities (\ref{eq: critical value property 1}) and (\ref{eq: conditions MHT}) above, we can use the multiplier bootstrap method. Let $(e_{i})_{i=1}^n$ be an i.i.d. sequence of $N(0,1)$ random variables that are independent of the data. Let $c_{1-\alpha,w}$ be the conditional
$(1-\alpha)$-quantile of $\max_{j \in w}\sum_{i=1}^n\hat{x}_{ij}e_i/\sqrt{n}$
given $(\hat{x}_{i})_{i=1}^n$.

To prove that so defined critical values $c_{1-\alpha,w}$ satisfy inequalities (\ref{eq: critical value property 1}) and (\ref{eq: conditions MHT}), the following two quantities play a key role:
$$
\Delta_1:=\max_{1\leq j\leq p}|r_{nj}|\text{ and } \Delta_2:=\max_{1\leq j\leq p}\En[(\hat{x}_{ij}-x_{ij})^2].
$$
We will assume the following regularity condition,
\begin{itemize}
\item[(M)] There are positive constants $c_2 $ and $C_2$:  (i) $\Pr\left(\sqrt{\log p}\Delta_1>C_2n^{-c_2}\right)$ $<$ $C_2n^{-c_2}$ and (ii) $\Pr\left((\log (pn))^2\Delta_2>C_2n^{-c_2}\right)< C_2n^{-c_2}$. In addition, one of the following conditions is satisfied:
(iii) (E.1) or (E.3) holds and $B_n^2 (\log (pn))^7/n\leq C_2 n^{-c_2}$ or
(iv) (E.2) or (E.4) holds and $B_n^4 (\log (pn))^7/n\leq C_2 n^{-c_2}$.
\end{itemize}



\begin{theorem}[Strong Control of Family-Wise Error Rate]\label{thm: MHT}
Suppose that (M) is satisfied uniformly over a class of data-generating processes $\Omega$. Then the stepdown procedure with the multiplier bootstrap critical values $c_{1-\alpha,w}$ given above satisfy (\ref{eq: strong control}) for this $\Omega$ with $o(1)$ strengthened to $Cn^{-c}$ for some constants $c>0$ and $C>0$ depending only on $c_1,C_1,c_2$, and $C_2$.
\end{theorem}

\begin{remark}[The case of sample means] Let us consider the simple case of testing multiple means.
In this case,  $\beta_j = \Ep[z_{ij}]$ and  $\hat \beta_j = \En[z_{ij}]$, where $z_i =(z_{ij})_{j=1}^p$ are i.i.d. vectors, so that the influence functions are $x_{ij} = z_{ij} - \Ep[z_{ij}]$, and the remainder is zero, $r_{n}=0$. The influence functions $x_i$ are not directly observable,  though easily estimable by demeaning, $\hat x_{ij} = z_{ij} - \En[z_{ij}]$ for all $i$ and $j$.  It is instructive to see the implications of Theorem \ref{thm: MHT} in this simple setting. Condition (i) of assumption (M) holds trivially in this case. Condition (ii) of assumption (M) follows from Lemma \ref{lem: symmetrization inequality 1} under conditions (iii) or (iv) of assumption (M). Therefore, Theorem \ref{thm: MHT} applies provided that $\underline \sigma^2 \leq \Ep[\x_{ij}^2] \leq \bar \sigma^2$, $(\log p)^7 \leq C_2n^{1-c_2}$ for arbitrarily small $c_2$ and, for example, either (a) $\Ep [ \exp (|x_{ij}|/C_{1}) ] \leq 2$ (condition (E.1)) or (b) $\Ep[\max_{1\leq j\leq p} x_{ij}^4] \leq C_{1}$ (condition (E.2)). Hence, the theorem implies that the Gaussian multiplier bootstrap as described above leads to a testing procedure with the strong control of the family-wise error rate for the multiple hypothesis testing problem of which the {\em logarithm} of the number of hypotheses is nearly of order $n^{1/7}$. Note here that no assumption that limits the dependence between $x_{i1},\dots,x_{ip}$ or the distribution of $x_i$ is made. Previously, \cite{ArlotBlanchardRoquain2010b} proved strong control of the family-wise error rate for the Rademacher multiplier bootstrap with some adjustment factors assuming that $x_i$'s are Gaussian with unknown covariance structure.  \qed \end{remark}

\begin{remark}[Relation to Simultaneous Testing]The question on how large $p$ can be was studied in \cite{FanHallYao07} but from a conservative perspective. The motivation there is to know how fast $p$ can grow to maintain the size of the simultaneous test when we calculate critical values (conservatively) ignoring the dependency among $t$-statistics $t_{j}$ and assuming that $t_{j}$ were distributed as, say,  $N(0,1)$. This framework is conservative in
that correlation amongst statistics is dealt away by independence, namely by \v{S}id\'{a}k  procedures. In contrast, our approach takes into account the correlation  amongst statistics and hence is asymptotically exact, that is, asymptotically non-conservative. \qed
\end{remark}

\appendix

\section{Preliminaries}
\label{sec: auxiliary lemmas}

\subsection{A Useful Maximal Inequality}\label{sec: auxiliary results}
The following lemma, which is derived in \cite{ChernozhukovChetverikovKato2012c},  is a useful variation of standard maximal inequalities.
\begin{lemma}[Maximal Inequality]\label{lem: symmetrization inequality 1}
Let $x_{1},\dots,x_{n}$ be independent random vectors in $\RR^{p}$ with $p \geq 2$.  Let $M=\max_{1\leq i\leq n}\max_{1\leq j\leq p}|x_{ij}|$ and $\sigma^2=\max_{1\leq j\leq p}\barEp[x_{ij}^2]$.
Then
\begin{equation*}
\Ep \left[ \max_{1\leq j\leq p}|\En[x_{ij}]-\barEp[x_{ij}]| \right ] \lesssim  \sigma\sqrt{ (\log p)/n} + \sqrt{\Ep[M^2]}(\log p)/n.
\end{equation*}
\end{lemma}

\begin{proof}
See \cite{ChernozhukovChetverikovKato2012c}, Lemma 8.
\end{proof}

\subsection{Properties of the Smooth Max Function}
We will use the following properties of the smooth max function.

\begin{lemma}[Properties of $F_\beta$]\label{lemma: smooth max property} For every $1 \leq j,k,l \leq p$,
\begin{eqnarray*}
\ \ \partial_{j} F_{\beta}(z) = \pi_{j}(z), \ \ \partial_{j} \partial_k F_{\beta}(z) = \beta w_{jk} (z), \ \  \partial_{j} \partial_k \partial_l F_{\beta}(z) = \beta^2 q_{jkl}(z).
\end{eqnarray*}
where, for $\delta_{jk}  := 1\{j=k\}$,
\begin{align*}
\pi_{j}(z) & :=    e^{\beta z_{j}}/{\textstyle \sum}_{m=1}^pe^{\beta z_m}, \ w_{jk}(z) :=  ( \pi_{j} \delta_{jk} - \pi_{j} \pi_{k}) (z), \ \  \\
q_{jkl}(z) & :=  (\pi_{j}\delta_{jl}\delta_{jk}-\pi_{j}\pi_{l}\delta_{jk} -\pi_{j}\pi_{k}(\delta_{jl}+\delta_{kl})    + 2\pi_{j}\pi_{k}\pi_{l})(z).
\end{align*}
Moreover,
\begin{equation*}
\pi_{j}(z) \geq 0,  \  {\textstyle \sum}_{j=1}^p \pi_{j} (z) = 1, \ {\textstyle \sum}_{j,k=1}^p |w_{jk}(z)|  \leq 2, \ {\textstyle \sum}_{j,k,l=1}^p |q_{jkl}(z)| \leq 6 .
\end{equation*}
\end{lemma}

\begin{proof}[Proof of Lemma \ref{lemma: smooth max property}]
The first property was noted in \cite{Chatterjee2005b}.  The other properties follow from repeated application of the chain rule.
\end{proof}

\begin{lemma}[Lipschitz Property of $F_\beta$]\label{lemma: lipschitz F} For every $x \in \RR^p$ and $z \in \RR^p$, we have $ |F_\beta(x) - F_{\beta}(z)| \leq \max_{1\leq j \leq p} |x_{j} - z_{j}|$.
\end{lemma}
\begin{proof}[Proof of Lemma \ref{lemma: lipschitz F}] The proof follows from the fact that $\partial_{j} F_{\beta}(z) = \pi_{j}(z)$ with $\pi_{j}(z) \geq 0$ and $\sum_{j=1}^{p} \pi_{j}(z)=1$.
\end{proof}

 We will also use the following properties of $m = g \circ F_{\beta}$. We assume $g \in C_{b}^{3}(\RR)$ in Lemmas \ref{lemma: second deriv m}-\ref{lemma: switching property} below.

\begin{lemma}[Three derivatives of $m = g\circ F_\beta$]\label{lemma: second deriv m} For every $1 \leq j,k,l \leq p$,
\begin{align*}
\partial_{j} m (z) & = ( \partial g(F_\beta) \pi_{j}) (z), \ \partial_{j} \partial_k m(z) =  (\partial^2 g (F_\beta) \pi_{j} \pi_{k} +  \partial g (F_\beta) \beta w_{jk}) (z),\\
\partial_{j} \partial_k \partial_l m(z) & = (\partial^3 g(F_\beta) \pi_{j} \pi_k \pi_l +
\partial^2 g (F_\beta)  \beta (w_{jk} \pi_l + w_{jl} \pi_k + w_{kl} \pi_{j}) \\
& \quad +   \partial g(F_\beta)  \beta^2 q_{jkl})(z), \ \
\end{align*}
where $\pi_{j}$, $w_{jk}$ and $q_{jkl}$ are defined in Lemma \ref{lemma: smooth max property}, and $(z)$ denotes evaluation at $z$, including evaluation of $F_\beta$ at $z$.
\end{lemma}
\begin{proof}[Proof of lemma \ref{lemma: second deriv m}]
The proof follows from  repeated application
of the chain rule and by the properties noted in Lemma \ref{lemma: smooth max property}.
\end{proof}

\begin{lemma}[Bounds on derivatives of $m = g \circ F_{\beta}$]\label{lemma: bounds on derivatives of m}
For every $1 \leq j,k,l \leq p$,
\begin{equation*}
|\partial_{j} \partial_k m(z)| \leq U_{jk}(z), \quad  |\partial_{j} \partial_k \partial_l m(z)| \leq  U_{jkl}(z),
\end{equation*}
where
\begin{align*}
&U_{jk}(z):= (G_2 \pi_{j} \pi_{k} +  G_1 \beta W_{jk}) (z), \ W_{jk}(z) :=  ( \pi_{j} \delta_{jk} + \pi_{j} \pi_{k}) (z), \\
&U_{jkl}(z) := ( G_3 \pi_{j} \pi_k \pi_l + G_2 \beta (W_{jk} \pi_l +W_{jl} \pi_k + W_{kl} \pi_{j}) + G_1  \beta^2 Q_{jkl})(z), \\
&Q_{jkl}(z)  :=  (\pi_{j}\delta_{jl}\delta_{jk}+\pi_{j}\pi_{l}\delta_{jk} +\pi_{j}\pi_{k}(\delta_{jl}+\delta_{kl})    + 2\pi_{j}\pi_{k}\pi_{l})(z).
\end{align*}
Moreover,
\begin{equation*}
{\textstyle \sum}_{j,k=1}^p U_{jk}(z) \leq (G_2  +  2 G_1 \beta), \ {\textstyle \sum}_{j,k,l=1}^p U_{jkl}(z) \leq ( G_3  +  6 G_2 \beta + 6G_1  \beta^2).
\end{equation*}
\end{lemma}

\begin{proof}[Proof of Lemma \ref{lemma: bounds on derivatives of m}]
The lemma follows from a direct calculation.
\end{proof}

The following lemma plays a critical role.

\begin{lemma}[Stability Properties of Bounds over Large Regions]\label{lemma: switching property}
For every $z \in \RR^p$, $w \in \RR^p$ with $\max_{j \leq p}|w_{j}| \beta \leq 1$, $\tau \in [0,1]$, and every $1 \leq j,k,l \leq p$, we have
\begin{align*}
U_{jk}(z)  \lesssim U_{jk}(z+ \tau w) \lesssim  U_{jk}(z), \
U_{jkl}(z)  \lesssim  U_{jkl}(z+ \tau w) \lesssim  U_{jkl}(z).
\end{align*}
\end{lemma}

\begin{proof}[Proof of Lemma \ref{lemma: switching property}]
Observe that
\begin{equation*}
\pi_{j}(z + \tau w)
=\frac{e^{z_{j} \beta + \tau w_{j} \beta}}{ \sum_{m=1}^p  e^{ z_m \beta + \tau w_m \beta}}
\leq \frac{e^{z_{j} \beta}}{ \sum_{m=1}^p  e^{z_m \beta}} \cdot \frac{e^{\tau \max_{j \leq p} |w_{j}| \beta }}{ e^{ - \tau \max_{j \leq p} |w_{j}| \beta }}  \leq  e^{2} \pi_{j}(z).
\end{equation*}
Similarly, $\pi_{j}(z + \tau w) \geq  e^{-2} \pi_{j}(z)$. Since $U_{jk}$ and $U_{jkl}$ are finite sums of  products  of terms such as $\pi_{j}$, $\pi_k$, $\pi_l$, $\delta_{jk}$, the claim of the lemma follows.
\end{proof}

\subsection{Lemma on Truncation}
\label{sec: truncation}

The proof of Theorem \ref{theorem:comparison non-Gaussian} uses the following properties of the truncation operation.
Define $\tilde x_{i} = ( \tilde x_{ij})_{j=1}^p$ and $\tilde X$ $ =$ $n^{-1/2}$ $\sum_{i=1}^n \tilde x_{i}$,
where ``tilde" denotes the truncation operation defined in Section 2. The following lemma
also covers the special case where $(x_{i})_{i=1}^n = (y_{i})_{i=1}^n$. The property (d) is a
consequence of sub-Gaussian inequality of \cite{Shao2009}, Theorem 2.16, for self-normalized sums.
\begin{lemma}[Truncation Impact]
\label{lemma: truncation}
For every $1 \leq j,k \leq p$ and  $q \geq 1$,
(a) $(\barEp[|\tilde x_{ij}|^q])^{1/q} \leq 2 (\barEp[|x_{ij}|^q])^{1/q}$;
(b) $\bar \Ep[ | \tilde x_{ij} \tilde x_{ik}-  x_{ij}  x_{ik} | ] \leq  (3/2) ( \barEp[x_{ij}^2] + \barEp[x^2_{ik}]) \varphi(u)$;
(c) $\En[(\Ep[ x_{ij}1\{ |x_{ij}| > u(\barEp[ x_{ij}^{2}])^{1/2} \} ])^{2}]  \leq \barEp[x_{ij}^2] \varphi^{2} (u)$.
Moreover, for a given $\gamma \in (0,1)$,  let  $u \geq u(\gamma)$ where $u(\gamma)$ is defined in Section \ref{sec: Gaus vs NonGaus}.  Then:  (d) with probability at least $1 - 5 \gamma$,  for all $1 \leq j \leq p$,
\begin{equation*}
 | X_{j} - \tilde X_{j}| \leq 5\sqrt{\barEp[x_{ij}^2] } \varphi(u)  \sqrt{2 \log (p/\gamma)}.
\end{equation*}
\end{lemma}

\begin{proof} See Section \ref{sec: additional proofs} of SM \cite{CCK13}.
\end{proof}

\section{Proofs for Section \ref{sec: Gaus vs NonGaus}}

\subsection{Proof of Theorem \ref{theorem:comparison non-Gaussian}}
The second claim of the theorem follows from property (\ref{eq: smooth max property}) of the smooth max function.
Hence we shall prove the first claim.
The proof strategy is similar to the proof of Lemma \ref{thm: warmup comparison}.
However, to control effectively the third order terms in the leave-one-out expansions
we shall use truncation and  replace $X$ and $Y$ by their truncated versions $\tilde X$ and $\tilde Y$, defined as follows:
let $\tilde x_{i}= ( \tilde x_{ij})_{j=1}^p$, where $\tilde x_{ij}$ was defined before the statement of the theorem, and define the truncated version of $X$
as $\tilde X  = n^{-1/2} \sum_{i=1}^n \tilde x_{i}$.
Also let
\begin{equation*}
\tilde y_{i}:= ( \tilde y_{ij})_{j=1}^p,  \ \tilde y_{ij} := y_{ij} 1 \left\{ |y_{ij}| \leq u (\barEp[y_{ij}^2])^{1/2} \right\},
\ \tilde Y  =\frac{1}{\sqrt{n}} \sum_{i=1}^n \tilde y_{i}.
\end{equation*}
Note that by the symmetry of the distribution of $y_{ij}$, $\Ep [\tilde y_{ij} ] = 0$.
Recall that we are assuming that sequences $(x_{i})_{i=1}^n$ and $(y_{i})_{i=1}^n$ are independent.

The proof consists of four steps.
Step 1 will show that we can replace $X$ by $\tilde X$ and $Y$ by $\tilde Y$.
Step 2 will bound the difference of the expectations of the relevant functions of $\tilde X$ and $\tilde Y$.
This is the main step of the proof.
Steps 3 and 4 will carry out supporting calculations.
The steps of the proof will also call on various technical lemmas collected in Appendix \ref{sec: auxiliary lemmas}.

\textbf{Step 1.}   Let $m := g \circ F_\beta$. The main goal is to bound $\Ep[m (X) - m(Y)]$.
Define
\begin{equation*}
\mathcal{I} = 1\left\{  \max_{1\leq j \leq p} | X_{j} - \tilde X_{j}| \leq \Delta(\gamma, u) \ \text{and}  \ \max_{1\leq j \leq p} | Y_{j} - \tilde Y_{j}|  \leq \Delta(\gamma, u) \right\},
\end{equation*}
where $\Delta(\gamma,u) :=  5M_2  \varphi(u) \sqrt{ 2\log (p/\gamma)}$.
By Lemma \ref{lemma: truncation}, we have  $\Ep[\mathcal{I}] \geq 1- 10 \gamma$.  Observe that by Lemma \ref{lemma: lipschitz F},
\begin{equation*}
| m (x) - m (y) | \leq G_{1} | F_{\beta}(x) - F_{\beta}(y) | \leq G_{1} \max_{1 \leq j \leq p} | x_{j} - y_{j} |,
\end{equation*}
so that
\begin{align*}
|\Ep[ m(X) - m(\tilde X)]| &\leq  | \Ep[ ( m(X) - m(\tilde X) ) \mathcal{I}] | + | \Ep [( m(X) - m(\tilde X) ) (1- \mathcal{I})]| \\
&\lesssim  G_1 \Delta(\gamma, u) + G_0  \gamma, \\
|\Ep[ m(Y) - m(\tilde Y)]| &\leq | \Ep[ ( m(Y) - m(\tilde Y) )  \mathcal{I}] | + | \Ep [( m(Y) - m(\tilde Y) ) (1- \mathcal{I})] |\\
&\lesssim  G_1 \Delta(\gamma, u) + G_0  \gamma,
\end{align*}
hence
\begin{align*}
|  \Ep[m (X) - m(Y)] | &\lesssim  |\Ep[m (\tilde X) - m(\tilde Y)]| + G_1 \Delta(\gamma, u) +   G_0 \gamma.
\end{align*}

\textbf{Step 2.} (Main Step)
The purpose of this step is to establish the bound:
\begin{equation*}
|\Ep[m (\tilde X) - m(\tilde Y)]| \lesssim  n^{-1/2}  ( G_3  +  G_2 \beta + G_1  \beta^2)M_3^3 + (G_2 +  \beta G_1) M_2^2 \varphi(u).
\end{equation*}

We define the Slepian interpolation $Z(t)$ between $\tilde Y$ and $\tilde Z$, Stein's leave-one-out version $Z^{(i)}(t)$ of $Z(t)$,
and other useful terms:
\begin{align*}
&Z(t) := \sqrt{t} \tilde X + \sqrt{1-t} \tilde Y =\sum_{i=1}^n Z_{i}(t),  \ \ Z_{i}(t) := \frac{1}{\sqrt{n}} ( \sqrt{t} \tilde x_{i}  + \sqrt{1-t} \tilde y_{i}), \ \text{and} \\
&Z^{(i)}(t): = Z(t) - Z_{i}(t), \ \dot Z_{ij}(t) =\frac{1}{\sqrt{n}} \(\frac{1}{\sqrt{t}} \tilde x_{ij} - \frac{1}{\sqrt{1-t}} \tilde y_{ij}\).
\end{align*}
We have by Taylor's theorem,
\begin{equation*}
 \Ep[m (\tilde X) - m(\tilde Y)] =\frac{1}{2} \sum_{j=1}^p \sum_{i=1}^n \int_0^1 \Ep[\partial_{j} m(Z(t)) \dot Z_{ij}(t) ]  dt = \frac{1}{2}(I+II+III),
\end{equation*}
where
{\small \begin{align*}
I &=  \sum_{j=1}^p \sum_{i=1}^n \int_0^1 \Ep [  \partial_{j} m(Z^{(i)}(t) )  \dot Z_{ij}(t) ]dt, \\
II &= \sum_{j,k=1}^p \sum_{i=1}^n \int_0^1 \Ep [ \partial_{j} \partial_k  m(Z^{(i)}(t) )   \dot Z_{ij}(t) Z_{ik}(t) ] dt, \\
III &=  \sum_{j,k,l=1}^p \sum_{i=1}^n \int_0^1 \int_0^1 (1-\tau) \Ep [  \partial_{j} \partial_k \partial_l m( Z^{(i)}(t) + \tau Z_{i}(t)) \dot Z_{ij}(t) Z_{ik}(t) Z_{il}(t) ] d \tau dt.
\end{align*}}
\!By independence of $ Z^{(i)}(t)$ and $ \dot Z_{ij}(t)$ together with the fact that  $\Ep [\dot Z_{ij}(t)] =0$, we have $I=0$.
Moreover, in Steps 3 and 4 below, we will show that
\begin{equation*}
| II | \lesssim (G_2 +  \beta G_1) M_2^2 \varphi(u), \
| III | \lesssim  n^{-1/2} ( G_3  +  G_2 \beta + G_1  \beta^2) M_3^3.
\end{equation*}
The claim of this step now follows.

\textbf{Step 3.} (Bound on $II$) By independence of $Z^{(i)}(t)$ and $\dot Z_{ij}(t) Z_{ik}(t)$,
\begin{align*}
| II | &= \left | \sum_{j,k=1}^p \sum_{i=1}^n  \int_0^1 \Ep[ \partial_{j} \partial_k m(Z^{(i)}(t))]  \Ep[ \dot Z_{ij}(t) Z_{ik}(t)] dt \right | \\
&\leq  \sum_{j,k=1}^p \sum_{i=1}^n  \int_0^1 \Ep[| \partial_{j} \partial_k m(Z^{(i)}(t) )| ] \cdot  | \Ep[ \dot Z_{ij}(t) Z_{ik}(t)] | dt \\
&\leq  \sum_{j,k=1}^p \sum_{i=1}^n  \int_0^1  \Ep[U_{jk}( Z^{(i)}(t)) ]\ \cdot  |\Ep[\dot Z_{ij}(t) Z_{ik}(t)] | dt,
\end{align*}
where the last step follows from Lemma \ref{lemma: bounds on derivatives of m}.
Since $|\sqrt{t}\tilde{x}_{ij}+\sqrt{1-t}\tilde{y}_{ij}|\leq 2\sqrt{2}u M_2$, so that $|\beta(\sqrt{t}\tilde{x}_{ij}+\sqrt{1-t}\tilde{y}_{ij})/\sqrt{n}|\leq 1$ (which is satisfied by the assumption $ \beta 2\sqrt{2}u M_2/\sqrt{n} \leq 1$),
by Lemmas \ref{lemma: switching property} and \ref{lemma: bounds on derivatives of m}, the last expression is bounded up to an absolute constant by
\begin{align*}
&\sum_{j,k=1}^p \sum_{i=1}^n  \int_0^1  \Ep[U_{jk}( Z(t)) ] \cdot  |\Ep[\dot Z_{ij}(t) Z_{ik}(t)] | dt \\
&=  \int_0^1  \left  \{ \sum_{j,k=1}^p \Ep[U_{jk}( Z(t)) ] \right \} \max_{1\leq j,k\leq p}\sum_{i=1}^n   |\Ep[\dot Z_{ij}(t) Z_{ik}(t)] | dt \\
&\lesssim  (G_2 + G_1 \beta) \int_0^1 \max_{1\leq j,k\leq p}\sum_{i=1}^n   |\Ep[\dot Z_{ij}(t) Z_{ik}(t)] | dt.
\end{align*}
Observe that since $\Ep[ x_{ij} x_{ik} ] = \Ep [ y_{ij} y_{ik} ]$, we have that $\Ep[\dot Z_{ij}(t) Z_{ik}(t)] = n^{-1}  \Ep [ \tilde{x}_{ij} \tilde{x}_{ik} - \tilde{y}_{ij} \tilde{y}_{ik} ] = n^{-1} \Ep [ \tilde{x}_{ij} \tilde{x}_{ik}  - x_{ij}x_{ik} ] + n^{-1} \Ep [ y_{ij} y_{ik} - \tilde{y}_{ij} \tilde{y}_{ik}],$ so that by  Lemma \ref{lemma: truncation} (b), $
 \sum_{i=1}^n | \Ep[\dot Z_{ij}(t) Z_{ik}(t)] | \leq  \bar \Ep [| \tilde{x}_{ij} \tilde{x}_{ik}  - x_{ij}x_{ik}| ] + \bar \Ep [ | y_{ij} y_{ik} - \tilde{y}_{ij} \tilde{y}_{ik} |]
\lesssim (\barEp[x_{ij}^2] + \barEp[x_{ik}^2] ) \varphi(u) \lesssim M_{2}^{2}  \varphi(u).
$ Therefore, we conclude that $| II | \lesssim (G_2 + G_1 \beta) M_{2}^{2} \varphi(u).$ \\

\textbf{Step 4.} (Bound on $III$) Observe that
\begin{align}
| III |  &\leq_{(1)}  \sum_{j,k,l=1}^p \sum_{i=1}^n \int_0^1 \int_0^1 \Ep  [ U_{jkl}( Z^{(i)}(t) + \tau Z_{i}(t)) |\dot Z_{ij}(t) Z_{ik}(t) Z_{il}(t)| ] d \tau dt \notag \\
& \lesssim_{(2)}  \sum_{j,k,l=1}^p \sum_{i=1}^n \int_0^1  \Ep  [ U_{jkl}( Z^{(i)}(t)) |\dot Z_{ij}(t) Z_{ik}(t) Z_{il}(t)| ] dt \notag \\
& =_{(3)} \sum_{j,k,l=1}^p \sum_{i=1}^n \int_0^1  \Ep [ U_{jkl}( Z^{(i)}(t)) ] \cdot \Ep [ |\dot Z_{ij}(t) Z_{ik}(t) Z_{il}(t)| ] dt, \label{eq:last}
\end{align}
where (1) follows from $|\partial_{j} \partial_k \partial_l m(z)| \leq U_{jkl}(z)$  (see Lemma \ref{lemma: bounds on derivatives of m}),
(2) from  Lemma \ref{lemma: switching property},
(3) from independence of $Z^{(i)}(t)$ and $\dot Z_{ij}(t) Z_{ik}(t) Z_{il}(t)$.
Moreover, the last expression is bounded as follows:
\begin{align*}
& \text{right-hand side of (\ref{eq:last})} \lesssim_{(4)}  \sum_{j,k,l=1}^p \sum_{i=1}^n \int_0^1  \Ep [ U_{jkl}( Z(t)) ] \cdot \Ep [ |\dot Z_{ij}(t) Z_{ik}(t) Z_{il}(t)| ] dt   \\
& =_{(5)} \sum_{j,k,l=1}^p  \int_0^1  \Ep [ U_{jkl}( Z(t)) ] \cdot n \barEp [  |\dot Z_{ij}(t) Z_{ik}(t) Z_{il}(t)| ] dt\\
& \leq_{(6)}   \int_0^1 \left( \sum_{j,k,l=1}^p  \Ep [ U_{jkl}( Z(t)) ]  \right) \max_{1\leq j,k,l \leq p } n \barEp[  |\dot Z_{ij}(t) Z_{ik}(t) Z_{il}(t)| ] dt \\
& \lesssim_{(7)}   ( G_3  +  G_2 \beta + G_1  \beta^2)   \int_0^1 \max_{1\leq j,k,l \leq p }  n \barEp [  |\dot Z_{ij}(t) Z_{ik}(t) Z_{il}(t)| ] dt,
\end{align*}
where
(4) follows from  Lemma \ref{lemma: switching property},
(5) from  definition of $\barEp$,
(6) from a trivial inequality,
(7) from Lemma \ref{lemma: bounds on derivatives of m}.
We have to bound the integral on the last line.
Let $\omega(t)=1/(\sqrt{t} \wedge \sqrt{1-t})$, and observe that
 \begin{align*}
&\int_0^1 \max_{1\leq j,k,l \leq p }  n \barEp[  |\dot Z_{ij}(t) Z_{ik}(t) Z_{il}(t)| ]dt   \\
&= \int_0^1 \omega(t) \max_{1\leq j,k,l \leq p }  n \barEp[  |(\dot Z_{ij}(t)/\omega(t))  Z_{ik}(t) Z_{il}(t)| ]dt \\
&\leq n \int_0^1 \omega(t) \max_{1\leq j,k,l \leq p } \(\barEp[|\dot Z_{ij}(t)/\omega(t)|^3] \barEp[| Z_{ik}(t) |^3] \barEp[| Z_{il}(t) |^3] \) ^{1/3} dt,
\end{align*}
where the last inequality is by H\"{o}lder. The last term is further bounded as
 \begin{align*}
&\leq_{(1)} n^{-1/2} \left \{ \int_0^1 \omega(t) dt \right \} \max_{1\leq j \leq p } \barEp[ \(|\tilde x_{ij}| + |\tilde y_{ij}|\)^3] \\
&\lesssim_{(2)}  n^{-1/2}   \max_{1\leq j \leq p } ( \barEp[|\tilde x_{ij}|^{3}] + \barEp[|\tilde y_{ij}|^{3}]) \\
&\lesssim_{(3)} n^{-1/2}   \max_{1\leq j \leq p } ( \barEp[| x_{ij}|^{3}] + \barEp[|y_{ij}|^{3}] ) \\
&\lesssim_{(4)}  n^{-1/2}   \max_{1 \leq j \leq p }  \barEp[| x_{ij} |^{3}],
\end{align*}
where (1) follows from the fact that:
$|\dot Z_{ij}(t)/\omega(t) | \leq (| \tilde x_{ij}| + | \tilde y_{ij}|)/\sqrt{n} $,
$|Z_{im}(t)| \leq (| \tilde x_{im}| + | \tilde y_{im}|)/\sqrt{n}$, and
the product of terms $\barEp[ \(|\tilde x_{ij}| + |\tilde y_{ij}|\)^3]^{1/3}$,
$\barEp[ \(|\tilde x_{ik}| + |\tilde y_{ik}|\)^3]^{1/3}$ and  $\barEp[ \(|\tilde x_{il}| + |\tilde y_{il}|\)^3]^{1/3}$
is trivially bounded by $\max_{1\leq j \leq p }$ $\barEp[ (|\tilde x_{ij}| + |\tilde y_{ij}|)^3];$
 (2) follows from $\int_0^1 \omega(t) dt \lesssim 1$,
(3) from Lemma \ref{lemma: truncation} (a),
and (4) from the normality of $y_{ij}$ with $\Ep [y_{ij}^2] =\Ep[x^2_{ij}] $, so that
$\Ep[|y_{ij}|^{3}] \lesssim (\Ep[y_{ij}^2])^{3/2} =  (\Ep[|x_{ij}^2|])^{3/2} \leq  \Ep[|x_{ij}|^{3}]$.
This completes the overall proof.
\qed

\subsection{Proof of Theorem \ref{cor: Gaussian to nonGaussian KS 2}} See Appendix \ref{Proof of cor: Gaussian to nonGaussian KS 2} of the SM \cite{CCK13}. \qed

\subsection{Proof of Lemma \ref{lem: bound on u}}
Since $ \barEp [ x_{ij}^{2} ] \geq c_{1}$ by assumption,  we have  $
1\{ | x_{ij} | > u (\bar \Ep [ x_{ij}^{2} ])^{1/2} \} \leq 1 \{ | x_{ij} | > c^{1/2}_{1} u \}.$
By Markov's inequality and the condition of the lemma, we have
\begin{align*}
&\Pr \left( | x_{ij} | > u (\bar \Ep [ x_{ij}^{2} ])^{1/2}, \ \text{for some $(i,j)$} \right)
\leq {\textstyle  \sum}_{i=1}^{n} \Pr \left( \max_{1 \leq j \leq p} | x_{ij} | > c^{1/2}_1u \right) \\
&\quad \leq  {\textstyle  \sum}_{i=1}^{n} \Pr \left( h(\max_{1 \leq j \leq p} | x_{ij} |/D) > h(c^{1/2}_1u/D) \right) \leq n/h(c^{1/2}_1u/D ).
\end{align*}
This implies $u_x(\gamma) \leq c_{1}^{-1/2}Dh^{-1}(n/\gamma)$.
For $u_{y}(\gamma)$, by $y_{ij} \sim N(0,\Ep[ x_{ij}^{2} ])$ with $\Ep [ x_{ij}^{2} ] \leq B^{2}$, we have $\Ep [ \exp ( y_{ij}^{2}/(4B^{2})) ] \lesssim 1$. Hence
\begin{align*}
&\Pr \left( | y_{ij} | > u (\bar \Ep [ y_{ij}^{2} ])^{1/2}, \ \text{for some $(i,j)$} \right)
\leq {\textstyle  \sum}_{i=1}^{n} {\textstyle  \sum}_{j=1}^{p}\Pr (  | y_{ij} | > c^{1/2}_1u ) \\
&\quad \leq  {\textstyle  \sum}_{i=1}^{n} {\textstyle  \sum}_{j=1}^{p}\Pr (  | y_{ij} |/(2B) > c^{1/2}_1u/(2B) ) \lesssim np \exp (- c_1u^{2}/(4B^{2})).
\end{align*}
Therefore,  $u_y(\gamma) \leq C B \sqrt{\log (pn/\gamma)}$ where $C >0$ depends only on $c_{1}$. \qed

\subsection{Proof of Corollary \ref{cor: central limit theorem}}
Since conditions (E.3) and (E.4) are special cases of (E.1) and (E.2), it suffices to prove the result under conditions (E.1) and (E.2) only. The proof consists of two steps.

\textbf{Step 1. }  In this step,  in each case of conditions (E.1) and (E.2), we shall compute the following bounds
on  moments $M_3$ and $M_{4}$ and parameters $B$ and $D$ in Lemma \ref{lem: bound on u} with
specific choice of $h$:
\begin{itemize}
\item[(E.1)] \ $B \vee M_3^3 \vee M_4^2 \leq CB_n$, $D \leq CB_{n}  \log p$, $h(v) = e^{v} - 1$;
\item[(E.2)] \ $B \vee D \vee M_3^3 \vee M_4 ^2\leq C B_n$, $h(v) = v^4$;
\end{itemize}
Here $C >0$ is a (sufficiently large) constant that depends only on $c_{1}$ and $C_{1}$. The bounds on ${B}$, $M_3$ and $M_4$ follow
from elementary computations using H\"{o}lder's inequality.
The bounds on $D$ follow from an elementary application of Lemma 2.2.2 in \cite{VW96}. For brevity, we omit the detail.

\textbf{Step 2. }  In all cases, there are sufficiently small constants $c_{3} >0 $ and $c_{4} > 0$, and a sufficiently large constant $C_{3} > 0$, depending only on $c_{1},C_1,c_{2},C_2$ such that, with $\ell_{n} := \log (pn^{1+c_{3}})$,
\begin{align*}
&n^{-1/2} \ell_{n}^{3/2} \max \{ B \ell_{n}^{1/2}, D h^{-1}(n^{1+c_{3}}) \} \leq C_{3} n^{-c_{4}}, \\
&n^{-1/8} (M^{3/4}_{3} \vee M^{1/2}_{4}) \ell_{n}^{7/8} \leq C_{3} n^{-c_{4}}.
\end{align*}
Hence taking $\gamma = n^{-c_3}$,  we conclude  from Theorem \ref{cor: Gaussian to nonGaussian KS 2} and Lemma \ref{lem: bound on u} that   $\rho  \leq Cn^{- \min \{ c_{3},c_{4} \}}$ where $C>0$ depends only on $c_{1},C_1,c_{2},C_{2}$.
\qed

\section{Proofs for Section \ref{sec: multiplier bootstrap}}\label{sec: multiplier bootstrap proofs}
\subsection{Proof of Lemma \ref{lem: quantile conditional to unconditional}} Recall that $\Delta= \max_{1\leq j,k\leq p}|\En[x_{ij}x_{ik}]-\barEp[x_{ij}x_{ik}]|$. By Lemma \ref{lemma: distances Gaussian to Gaussian}, on the event $\{ (x_{i})_{i=1}^{n}: \Delta \leq \vartheta\}$, we have $|\Pr(Z_0\leq t)- \Pr_e(W_0\leq t)|\leq \pi(\vartheta)$ for all $t\in\RR$, and so on this event
{\small \begin{equation*}
\Pr_e(W_0\leq c_{Z_0}(\alpha+\pi(\vartheta)))\geq \Pr(Z_0\leq c_{Z_0}(\alpha+\pi(\vartheta)))-\pi(\vartheta)\geq \alpha+\pi(\vartheta)-\pi(\vartheta)=\alpha,
\end{equation*}}
\!implying the first claim. The second claim follows similarly.
\qed

\subsection{Proof of Lemma \ref{lem: quantile approximated to exact}}
By equation (\ref{eq: conditional quantiles}), the probability of the event $\{ (x_{i})_{i=1}^{n} : \Pr_e(|W-W_0|>\zeta_1)\leq \zeta_2\}$ is at least $1-\zeta_2$. On this event,
\begin{equation*}
\Pr_e(W\leq c_{W_0}(\alpha+\zeta_2)+\zeta_1)
\geq
\Pr_e(W_0\leq c_{W_0}(\alpha+\zeta_2))-\zeta_2
\geq
\alpha+\zeta_2-\zeta_2=\alpha,
\end{equation*}
implying that $\Pr(c_{W}(\alpha)\leq c_{W_0}(\alpha+\zeta_2)+\zeta_1)\geq 1-\zeta_2$. The second claim of the lemma follows similarly.
\qed

\subsection{Proof of Theorem \ref{thm: multiplier bootrstrap I}}
For $\vartheta>0$, let $\pi(\vartheta):=C_2\vartheta^{1/3}(1 \vee \log(p/\vartheta))^{2/3}$ as defined in Lemma \ref{lem: quantile conditional to unconditional}. To prove the first inequality, note that
\begin{align*}
&\Pr(\{T_0\leq c_{W_0}(\alpha)\}\ominus\{T_0\leq c_{Z_0}(\alpha)\})\\
&\leq_{(1)}\Pr(c_{Z_0}(\alpha-\pi(\vartheta)) <  T_0\leq c_{Z_0}(\alpha+\pi(\vartheta)))+2\Pr(\Delta>\vartheta)\\
&\leq_{(2)}\Pr(c_{Z_0}(\alpha-\pi(\vartheta)) <  Z_0\leq c_{Z_0}(\alpha+\pi(\vartheta)))+2\Pr(\Delta>\vartheta)+2\rho\\
&\leq_{(3)}2\pi(\vartheta)+2\Pr(\Delta>\vartheta)+2\rho,
\end{align*}
where (1) follows from Lemma \ref{lem: quantile conditional to unconditional}, (2) follows from the definition of $\rho$, and (3) follows from the fact that $Z_0$ has no point masses. The first inequality follows. The second inequality follows from the first inequality and the definition of $\rho$.
\qed

\subsection{Proof of Theorem \ref{thm: multiplier bootrstrap II}}
For $\vartheta>0$, let $\pi(\vartheta):=C_2\vartheta^{1/3}(1 \vee \log(p/\vartheta))^{2/3}$ with $C_2>0$ as in Lemma \ref{lem: quantile conditional to unconditional}. In addition, let $\kappa_1(\vartheta):=c_{Z_0}(\alpha-\zeta_2-\pi(\vartheta))$  and $\kappa_2(\vartheta):=c_{Z_0}(\alpha+\zeta_2+\pi(\vartheta))$. To prove the first inequality, note that
\begin{align*}
&\Pr(\{T\leq c_{W}(\alpha)\}\ominus\{T_0\leq c_{Z_0}(\alpha)\})\\
&\leq_{(1)}\Pr(\kappa_1(\vartheta)-2\zeta_1 < T_0\leq \kappa_2(\vartheta)+2\zeta_1)+2\Pr(\Delta>\vartheta)+3\zeta_2\\
&\leq_{(2)}\Pr(\kappa_1(\vartheta)-2\zeta_1 <  Z_0\leq \kappa_2(\vartheta)+2\zeta_1)+2\Pr(\Delta>\vartheta)+2\rho+3\zeta_2\\
&\leq_{(3)} 2\pi(\vartheta)+2\Pr(\Delta>\vartheta)+2\rho+C_3\zeta_1\sqrt{1 \vee \log (p/\zeta_1)}+5\zeta_2
\end{align*}
\!where $C_3>0$ depends on $c_1$ and $C_1$ only and
where (1) follows from equation (\ref{eq: statistic approximation}) and Lemmas \ref{lem: quantile conditional to unconditional} and \ref{lem: quantile approximated to exact}, (2) follows from the definition of $\rho$, and (3) follows from Lemma \ref{lem: anticoncentration} and the fact that $Z_0$ has no point masses. The first inequality follows. The second inequality follows from the first inequality and the definition of $\rho$.
\qed

\subsection{Proof of Corollary \ref{cor: multiplier bootstrap examples}}

Since conditions (E.3) and (E.4) are special cases of (E.1) and (E.2), it suffices to prove the result under conditions (E.1) and (E.2) only. The proof of this corollary relies on:
\begin{lemma}
\label{lem: technical2}
Recall conditions (E.1)-(E.2) in Section \ref{sub: examples of applications GAR}. Then
\begin{equation*}
\Ep [ \Delta ]
\leq C \times
\begin{cases}
\sqrt{\frac{B_{n}^{2}\log p}{n}} \bigvee \frac{B_{n}^{2}(\log (pn))^{2}(\log p)}{n}, & \text{under (E.1)}, \\
\sqrt{\frac{B_{n}^{2}\log p}{n}} \bigvee \frac{B_{n}^{2}(\log p)}{\sqrt{n}}, & \text{under (E.2)},
\end{cases}
\end{equation*}
where $C>0$ depends only on $c_{1}$ and $C_{1}$ that appear in (E.1)-(E.2).
\end{lemma}

\begin{proof}
By Lemma \ref{lem: symmetrization inequality 1} and H\"{o}lder's inequality, we have
\begin{equation*}
\Ep [ \Delta ] \lesssim   M_{4}^{2}\sqrt{(\log p)/n}  + (\Ep [ \max_{i,j} | x_{ij} |^{4} ])^{1/2}(\log p)/n.
\end{equation*}
The conclusion of the lemma follows from elementary calculations  with help of Lemma 2.2.2 in \cite{VW96}.
\end{proof}

\begin{proof}[Proof of Corollary \ref{cor: multiplier bootstrap examples}]
To prove the first inequality, we make use of Theorem \ref{thm: multiplier bootrstrap II}.
Let $c  > 0$ and $C > 0$ denote generic constants depending only on $c_{1},C_1,c_{2},C_{2}$, and their values may change from place to place.
By Corollary \ref{cor: central limit theorem}, in all cases,  $\rho \leq Cn^{-c}$.
Moreover,  $\zeta_1\sqrt{\log p}\leq C_{2} n^{-c_{2}}$ implies that
$\zeta_1 \leq C_{2}n^{-c_{2}}$ (recall $p \geq 3$), and hence $\zeta_1\sqrt{\log(p/\zeta_1)} \leq Cn^{-c}$. Also, $\zeta_2\leq Cn^{-c}$ by assumption.

Let $\vartheta=\vartheta_n:=(\Ep[\Delta])^{1/2}/\log p$. By Lemma \ref{lem: technical2}, $\Ep[\Delta](\log p)^2\leq Cn^{-c}$. Therefore, $\pi(\vartheta)\leq Cn^{-c}$ (with possibly different $c,C>0$). In addition, by Markov's inequality, $\Pr(\Delta>\vartheta)\leq\Ep[\Delta]/\vartheta\leq Cn^{-c}$. Hence, by Theorem \ref{thm: multiplier bootrstrap II}, 
the first inequality follows. The second inequality follows from the first inequality and the fact that $\rho\leq Cn^{-c}$ as shown above.
\end{proof}

\section*{Acknowledgments}
The authors would like to express their appreciation to L.H.Y. Chen, David Gamarnik, Qi-Man Shao, Vladimir Koltchinskii, Enno Mammen, Axel Munk, Steve Portnoy, Adrian R\"{o}llin, Azeem Shaikh, and Larry Wasserman for enlightening discussions.  We thank the editors and referees for the comments of the highest quality that have lead to substantial improvements.

\begin{supplement}
\stitle{Supplement to ``Gaussian approximations and multiplier bootstrap for maxima of sums of high-dimensional random vectors''}
\slink[url]{}
\sdescription{This supplemental  file contains the additional technical proofs, theoretical and simulation results.}
\end{supplement}


\newpage

\begin{center}
\textbf{{\Large Supplement to ``Gaussian Approximations and Multiplier Bootstrap for Maxima of Sums of High-Dimensional Random Vectors''}}\\
\text{}\\
\end{center}

\begin{center}
{\Large by V. Chernozhukov, D. Chetverikov, and K. Kato}\\
\text{}\\
\textit{{\Large MIT, UCLA, and University of Tokyo}}\\
\text{}\\
\end{center}

\appendix

\setcounter{section}{3}

\begin{center}
{\Large  Supplementary Material I} \\
\text{} \\
{\textbf{Deferred Proofs for Results from Main Text}} \\
\end{center}


\section{Deferred proofs for Section 2}
\label{sec: additional proofs}

\subsection{Proof of Lemma \ref{lemma: truncation}}
Claim (a).
Define  $I_{ij} = 1\{|x_{ij}| \leq u (\barEp[x_{ij}^2])^{1/2} \}$, and observe that
\begin{align*}
( \barEp[|\tilde x_{ij}|^q] )^{1/q}
&\leq  ( \barEp[|x_{ij} I_{ij}|^q] )^{1/q} +   ( \En [ | \Ep[  x_{ij} I_{ij}  ] |^{q}])^{1/q} \\
&\leq ( \barEp[|x_{ij} I_{ij}|^q] )^{1/q}  + ( \barEp[|x_{ij} I_{ij}|^q] )^{1/q} \leq 2 ( \barEp[|x_{ij}|^q] )^{1/q}.
\end{align*}


Claim (b). Observe that
\begin{align*}
\barEp[ | \tilde x_{ij} \tilde x_{ik} -  x_{ij} x_{ik}| ]
&\leq \barEp[|  (\tilde x_{ij}-  x_{ij}) \tilde x_{ik} | ]  +\barEp[| x_{ij} (\tilde x_{ik} - x_{ik} ) |] \\
&\leq \sqrt{\barEp[(\tilde x_{ij} -x_{ij} )^2]} \sqrt{ \barEp[\tilde x^2_{ik}]} +   \sqrt{\barEp[(\tilde x_{ik} -x_{ik} )^2]} \sqrt{ \barEp[x^2_{ij}]} \\
&\leq  2 \varphi(u) \sqrt{ \barEp[x^2_{ij}]} \sqrt{ \barEp[x^2_{ik}]} +    \varphi(u) \sqrt{ \barEp[x^2_{ik}]} \sqrt{ \barEp[x^2_{ij}]}    \\
&\leq (3/2) \varphi(u) ( \barEp[x^2_{ij}] +  \barEp[x^2_{ik}]),
\end{align*}
where the first inequality follows from the triangle inequality, the second  from the Cauchy-Schwarz inequality, the third from the definition
of $\varphi(u)$ together with claim (a), and the last from inequality $ |a b| \leq (a^2 + b^2)/2$.

Claim (c). This follows from the Cauchy-Schwarz inequality and the definition of $\varphi(u)$.


Claim (d).
We shall use the following lemma.

\begin{lemma}[Tail Bounds for Self-Normalized Sums]\label{lemma: SN}
Let $\xi_{1},\dots,\xi_{n}$  be independent real-valued random variables such that $\Ep[\xi_{i}]=0$ and $\Ep[\xi_{i}^2] < \infty$ for all $1 \leq i \leq n$.
Let $S_n = \sum_{i=1}^n \xi_{i}$. Then for every $x > 0$,
\begin{equation*}
\Pr ( |S_n| > x (4 B_n + V_n) ) \leq 4 \exp(-x^2/2),
\end{equation*}
where $B^2_n = \sum_{i=1}^n \Ep [\xi_{i}^2]$ and $V_n^2= \sum_{i=1}^n \xi_{i}^2$.
\end{lemma}

\begin{proof}[Proof of Lemma \ref{lemma: SN}]
See \cite{Shao2009}, Theorem 2.16.
\end{proof}


Define
\begin{equation*}
\Lambda_{j} := 4 \sqrt{ \barEp[(x_{ij} - \tilde x_{ij})^2] } + \sqrt{ \En[(x_{ij} - \tilde x_{ij})^2] }.
\end{equation*}
Then by Lemma \ref{lemma: SN} and the union bound, with probability at least $1- 4 \gamma$,
\begin{equation*}
|X_{j} - \tilde X_{j}| \leq \Lambda_{j} \sqrt{ 2 \log (p/\gamma) }, \ \text{for all $1 \leq j \leq p$}.
\end{equation*}
By claim (c), for $u \geq  u(\gamma)$, with probability at least $1-\gamma$,  for all $1 \leq j \leq p$,
\begin{multline*}
\Lambda_{j} = 4 \sqrt{ \barEp[(x_{ij} - \tilde x_{ij})^2] } + \sqrt{\En[ (\Ep [x_{ij}1\{|x_{ij}|>u(\barEp[x_{ij}^2])^{1/2}\}])^{2} ]} \\
 \leq 5 \sqrt{ \barEp[x_{ij}^{2}]} \varphi(u).
\end{multline*}
The last two assertions imply claim (d).
\qed



\subsection{Proof of Theorem \ref{cor: Gaussian to nonGaussian KS 2}}\label{Proof of cor: Gaussian to nonGaussian KS 2}
Since $M_2$ is bounded from below and above by positive constants, we may normalize $M_2 =1$, without loss of generality.
 In this proof, let $C>0$ denote a generic constant depending only on $c_{1}$ and $C_{1}$, and its value may change from place to place.

For given $\gamma \in (0,1)$, denote $\ell_n := \log(pn/\gamma) \geq 1$ and let
$$
u_1:=n^{3/8}\ell_n^{-5/8}M_3^{3/4}\text{ and } u_2 :=n^{3/8}\ell_n^{-5/8}M_4^{1/2}.
$$
Define
$u := u(\gamma) \vee u_1 \vee u_2$ and $\beta:=\sqrt{n}/(2\sqrt{2}u)$. Then $u\geq u(\gamma)$ and the choice of $\beta$ trivially obeys $ 2 \sqrt{2} u \beta \leq \sqrt{n}$. So, by Theorem \ref{theorem:comparison non-Gaussian} and using the  argument as that in the proof of Corollary \ref{cor: Gaussian to nonGaussian KS 1}, for every $\psi > 0$ and any $\bar \varphi(u) \geq \varphi(u)$, we have
\begin{eqnarray}
& \rho  \leq
 C \big [ n^{-1/2} (\psi^3+\psi^2\beta+\psi\beta^2)M_3^3
+(\psi^2+\psi\beta) \bar \varphi(u) \nonumber \\
& \ \ \ \ \ \ \ \ \qquad  + \psi \bar \varphi(u)\sqrt{\log(p/\gamma)}+( \beta^{-1} \log p + \psi^{-1}) \sqrt{1 \vee \log (p \psi)} +\gamma\big ].
\label{key to 1}
\end{eqnarray}

\textbf{Step 1.} We claim that we can take $\bar\varphi(u) := C M_4^2/ u$ for all $u >0$. Since $ \barEp [ x_{ij}^{2} ] \geq c_{1}$, we have
$1\{ | x_{ij} | > u (\bar \Ep [ x_{ij}^{2} ])^{1/2} \} \leq 1 \{ | x_{ij} | > c^{1/2}_{1} u \}.$
Hence \begin{align*}
&\bar \Ep [ x_{ij}^{2} 1\{  | x_{ij} | > u (\bar \Ep [ x_{ij}^{2} ])^{1/2} \} ] \leq  \bar \Ep [ x_{ij}^{2} 1\{  | x_{ij} | > c^{1/2}_{1} u\} ] \\
&\quad \leq \bar \Ep [ x_{ij}^{4} 1\{  | x_{ij} | > c^{1/2}_1 u\} ]/(c_1u^2)   \leq \bar \Ep [ x_{ij}^{4}]/(c_{1} u^{2}) \leq  M_4^4/(c_{1} u^{2}).
\end{align*}
This implies $\varphi_{x}(u) \leq C M_4^2/u$.
For $\varphi_{y}(u)$, note that
$$
\barEp [ y_{ij}^{4}]=\En[\Ep[y_{ij}^4]] = 3 \En[(\Ep[ y_{ij}^{2}])^2] = 3 \En[(\Ep[ x_{ij}^{2}])^2]\leq 3\En[\Ep[ x_{ij}^{4}]] = 3\barEp [ x_{ij}^{4}],
$$
and hence $\varphi_y(u) \leq CM_4^2 / u$ as well. This implies the claim of this step.

\textbf{Step 2.}
We shall bound the right side of (\ref{key to 1}) by suitably choosing $\psi$ depending on the range of $u$.
In order to set up this choice we define $u^\star$ by the following equation:
$$
\bar\varphi(u^\star)n^{3/8} / (M_3^{3}\ell_n^{5/6})^{3/4}=1.
$$
We then take
\begin{eqnarray}
\psi = \psi(u) := \left\{ \begin{array}{lll} n^{1/8}\ell_n^{-3/8}M_3^{-3/4} & \text{ if } u \geq u^\star, \\
 \ell_n^{-1/6}(\bar\varphi(u))^{-1/3} & \text{ if } u < u^\star.
\end{array}\right.
\end{eqnarray}
We note that  for $u < u^\star$,
$$
 \psi(u) \leq  \psi(u^\star) = n^{1/8}\ell_n^{-3/8}M_3^{-3/4}.
$$
That is, when $u < u^\star$ the smoothing parameter $\psi$ is smaller than when $ u \geq u^\star$.

Using these choices of parameters $\beta$ and $\psi$ and elementary calculations (which will be done in Step 3 below), we conclude from (\ref{key to 1}) that whether $u < u^\star$ or $u\geq u^\star$,
\begin{equation*}
\rho \leq C(n^{-1/2} u\ell_n^{3/2} + \gamma).
\end{equation*}
The bound in the theorem follows from this inequality.
\

\textbf{Step 3.}  (Computation of the bound on $\rho$). Note that since $\rho \leq 1$, we only had to consider  the case where $n^{-1/2} u\ell_n^{3/2} \leq 1$ since otherwise the inequality is trivial by taking, say, $C = 1$. Since  $u_1=n^{3/8}M_3^{3/4}/\ell_n^{5/8}$ and $u_2 =n^{3/8}M_4^{1/2}/\ell_n^{5/8}$, we have
\begin{align*}
&(\bar\varphi(u^\star))^{4/3} =n^{-1/2}\ell_n^{5/6} M_3^{3},\\
&\bar\varphi(u_1) \leq C n^{-3/8}\ell_n^{5/8} M_4^2/M_3^{3/4},\\
&\bar\varphi(u_2) \leq C n^{-3/8}\ell_n^{5/8}M_4^{3/2}.
\end{align*}
Also note that $\psi\leq n^{1/8}$, and so $1 \vee \log (p \psi) \lesssim \log(pn) \leq \ell_n$. Therefore,
$$
\beta^{-1}\log p\sqrt{1\vee\log(p\psi)}\lesssim \beta^{-1}\ell_n^{3/2}\lesssim n^{-1/2}u\ell_n^{3/2}.
$$
 In addition, note that $\beta\lesssim \sqrt{n}/ u\leq \sqrt{n}/u_1 = n^{1/8} \ell_{n}^{5/8} M_{3}^{-3/4} =: \bar \beta$ and $\psi \leq \bar \beta$ under either case. This implies
that $(\psi^3+\psi^2\beta+\psi\beta^2) \lesssim \psi \bar \beta^2$ and $(\psi^2+\psi\beta) \leq \psi \bar \beta$. \\
Using these inequalities, we can compute the bounds claimed above.

\textbf{(a)}.  Bounding $\rho$ when $u\geq u^\star$. Then
\begin{align*}
&n^{-1/2} (\psi^3+\psi^2\beta+\psi\beta^2)M_3^3
\lesssim n^{-1/2} \psi\bar{\beta}^2M_3^3
\leq n^{-1/8}\ell_n^{7/8}M_3^{3/4} \leq n^{-1/2} u\ell_n^{3/2}; \\
&(\psi^2+\psi\beta)\bar \varphi(u)
\lesssim \psi\bar{\beta}\bar \varphi(u) \leq  \psi\bar{\beta}\bar \varphi(u^{\star})
\leq n^{-1/8}\ell_n^{7/8}M_3^{3/4}  \leq n^{-1/2} u\ell_n^{3/2}; \\
& \psi \bar \varphi(u)\sqrt{\log(p/\gamma)} \leq \psi\bar{\beta}\bar \varphi(u)\sqrt{\ell_n}/\bar{\beta} \leq \psi\bar{\beta}\bar \varphi(u^{\star})
\leq n^{-1/2} u\ell_n^{3/2};  \ \text{and} \\
&\psi^{-1} \sqrt{\ell_n}\leq n^{-1/8}\ell_n^{7/8}M_3^{3/4} \leq n^{-1/2} u\ell_n^{3/2};
\end{align*}
 where we have used Step 1 and the fact that $$\sqrt{\ell_n}/\bar{\beta} = \ell_{n}^{-1/2} \psi^{-1} \leq n^{-1/8} \ell_{n}^{-1/8} M_{3}^{3/4} \leq n^{-1/2} u\ell_n^{3/2} \leq 1.$$ The claimed bound on $\rho$ now follows.

\textbf{(b)}. Bounding $\rho$ when $u<u^\star$. Since $\psi$ is smaller than in case (a), by the calculations in Step (a)
\begin{equation*}
n^{-1/2} (\psi^3+\psi^2\beta+\psi\beta^2)M_3^3/\sqrt{n}\lesssim n^{-1/2} u\ell_n^{3/2}.
\end{equation*}
Moreover, using definition of $\psi$, $u>u_2$, definition of $u_2$,
 we have
\begin{align*}
&\psi\beta \bar \varphi(u) \leq \beta \bar \varphi(u)^{2/3}\ell_n^{-1/6} \leq \beta \bar \varphi(u_2)^{2/3}\ell_n^{-1/6}
\leq  n^{-1} \beta u_2^2\ell_n^{5/3-1/6}
\lesssim n^{-1/2} u \ell_n^{3/2}; \\
&\psi^2 \bar \varphi(u)
\leq \bar \varphi(u)^{1/3}\ell_n^{-1/3}  \leq \bar \varphi(u_2)^{1/3}\ell_n^{-1/3}
\leq n^{-1/2} u_2\sqrt{\ell_n} \leq n^{-1/2} u\ell_n^{3/2}.
\end{align*}
Analogously and using $n^{-1/2} u\ell_n^{3/2} \leq 1$, we have
\begin{align*}
&\psi \bar \varphi(u)\sqrt{\log(p/\gamma)} \leq \bar \varphi(u)^{2/3}\ell_n^{1/3}  \leq \bar \varphi(u_2)^{2/3}\ell_n^{1/3} \leq n^{-1} u_{2}^{2} \ell_{n}^{2}   \leq n^{-1/2} u\ell_n^{3/2}. \\
&\psi^{-1} \sqrt{\ell_n} = \bar \varphi(u)^{1/3}\ell_n^{2/3} \leq n^{-1/2}  u\ell_n^{3/2}.
\end{align*}
This completes the proof.
\qed

\section{Deferred Proofs for Section \ref{sec: Dantzig}}
\label{sec: proof for Dantzig}

\subsection{Proof of Theorem \ref{thm: dantzig estimator}}  The proof proceeds in three steps.
In the proof $(\hat \beta, \lambda)$ denotes  $(\hat \beta^{(k)}, \lambda^{(k)})$
with $k$ either $0$ or $1$.

\textbf{Step 1.} Here we show that there exist some constants $c>0$ and $C > 0$ (depending only $c_{1},C_{1}$ and $\sigma^{2}$) such that
\begin{equation}\label{eq: d constraint}
\Pr( T_0 \leq \lambda) \geq 1- \alpha - \nu_n,
\end{equation}
with $\nu_{n} = C n^{-c}$.
We first note that
$T_0 = \sqrt{n} \max_{1 \leq k \leq 2p}\En [ \tilde z_{ik} \eps_{i} ]$,
where $\tilde z_{i} = (z_{i}', -z_{i}')'$.  Application of Corollary \ref{cor: central limit theorem}-(ii) gives
\begin{equation*}
|\Pr ( T_0 \leq \lambda )- \Pr ( Z_0 \leq \lambda )| \leq C  n^{-c},
\end{equation*}
where $c>0$ and $C>0$ are constants depending only on $c_{1},C_{1}$ and $\sigma^{2}$.
The claim follows since $\lambda \geq c_{Z_0}(1-\alpha)$, which holds because  $\lambda^{(1)} = c_{Z_0}(1-\alpha)$,
and $\lambda^{(1)} \leq \lambda^{(0)} = c_0(1-\alpha) := \sigma \Phi^{-1}(1- \alpha/(2p))$ 
(by the union bound
$\Pr ( Z_0 \geq c_0(1-\alpha) )  \leq 2p \Pr ( \sigma N(0,1)\geq c_0(1-\alpha) ) = \alpha$).

\textbf{Step 2.} We claim that with probability $\geq 1- \alpha - \nu_n$, $\hat \delta = \hat \beta - \beta$ obeys:
\begin{equation*}
\sqrt{n} \max_{1 \leq j \leq p} | \En[z_{ij}(z_{i}'\hat \delta)] | \leq 2 \lambda.
\end{equation*}
Indeed, by definition of $\hat \beta$,
$
\sqrt{n} \max_{1 \leq j \leq p} | \En [ z_{ij}(y_{i} - z_{i}' \hat \beta)] | \leq \lambda,
$
which by the triangle inequality  implies
$
\sqrt{n} \max_{1 \leq j \leq p} | \En[z_{ij}(z_{i}'\hat \delta)] | \leq  T_0 + \lambda.
$
The claim follows from Step 1.

\textbf{Step 3.}   By
Step 1,  with probability $\geq 1- \alpha - \nu_n$, the true value $\beta$ obeys the constraint in optimization problem (\ref{eq: dantzig estimator}) in the main text, in which case by definition of $\hat \beta$, $\| \hat \beta \|_{\ell_1} \leq \| \beta\|_{\ell_1}$. Therefore, with the same probability, $\hat \delta  \in \mathcal{R}(\beta)= \{ \delta \in \RR^d: \| \beta + \delta \|_{\ell_1} \leq \| \beta\|_{\ell_1}\}.$  By definition of $\kappa_I(\beta)$ we have that with the same probability,
$$
\kappa_I(\beta) \| \hat \delta\|_I \leq  \max_{1 \leq j \leq p} | \En[z_{ij}(z_{i}'\hat \delta)] |.
$$
Combining this inequality with Step 2 gives the claim of the theorem. \qed

\subsection{Proof of Theorem \ref{thm: dantzig estimator 2}}  The proof has four steps. In the proof,
we let $\varrho_{n} = C n^{-c}$ for sufficiently small $c > 0$ and sufficiently large $C>0$ depending only on $c_{1},C_{1},\underline{\sigma}^{2},\sigma^{2}$, where $c$ and $C$ (and hence $\varrho_{n}$) may change from place to place.

\textbf{Step 0.} The same argument as in the previous proof applies to $\hat \beta^{(0)}$ with $\lambda = \lambda^{(0)}:= c_0(1-1/n)$,
where now $\sigma^{2}$ is the upper bound on $\Ep[\eps_{i}^2]$. Thus, we conclude that with probability  at least $1- \varrho_n$,
\begin{equation*}
\| \hat  \beta^{(0)} - \beta\|_{\pr} \leq  \frac{2 c_0(1-1/n)}{\sqrt{n} \kappa_{\pr}(\beta)}.
\end{equation*}

\textbf{Step 1.}
We claim that with probability at least $1- \varrho_n$,
\begin{equation*}
\max_{1 \leq j \leq p} \(\En[z_{ij}^2(\hat{\eps}_{i}-\eps_{i})^2]\)^{1/2} \leq  B_n  \frac{2 c_0(1-1/n)}{\sqrt{n} \kappa_{\pr}(\beta)}=: \iota_n.
\end{equation*}
Application of H\"{o}lder's inequality and identity $\eps_{i} - \hat{\eps}_{i}=  z_{i}'(\hat \beta^{(0)} - \beta)$
 gives
\begin{equation*}
\max_{1 \leq j \leq p} \(\En[z_{ij}^2(\hat{\eps}_{i}-\eps_{i})^2]\)^{1/2}
\leq  B_n (\En[z_{i}' (\hat  \beta^{(0)} - \beta)]^2)^{1/2} \leq B_n  \| \hat   \beta^{(0)} - \beta\|_{\pr}.
\end{equation*}
The claim follows from Step 0.

\textbf{Step 2.} In this step, we apply Corollary \ref{cor: multiplier bootstrap examples}-(ii) to
\begin{align*}
&T =T_0 =\sqrt{n}\max_{1\leq j \leq 2p} \En[\tilde z_{ij}\eps_{i}], \ W =\sqrt{n}\max_{1\leq j \leq 2p} \En[\tilde z_{ij}\hat{\eps}_{i}e_{i}], \ \text{and} \\
&W_0 =\sqrt{n}\max_{1\leq j \leq 2p} \En[\tilde z_{ij}\eps_{i}e_{i}],
\end{align*}
where $\tilde z_{i} = (z_{i}',-z_{i}')'$,  to conclude that uniformly in $\alpha \in (0,1)$
\begin{equation}\label{eq: given}
\Pr ( T_0 \leq c_W(1-\alpha) ) \geq 1- \alpha - \varrho_n.
\end{equation}
 To show applicability of Corollary \ref{cor: multiplier bootstrap examples}-(ii), we note that for any $\zeta_1>0$,
{\small \begin{align*}
 \Pr_e(|W-W_0|>\zeta_1) & \leq \Ep_e[|W-W_0|]/\zeta_1 \leq \sqrt{n}\Ep_e\left[\max_{1\leq j  \leq p}|\En[z_{ij}(\hat{\eps}_{i}-\eps_{i})e_{i}]|\right]/\zeta_1 \\ & \lesssim \sqrt{\log p} \max_{1\leq j \leq p}(\En[z_{ij}^2(\hat{\eps}_{i}-\eps_{i})^2])^{1/2}/\zeta_1,
\end{align*}}
\!where the third inequality is due to Pisier's inequality.
The last quantity is bounded by $ (\iota_n^2 \log p)^{1/2}/\zeta_1$  with probability $\geq 1-\varrho_n$ by Step 1.

Since $\iota_n\log p\leq C_1n^{-c_{1}}$ by  assumption (vi) of the theorem, we can take $\zeta_1$ in such a way that
$\zeta_1(\log p)^{1/2}\leq \varrho_n$ and $(\iota_n^2\log p)^{1/2}/\zeta_1\leq \varrho_n$.
Then all the conditions of Corollary \ref{cor: multiplier bootstrap examples}-(ii) with so defined $\zeta_1$ and $\zeta_2=\varrho_n\vee((\iota_n^2\log p)^{1/2}/\zeta_1)$ are satisfied,
and hence application of the corollary gives that uniformly in $\alpha \in (0,1)$,
\begin{equation}\label{equation meta}
| \Pr ( T_0 \leq c_W(1-\alpha) ) - 1-\alpha | \leq \varrho_n,
\end{equation}
which implies the claim of this step.

\textbf{Step 3.} In this step we claim that with probability at least $1- \varrho_n$,
$$
c_W( 1-\alpha ) \leq c_{Z_0}(1-\alpha + 2\varrho_n).
$$
Combining Step 2 and Lemma \ref{lem: quantile approximated to exact} gives that with probability at least $1-\zeta_2$, $c_W(1-\alpha)\leq c_{W_0}(1-\alpha+\zeta_2)+\zeta_1$, where $\zeta_1$ and $\zeta_2$ are chosen as in Step 2. In addition, Lemma \ref{lem: quantile conditional to unconditional} shows that $c_{W_0}(1-\alpha+\zeta_2)\leq c_{Z_0}(1-\alpha+\varrho_n)$. Finally, Lemma \ref{lem: anticoncentration} yields $c_{Z_0}(1-\alpha+\varrho_n)+\zeta_1\leq c_{Z_0}(1-\alpha+2\varrho_n)$. Combining these bounds gives the claim of this step.

\textbf{Step 4.} Given (\ref{eq: given}), the rest of the proof is identical to
Steps 2-3 in the proof of Theorem \ref{thm: dantzig estimator}
with $\lambda = c_W(1-\alpha)$. The result follows for $\nu_n = 2\varrho_n$.\qed

\section{Deferred Proofs for Section \ref{sub: MHT}}
\subsection{Proof of Theorem \ref{thm: MHT}}
The multiplier bootstrap critical values $c_{1-\alpha,w}$ clearly satisfy $c_{1-\alpha,w}\leq c_{1-\alpha,w^\prime}$ whenever $w\subset w^\prime$,
so  inequality (\ref{eq: critical value property 1}) in the main text is satisfied. Therefore, it suffices to prove (\ref{eq: conditions MHT}) in the main text.

Let $w$ denote the set of true null hypotheses. Then for any $j\in w$,
$$
t_j=\sqrt{n}(\hat{\beta}_j-\beta_{0j})\leq \frac{1}{\sqrt{n}}\sum_{i=1}^n x_{ij}+r_{nj}.
$$
Therefore, for all $j\in w$, we can and will assume that $\beta_j=\beta_{0j}$.

For $w\subset\mathcal{W}=\{1,\dots,p\}$, define
\begin{equation*}
T:=T(w):=\max_{j\in w}\sqrt{n}(\hat{\beta}_j-\beta_{0j}), \ W:=W(w):=\max_{j\in w}\frac{1}{\sqrt{n}}\sum_{i=1}^{n}\hat{x}_{ij}e_i.
\end{equation*}
In addition, define
\begin{equation*}
T_0:=T_0(w):=\max_{j\in w}\frac{1}{\sqrt{n}}\sum_{i=1}^{n}x_{ij}, \ W_0:=W_0(w):=\max_{j\in w}\frac{1}{\sqrt{n}}\sum_{i=1}^{n}x_{ij}e_i.
\end{equation*}
To prove (\ref{eq: conditions MHT}), we will apply Corollary \ref{cor: multiplier bootstrap examples}.
By assumption, either (i) or (ii) of Corollary \ref{cor: multiplier bootstrap examples} holds. Therefore, it remains to verify conditions in equations (\ref{eq: statistic approximation}) and (\ref{eq: conditional quantiles}) in the main text with $\zeta_1\sqrt{\log p}+\zeta_2\leq Cn^{-c}$ for some $c>0$ and $C>0$ uniformly over all $w\subset\mathcal{W}$.

Set $\zeta_1=(C/2)n^{-c}/\sqrt{\log p}$ and $\zeta_2=(C/2)n^{-c}$ for sufficiently small $c>0$ and large $C>0$ depending on $c_1,C_1,c_2$, and $C_2$ only. Note that $\zeta_1\sqrt{\log p}+\zeta_2\leq Cn^{-c}$.
Also note that
$$
|T-T_0|\leq \max_{1\leq j\leq p}\left|r_{nj}\right|=\Delta_1
$$
for all $w\subset\mathcal{W}$.
Therefore, it follows from assumption (i) that $\Pr(|T-T_0|>\zeta_1)< \zeta_2$ for all $w\subset\mathcal{W}$,  i.e. condition in equation (\ref{eq: statistic approximation}) holds uniformly over all $w\subset\mathcal{W}$.
Further, note that $\sum_{i=1}^n(\hat{x}_{ij}-x_{ij})e_i/\sqrt{n}$ conditional on $(x_i)_{i=1}^n$ and $(\hat{x}_i)_{i=1}^n$ is distributed as $N(0,\En[(\hat{x}_{ij}-x_{ij})^2])$ random variable and
$$
|W-W_0|\leq \max_{1\leq j\leq p}\left|\frac{1}{\sqrt{n}}\sum_{i=1}^n(\hat{x}_{ij}-x_{ij})e_i\right|
$$
for all $w\subset\mathcal{W}$.
Therefore, $\Ep_e[|W-W_0|]\leq (C/2)\sqrt{\Delta_2\log p}$, and so it follows from Borell inequality and assumption (ii) that $\Pr(\Pr_e(|W-W_0|>\zeta_1)>\zeta_2)<\zeta_2$ for all $w\subset\mathcal{W}$, i.e. condition in equation (\ref{eq: conditional quantiles}) holds  uniformly over all $w\subset\mathcal{W}$. This completes the proof by applying Corollary \ref{cor: multiplier bootstrap examples}.
\qed

\section{Monte Carlo Experiments in Support of Section 4}\label{sec: monte carlo}
In this section, we present results of Monte Carlo simulations that illustrate our theoretical results on Dantzig selector given in Section \ref{sec: Dantzig}. We consider Gaussian and Non-Gaussian noise with homoscedasticity and heteroscedasticity. We study 3 types of Dantzig selector depending on the choice of the penalty level: canonical, ideal (based on gaussian approximation, GAR), and multiplier bootstrap (MB).

We consider the following regression model:
\begin{equation*}
y_{i}=z_{i}^{\prime}\beta+\eps_{i},
\end{equation*}
where observations are independent across $i$, $y_i$ is a scalar dependent variable, $z_i$ is a $p$-dimensional vector of covariates, and $\varepsilon_i$ is noise. The first component of $z_i$ equals 1 in all experiments (an intercept). Other $p-1$ components are simulated as follows: first, we simulate a vector $w_i\in\RR^{p-1}$ from the Gaussian distribution with zero mean so that $\Ep[w_{ij}^2]=1$ for all $1\leq j\leq p-1$ and $\Ep[w_{ij}w_{ik}]=\rho$ for all $1\leq j,k\leq p-1$ with $j\neq k$; second, we set $z_{ij+1}=w_{ij}/(\En[w_{ij}^2])^{1/2}$ (equicorrelated design). Depending on the experiment, we set $\rho=0$, $0.5$, $0.9$, or $0.99$. We simulate $\epsilon_i=\sigma_0\sigma(z_i)e_i$ where depending on the experiment, $\sigma_0=0.5$ or $1.0$ and $e_i$ is taken either from $N(0,1)$ distribution (Gaussian noise) or from t-distribution with 5 degrees of freedom normalized to have variance 1 (Non-Gaussian noise). To investigate the effect of heteroscedasticity on the properties of different estimators, we set $$
\sigma(z_i)=\frac{2\exp(\gamma z_{i2})}{1+\exp(\gamma z_{i2})}
$$
where $\gamma$ is either 0 (homoscedastic case) or 1 (heteroscedastic case).

\begin{table}
\caption{Results of Monte Carlo experiments for prediction error. Non-Gaussian noise.}

\begin{tabular}{cccccc}
\hline
\multirow{2}{*}{Distribution $\varepsilon$} & \multirow{2}{*}{~~~~$\sigma_0$~~~~} & \multirow{2}{*}{~~~~$\rho$~~~~} & \multicolumn{3}{c}{Method}\tabularnewline
\cline{4-6}
 &  &  & Canonical & ~~~GAR~~~ & ~~~MB~~~\tabularnewline
\hline
\multirow{8}{*}{Homoscedastic} & \multirow{4}{*}{0.5} & 0.00 & 0.224 & 0.207 & 0.208\tabularnewline
 &  & 0.50 & 0.390 & 0.353 & 0.379\tabularnewline
 &  & 0.90 & 0.352 & 0.317 & 0.340\tabularnewline
 &  & 0.99 & 0.107 & 0.057 & 0.058\tabularnewline
\cline{2-6}
 & \multirow{4}{*}{1.0} & 0.00 & 0.648 & 0.604 & 0.674\tabularnewline
 &  & 0.50 & 0.695 & 0.643 & 0.644\tabularnewline
 &  & 0.90 & 0.538 & 0.406 & 0.412\tabularnewline
 &  & 0.99 & 0.348 & 0.139 & 0.137\tabularnewline
\hline
\multirow{8}{*}{Heteroscedastic} & \multirow{4}{*}{0.5} & 0.00 & 0.656 & 0.252 & 0.393\tabularnewline
 &  & 0.50 & 0.660 & 0.407 & 0.469\tabularnewline
 &  & 0.90 & 0.516 & 0.326 & 0.342\tabularnewline
 &  & 0.99 & 0.339 & 0.066 & 0.064\tabularnewline
\cline{2-6}
 & \multirow{4}{*}{1.0} & 0.00 & 1.588 & 0.661 & 0.909\tabularnewline
 &  & 0.50 & 1.590 & 0.722 & 0.755\tabularnewline
 &  & 0.90 & 1.336 & 0.445 & 0.454\tabularnewline
 &  & 0.99 & 1.219 & 0.153 & 0.156\tabularnewline
\hline
\end{tabular}
\end{table}

\begin{table}
\caption{Results of Monte Carlo experiments for prediction error. Gaussian noise.}

\begin{tabular}{cccccc}
\hline
\multirow{2}{*}{Distribution $\varepsilon$} & \multirow{2}{*}{~~~~$\sigma_0$~~~~} & \multirow{2}{*}{~~~~$\rho$~~~~} & \multicolumn{3}{c}{Method}\tabularnewline
\cline{4-6}
 &  &  & Canonical & ~~~GAR~~~ & ~~~MB~~~\tabularnewline
\hline
\multirow{8}{*}{Homoscedastic} & \multirow{4}{*}{0.5} & 0.00 & 0.229 & 0.211 & 0.210\tabularnewline
 &  & 0.50 & 0.421 & 0.386 & 0.417\tabularnewline
 &  & 0.90 & 0.365 & 0.315 & 0.350\tabularnewline
 &  & 0.99 & 0.109 & 0.059 & 0.059\tabularnewline
\cline{2-6}
 & \multirow{4}{*}{1.0} & 0.00 & 0.663 & 0.618 & 0.671\tabularnewline
 &  & 0.50 & 0.679 & 0.627 & 0.624\tabularnewline
 &  & 0.90 & 0.554 & 0.424 & 0.429\tabularnewline
 &  & 0.99 & 0.326 & 0.127 & 0.120\tabularnewline
\hline
\multirow{8}{*}{Heteroscedastic} & \multirow{4}{*}{0.5} & 0.00 & 0.674 & 0.257 & 0.395\tabularnewline
 &  & 0.50 & 0.643 & 0.412 & 0.451\tabularnewline
 &  & 0.90 & 0.503 & 0.310 & 0.324\tabularnewline
 &  & 0.99 & 0.319 & 0.060 & 0.059\tabularnewline
\cline{2-6}
 & \multirow{4}{*}{1.0} & 0.00 & 1.690 & 0.708 & 0.976\tabularnewline
 &  & 0.50 & 1.537 & 0.665 & 0.679\tabularnewline
 &  & 0.90 & 1.334 & 0.439 & 0.452\tabularnewline
 &  & 0.99 & 1.189 & 0.155 & 0.148\tabularnewline
\hline
\end{tabular}
\end{table}

Tables 1 and 2 present results on prediction error of Dantzig selector for the case of Non-Gaussian and Gaussian noise, respectively. Prediction error is defined as
$$
\Vert \hat{\beta}-\beta\Vert_{pr}=\sqrt{\En[z_i^\prime(\hat{\beta}-\beta)]}
$$
where $\hat{\beta}$ is the Dantzig selector; see Section \ref{sec: Dantzig} for the definition of the Dantzig selector. Recall that implementing the Dantzig selector requires selecting the penalty level $\lambda$. Both tables show results for 3 different choices of the penalty level. Canonical penalty is $\lambda=\bar{\sigma}\Phi^{-1}(1-\alpha/(2p))$ where $\bar{\sigma}=\sigma_0(1+I\{|\gamma|>0\})$, the upper bound on the variance of $\varepsilon_i$'s. Ideal (based on gaussian approximation, GAR) penalty is $\lambda=c_{Z_0}(1-\alpha)$, the conditional $(1-\alpha)$ quantile of $Z_0$ given $(z_i)_{i=1}^n$ where
$$
Z_0=\sqrt{n}\max_{1\leq j\leq p}|\En[z_{ij}\sigma_0\sigma(z_i)e_i]|
$$
where $e_i\sim N(0,1)$ independently across $i$.
Finally, multiplier bootstrap (MB) penalty is defined as follows. First, we calculate the Dantzig selector with the canonical choice of the penalty level, $\hat{\beta}$, and select regressors corresponding to non-zero components of $\hat{\beta}$. Second, we run the OLS regression of $y_i$ on the set of selected regressors, and take residuals from this regression, $(\hat{\epsilon}_i)_{i=1}^n$. Then the multiplier bootstrap penalty level is $\lambda=c_{W}(1-\alpha)$, the conditional $(1-\alpha)$ quantile of $W$ given $(z_{i},\hat{\varepsilon}_i)_{i=1}^n$ where
$$
W=\sqrt{n}\max_{1\leq j\leq p}|\En[z_{ij}\hat{\varepsilon}_ie_i]|
$$
where $e_i\sim N(0,1)$ independently across $i$.

The results show that the GAR penalty always yields smaller prediction error than that of the canonical penalty. Moreover, as predicted by the theory, GAR penalty works especially good in comparison with the canonical penalty in heteroscedastic case and/or in the case with high correlation between regressors (high $\rho$). In addition, in most cases, the results for the MB penalty are similar to those for the GAR penalty. In particular, the MB penalty in most cases yields smaller prediction error than that of the canonical penalty. Finally, the GAR penalty in most cases is slightly better than the MB penalty. Note, however, that when heteroscedasticity function $\sigma(z_i)$ is unknown, the GAR penalty becomes infeasible but the MB penalty is feasible given that the upper bound on the variance of $\varepsilon_i$'s exists.

\newpage

\begin{center}
{\Large  Supplementary Material II} \\
\text{} \\
{\textbf{Additional Results and Discussions}} \\
\end{center}


\section{A note on Slepian-Stein type methods for normal approximations}\label{sec: Slepian-Stein note}
To keep the notation simple, consider a random vector $X$ in $\RR^{p}$ and a standard normal vector
$Y$ in $\RR^{p}$.  We are interested in bounding
\begin{equation*}
\Ep[g(X)] - \Ep[g(Y)],
\end{equation*}
over some collection of test functions $g \in \mathcal{G}$.  Without loss of generality, suppose that $Y$ and $X$ are independent.

Consider Stein's partial differential equation:
\begin{equation*}
g(x) -\Ep[g(Y)]  = \triangle h(x) -  x'\nabla h(x)
\end{equation*}
where $\triangle h(X)$ and $\nabla h(X)$ refer to the Laplacian and the gradient of $h(X)$.
It is well known, e.g. \cite{GoldsteinRinott1996} and \cite{ChatterjeeMeckes2008}, that an explicit solution for $h$ in this equation is given by
\begin{equation*}
h(x) := -\int_0^1 \frac{1}{2t} \left[\Ep[g( \sqrt{t} x + \sqrt{1-t} Y )] - \Ep[g(Y)] \right ] dt,
\end{equation*}
so that
\begin{equation*}
\Ep[g(X)] - \Ep[g(Y)] = \Ep[\triangle h(X) -  X ' \nabla h(X)].
\end{equation*}
The Stein type method for normal approximation bounds the right side for $g \in \mathcal{G}$.

Next, let us consider the Slepian smart path interpolation:
\begin{equation*}
Z(t) = \sqrt{t} X + \sqrt{1-t} Y.
\end{equation*}
Then we have
\begin{equation*}
\Ep[g(X)] - \Ep[g(Y)] = \Ep \left [\int_0^1 \frac{1}{2}  \nabla g(Z(t))' \left (\frac{X}{\sqrt{t}} - \frac{Y}{\sqrt{1-t}}\right) \right] dt.
\end{equation*}
The Slepian type method, as used in our paper, bounds the right side for $g \in \mathcal{G}$.  We also
refer the reader to \cite{Rollin2011} for a related discussion and interesting results (see in particular Lemma 2.1 in  \cite{Rollin2011}).

Elementary calculations and integration by parts yield the following observation.
\begin{lemma}\label{compare Slepian vs Stein} Suppose that $g : \RR^p \to \RR$ is a $C^2$-function with
uniformly bounded derivatives up to order two. Then
\begin{align*}
I &:= \Ep \left [\int_0^1 \frac{1}{2}  \nabla g(Z(t))' \left (\frac{X}{\sqrt{t}}\right) \right] =  -\Ep[X' \nabla h(X)]
\intertext{and}
II &:= \Ep \left [\int_0^1 \frac{1}{2} \nabla g(Z(t))' \left ( \frac{Y}{\sqrt{1-t}}\right) \right] = -\Ep[\triangle h(X)].
\end{align*}
\end{lemma}
Hence the Slepian and Stein methods both show that difference between $I$ and $II$ is small or approaches zero under suitable
conditions on $X$; therefore, they are very similar in spirit, if not identical.   The details of treating terms may be different from application
to application; see more on this in \cite{Rollin2011}.

\begin{proof}[Proof of Lemma \ref{compare Slepian vs Stein}]  By definition of $h$, we have
\begin{eqnarray*}
-\Ep[X'\nabla h(X)]  =   \Ep\left [ X' \int_0^1 \frac{1}{2t}  \nabla g (Z(t))  \sqrt{t} dt\right] =  \Ep\left [ \int_0^1  \nabla g(Z(t))' \frac{X}{2 \sqrt{t}}  dt \right].
\end{eqnarray*}
On the other hand,  by definition of $h$ and Stein's identity (Lemma \ref{lemma: talagrand}),
\begin{eqnarray*}
-\Ep[\triangle h(X)] =  \Ep\left[ \frac{1}{2} \int_0^1 \triangle g ( Z(t))    dt \right]  =   \Ep \left [ \frac{1}{2} \int_0^1 \nabla g(Z(t))' \left ( \frac{Y}{\sqrt{1-t}}\right) dt \right].
\end{eqnarray*}
This completes the proof.
\end{proof}


\begin{lemma}[Stein's identity]
\label{lemma: talagrand}
Let $W =(W_{1},\dots,W_{p})^{T}$ be a centered Gaussian random vector in $\RR^{p}$.
Let $f: \RR^{p} \to \RR$ be a $C^1$-function such that $\Ep [ |\partial_{j} f(W) | ] < \infty$ for all $1 \leq j \leq p$. Then for every $1 \leq j \leq p$,
\begin{equation*}
\Ep[W_{j}f(W)]=\sum_{k=1}^p\Ep[W_{j} W_{k}] \Ep [\partial_k f (W)].
\end{equation*}
\end{lemma}
\begin{proof}[Proof of Lemma \ref{lemma: talagrand}]
See Section A.6 of \cite{Talagrand2003}, and also \cite{Stein1981}.
\end{proof}

\section{A Simple Gaussian Approximation Result}\label{sec: elementary GAR}
This section can be helpful to the reader wishing to see how Slepian-Stein
methods can be used to prove a simple Gaussian approximation (whose
applicability is limited however.)  We start with the following elementary lemma.
\begin{lemma}[A Simple Comparison of Gaussian to Non-Gaussian Maxima] \label{thm: warmup comparison}
For every $g \in C_{b}^{3}(\RR)$ and $\beta>0$,
\begin{align*}
&|\Ep[g(F_\beta(X))-g(F_\beta(Y))]| \lesssim n^{-1/2} (G_3+G_2 \beta + G_1\beta^2 ) \barEp [S^3_i], \\
\intertext{and hence}
& |\Ep[g(T_{0})- g(Z_{0})]| \lesssim n^{-1/2}(G_3+G_2 \beta + G_1\beta^2) \barEp [S^3_i]
 + \beta^{-1}G_1\log p.
\end{align*}
\end{lemma}

The optimal value
of the last bound is given by taking the minimum over $\beta$. We postpone choices of $\beta$ to the proof of the subsequent corollary, leaving ourselves more flexibility in optimizing bounds in the corollary.

\begin{remark} The bound above per se seems new, though it is merely a simple extension of results in \cite{Chatterjee2005a},  who obtained the bound for the case with $X$ having a special structure like in our example (E.4), related to spin glasses, using classical Lindeberg's method.  We give a  proof using a variant of Slepian-Stein method, since this is the tool we end up using to prove our main results, as the Lindeberg's method, in its pure form, did not yield the same sharp results.  Our proof is related but rather different in details from the more abstract/general arguments based on Stein triplets given in \cite{Rollin2011} (Lemma 2.1), but given for the special case of data $(x_i)_{i=1}^n$ with coordinates $x_i$'s that $\Bbb{R}$-valued, in contrast to the $\Bbb{R}^p$-valued case treated here.   \cite{Rollin2011}  re-analyzed \cite{Chatterjee2005a}'s setup under local dependence and gave a number of other interesting applications. \qed \end{remark}

The next result states a bound on the Kolmogorov distance between distributions of $T_{0}$ and $Z_{0}$. The result follows from Lemma \ref{thm: warmup comparison} and the anti-concentration inequality for maxima of Gaussian random variables stated in  Lemma \ref{lem: anticoncentration}.  Note that this result was not included in either \cite{Chatterjee2005a} or  \cite{Rollin2011} for the cases that they have analyzed.


\begin{corollary}[\textbf{A Simple Gaussian Approximation}]\label{cor: Gaussian to nonGaussian KS 1}
Suppose that there are some constants $c_{1} > 0$ and  $C_{1} > 0$ such that $c_{1}  \leq \barEp[x_{ij}^2] \leq C_{1}$ for all $1\leq j \leq p$. Then there exists a constant $C>0$ depending only on $c_{1}$ and $C_{1}$ such that
\begin{align*}
\rho &:= \sup_{t\in\RR}\left|\Pr ( T_{0} \leq t ) - \Pr (Z_{0} \leq t) \right|   \leq C (n^{-1} (\log (pn))^7)^{1/8} (\barEp [S^3_i])^{1/4}.
 \end{align*}
\end{corollary}

Theorem \ref{thm: warmup comparison} and Corollary \ref{cor: Gaussian to nonGaussian KS 1} imply that the error of approximating the maximum coordinate in the sum of independent random vectors by its Gaussian analogue
depends on $p$ (possibly) only through $\log p$. This is the main qualitative feature of all the results in this paper.
Both Lemma \ref{thm: warmup comparison} and Corollary \ref{cor: Gaussian to nonGaussian KS 1} and all the results in this paper do not limit the dependence among the coordinates in $x_{i}$.

While Lemma \ref{thm: warmup comparison} and Corollary \ref{cor: Gaussian to nonGaussian KS 1} convey an important qualitative aspect of
the problem and admit easy-to-grasp proofs, an important disadvantage of these results is that the bounds depend on $\bar\Ep[S_{i}^3]$.
When $\bar\Ep[S_{i}^3]$ increases with $n$,  for example when $|x_{ij}|\leq B_n$ for all $i$ and $j$ and $B_n$ grows with $n$,  the simple bound
above may be too poor, and can be improved considerably using several inputs.
We derive in Theorem \ref{theorem:comparison non-Gaussian} in the main text a bound that can be much better in the latter scenario.  The improvement there comes at a cost of  more involved statements and proofs.

\begin{proof}[Proof of Lemma \ref{thm: warmup comparison}]
Without loss of generality, we are assuming that sequences $(x_{i})_{i=1}^n$ and $(y_{i})_{i=1}^n$ are independent. For $t\in[0,1]$, we consider the Slepian interpolation between $Y$ and $X$:
\begin{equation*}
Z(t):=\sqrt{t}X+\sqrt{1-t}Y=\sum_{i=1}^{n}Z_{i}(t)\text{, }Z_{i}(t):=\frac{1}{\sqrt{n}}(\sqrt{t}x_{i}+\sqrt{1-t}y_{i}).
\end{equation*}
We shall also employ Stein's leave-one-out expansions:
\begin{equation*}
Z^{(i)}(t):=Z(t)-Z_{i}(t).
\end{equation*}
Let $\Psi(t)=\Ep[m(Z(t))]$ for $m:=g\circ F_\beta$. Then by Taylor's theorem,
\begin{align*}
\Ep[m(X)-m(Y)] &= \Psi (1) - \Psi (0) =\int_0^1 \Psi'(t) dt \\
&=\frac{1}{2} \sum_{j=1}^{p}\sum_{i=1}^{n}\int_0^1\Ep[\partial_{j}m(Z(t)) \dot Z_{ij}(t)]dt=\frac{1}{2} (I + II + III),
\end{align*}
where
{\small \begin{align*}
\dot Z_{ij}(t) &= \frac{d}{dt} Z_{ij}(t) = \frac{1}{\sqrt{n}} \( \frac{1}{\sqrt{t}}  x_{ij} - \frac{1}{\sqrt{1-t}}  y_{ij} \),  \ \text{and} \\
I &=  \sum_{j=1}^{p} \sum_{i=1}^{n} \int_0^1 \Ep [  \partial_{j} m(Z^{(i)}(t) )  \dot Z_{ij}(t) ] dt, \\
II &= \sum_{j,k=1}^{p} \sum_{i=1}^{n} \int_0^1 \Ep [ \partial_{j} \partial_k  m(Z^{(i)}(t) )   \dot Z_{ij}(t) Z_{ik}(t) ] dt, \\
III &= \sum_{j,k,l=1}^{p} \sum_{i=1}^{n} \int_0^1 \int_0^1 (1-\tau) \Ep [ \partial_{j} \partial_k \partial_l m( Z^{(i)}(t) + \tau Z_{i}(t)) \dot Z_{ij}(t) Z_{ik}(t) Z_{il}(t) ] d \tau dt.
\end{align*}}
\! Note that random vector $Z^{(i)}(t)$ is independent of $(\dot Z_{ij}(t), Z_{ij}(t))$, and $\Ep [\dot Z_{ij}(t)] =0$. Hence we have $I=0$; moreover, since $\Ep[\dot Z_{ij}(t)Z_{ik}(t)]=n^{-1}\Ep[x_{ij}x_{ik}-y_{ij}y_{ik}]=0$ by construction of
$(y_i)_{i=1}^n$, we also have $II=0$.
Consider the third term $III$. We have that
\begin{align*}
| III | &\lesssim_{(1)} (G_3+G_2\beta+G_1\beta^2)  n \int \barEp \left [ \max_{1 \leq j,k,l \leq p} |  \dot Z_{ij}(t) Z_{ik}(t) Z_{il}(t) |\right ] dt, \\
&\lesssim_{(2)} n^{-1/2} (G_3+G_2\beta+G_1\beta^2)  \barEp\left[\max_{1\leq j \leq p }  \(| x_{ij}| + | y_{ij}|\)^3\right],
\end{align*}
where (1) follows from $| \partial_{j} \partial_k \partial_l m( Z^{(i)}(t) + \tau Z_{i}(t))| \leq U_{jkl}( Z^{(i)}(t) + \tau Z_{i}(t))  \lesssim
(G_{3}+ G_{2}\beta + G_{1}\beta^{2})$ holding by Lemma \ref{lemma: bounds on derivatives of m}, and (2) is shown below.
The first claim of the theorem now follows. The second claim follows directly from property (\ref{eq: smooth max property}) in the main text of the smooth max function.

It remains to show  (2). Define $\omega(t)=1/(\sqrt{t} \wedge \sqrt{1-t})$ and note,
{\small \begin{align*}
&\int_0^1   n \barEp \left [\max_{1\leq j,k,l \leq p }  |\dot Z_{ij}(t) Z_{ik}(t) Z_{il}(t)|\right ]dt   \\
&= \int_0^1 \omega(t) n \barEp \left [  \max_{1\leq j,k,l \leq p }  |\dot Z_{ij}(t)/\omega(t))  Z_{ik}(t) Z_{il}(t)| \right ]dt \\
&\leq n \int_0^1 \omega(t) \(\barEp[  \max_{1\leq j \leq p } | \dot Z_{ij}(t)/\omega(t)|^3] \barEp[  \max_{1\leq j \leq p }| Z_{ij}(t) |^3] \barEp[ \max_{1\leq j \leq p } | Z_{ij}(t) |^3] \) ^{1/3} dt \\
&\leq n^{-1/2} \left \{ \int_0^1 \omega(t) dt \right \} \barEp\left [ \max_{1\leq j \leq p }  \(| x_{ij}| + | y_{ij}|\)^3\right ] 
\end{align*}}
where the first inequality follows from H\"{o}lder's inequality, and the second from the fact that
$|\dot Z_{ij}(t)/\omega(t)| \leq (|  x_{ij}| + |  y_{ij}|)/\sqrt{n} $,
$|Z_{ij}(t)| \leq (|  x_{ij}| + |  y_{ij}|)/\sqrt{n}$.  Finally we note that $\int_0^1 \omega(t) dt \lesssim 1$, so inequality (2) follows.  This completes the overall proof.
\end{proof}

\begin{proof}[Proof of Corollary \ref{cor: Gaussian to nonGaussian KS 1}]
In this proof, let $C>0$ denote a generic constant depending only on $c_{1}$ and $C_{1}$, and its value may change from place to place.
For $\beta > 0$, define $e_{\beta} := \beta^{-1} \log p$. Recall that $S_i:= \max_{1\leq j\leq p}(|x_{ij}|+|y_{ij}|)$.
Consider and fix a $C_b^{3}(\mathbb{R})$-function $g_0:\RR \to [0,1]$ such that $g_0(s)=1$ for $s\leq 0$ and $g_0(s)=0$ for $s\geq 1$.
Fix any $t \in \RR$, and define $g(s)=g_0(\psi(s-t-e_\beta))$.  For this function $g$, $G_0=1, \ G_1\lesssim \psi, \ G_2 \lesssim \psi^2$ and $G_3 \lesssim \psi^3$.

Observe now that
\begin{align*}
\Pr ( T_{0}  \leq t )
&\leq \Pr ( F_{\beta} (X) \leq t + e_{\beta})
\leq \Ep[g(F_\beta(X))] \\
&\leq \Ep[g(F_\beta(Y))] + C(\psi^3+\beta\psi^2+\beta^2\psi)(n^{-1/2}\barEp[S_i^3]) \\
&\leq \Pr (F_\beta(Y) \leq t+e_\beta+\psi^{-1} ) + C(\psi^3+\beta\psi^2+\beta^2\psi)(n^{-1/2}\barEp[S_i^3]) \\
&\leq \Pr (Z_{0} \leq t+e_\beta+\psi^{-1} ) + C(\psi^3+\beta\psi^2+\beta^2\psi)(n^{-1/2}\barEp[S_i^3]).
\end{align*}
where the first inequality follows from (\ref{eq: smooth max property}), the second from construction of $g$,  the third from Theorem \ref{thm: warmup comparison}, and the fourth from construction of $g$, and the last from (\ref{eq: smooth max property}).
The remaining step is to compare $\Pr (Z_{0} \leq t+e_\beta+\psi^{-1} )$ with $\Pr (Z_{0} \leq t )$ and this is where Lemma \ref{lem: anticoncentration} plays its role.
By Lemma  \ref{lem: anticoncentration},
\begin{align*}
\Pr (Z_{0} \leq t+e_\beta+\psi^{-1} ) - \Pr (Z_{0} \leq t)   \leq C(e_{\beta}+\psi^{-1})\sqrt{1 \vee \log (p \psi)}.
\end{align*}
by which we have
{\small \begin{equation*}
\Pr ( T_{0}  \leq t ) - \Pr (Z_{0} \leq t)  \leq C [ (\psi^3+\beta\psi^2+\beta^2\psi)(n^{-1/2}\barEp[S_i^3]) + (e_{\beta}+\psi^{-1})\sqrt{1 \vee \log (p \psi)}  ].
\end{equation*}}
\!We have to minimize the right side with respect to $\beta$ and $\psi$. It is reasonable to choose $\beta$ in such a way that $e_{\beta}$ and $\psi^{-1}$ are balanced, i.e., $\beta = \psi \log p$. With this $\beta$, the bracket on the right side is bounded from above by
\begin{equation*}
C[\psi^{3} (\log p)^{2} (n^{-1/2}\barEp[S_i^3])  + \psi^{-1} \sqrt{1 \vee \log (p \psi)}],
\end{equation*}
which is approximately minimized by $\psi = (\log p)^{-3/8} (n^{-1/2}\barEp[S_i^3])^{-1/4}$. With this $\psi$, $\psi \leq  (n^{-1/2}\barEp[S_i^3])^{-1/4} \leq C n^{1/8}$ (recall that $p \geq 3$), and hence $\log (p \psi) \leq C\log (pn)$. Therefore,
\begin{equation*}
\Pr ( T_{0}  \leq t ) - \Pr (Z_{0} \leq t) \leq C (n^{-1/2}\barEp[S_i^3])^{1/4} (\log (pn))^{7/8}.
\end{equation*}
This gives one half of the claim. The other half follows similarly.
\end{proof}

\section{Gaussian Approximation and Multiplier Bootsrap Results,  allowing for Low Variances}\label{sec: low variance}
The purpose of this section is to provide results without an assumption that $\barEp[x_{ij}^2]>c$ for all $1\leq j\leq p$ and some constant $c>0$.
\subsection{Gaussian Approximation Results}
In this subsection, we use the same setup and notation as those in Section \ref{sec: Gaus vs NonGaus}. In particular, $x_1,\dots,x_n$ is a sequence of independent centered random vectors in $\RR^p$, $y_1,\dots,y_n$ is a sequence of independent centered Gaussian random vectors such that $\Ep[y_iy_i^\prime]=\Ep[x_ix_i^\prime]$, $T_0=\max_{1\leq j\leq p}X_j$ where $X=\sum_{i=1}^nx_i/\sqrt{n}$, $Z_0=\max_{1\leq j\leq p}Y_j$ where $Y=\sum_{i=1}^ny_i/\sqrt{n}$, and $$\rho=\sup_{t\in\RR}|\Pr(T_0\leq t)-\Pr(Z_0\leq t)|.$$
 In addition, denote
$$
M_{k,2}:=\max_{1\leq j\leq p}\frac{\barEp[|x_{ij}|^k]^{1/k}}{\barEp[x_{ij}^2]^{1/2}}\text{ and }\ell_n:=\log(pn/\gamma).
$$
We will impose the following condition:
\begin{itemize}
\item[(SM)] There exists $J\subset\{1,\dots,p\}$ such that $|J|\geq \nu p$ and for all $(j,k)\in J\times J$ with $j\neq k$, $\barEp[x_{ij}^2]\geq c_1$ and $|\barEp[x_{ij}x_{ik}]|\leq (1-\nu^\prime)(\barEp[x_{ij}^2]\barEp[x_{ik}^2])^{1/2}$ for some strictly positive constants $\nu,\nu^\prime$, and $c_1$ independent of $n$.
\end{itemize}

\begin{theorem}\label{thm: main_result1_extended}
Suppose that condition (SM) holds. In addition, suppose that there is some constant $C_1>0$  such that $\barEp[x_{ij}^2] \leq C_{1}$ for $1\leq j \leq p$. Then for every $\gamma \in (0,1)$,
\begin{equation*}
\rho \leq
C \left \{  n^{-1/8} (M_{3}^{3/4} \vee M_{4,2}^{1/2} ) \ell_n^{7/8} + n^{-1/2} \ell_n^{3/2} u(\gamma)+ p^{-c}\ell_n^{1/2} + \gamma \right \},
\end{equation*}
 where $c,C>0$ are constants that depend only on $\nu,\nu^\prime,c_1$ and $C_1$.
\end{theorem}
Theorem \ref{thm: main_result1_extended} has the following applications.
Let $C_{1} > 0$ be some constant that is independent of $n$, and  let $B_{n} \geq 1$ be a sequence of constants.
We allow for the case where $B_{n} \to \infty$ as $n \to \infty$.
We will assume that one of the following  conditions is satisfied   {\em uniformly in} $1 \leq i \leq n$ and $1 \leq j \leq p$:
\begin{itemize}
\item[(E.5)] $\barEp[x^2_{ij}] \leq C_{1}$ and $\displaystyle \max_{k =1,2}M_{k+2,2}^{k+2}/B_n^k + \Ep[\exp(|x_{ij}|/B_n)] \leq 2$;
\item[(E.6)] $\barEp[x^2_{ij}] \leq C_{1}$ and  $\displaystyle  \max_{k =1,2}M_{k+2,2}^{k+2}/B_n^k +\Ep[ (\max_{1\leq j \leq p} |x_{ij}| / B_{n})^4] \leq 2$.
\end{itemize}

\begin{corollary}\label{cor: central limit theorem extended}
Suppose that there exist constants $c_2>0$ and $C_2>0$ such that one of the following conditions is satisfied:
(i) (E.5) holds and $B_n^2 (\log (pn))^7/n\leq C_2 n^{-c_2}$ or
(ii) (E.6) holds and $B_n^4 (\log (pn))^7/n\leq C_2 n^{-c_2}$.
In addition, suppose that condition (SM) holds and $p\geq C_3n^{c_3}$ for some constants $c_3>0$ and $C_3>0$.
Then there exist constants $c > 0$ and $C>0$ depending only on $\nu,\nu^\prime,c_{1}, C_1, c_{2},C_2,c_3$, and $C_{3}$ such that $$\rho  \leq Cn^{-c}.$$
\end{corollary}

\subsection{Multiplier Bootstrap Results}
In this subsection, we use the same setup and notation as those in Section \ref{sec: multiplier bootstrap}. In particular, in addition to the notation used above, we assume that random variables $T$ and $W$ satisfy conditions (\ref{eq: statistic approximation}) and (\ref{eq: conditional quantiles}) in the main text, respectively, where $\zeta_1\geq 0$ and $\zeta_2\geq 0$ depend on $n$ and where $W_0$ appearing in condition (\ref{eq: conditional quantiles}) is defined in equation (\ref{average-multiplier}) in the main text. Recall that $\Delta=\max_{1\leq j\leq p}|\En[x_{ij}]-\barEp[x_{ij}]|$.
\begin{theorem}\label{thm: multiplier bootstrap}
Suppose that there is some constant $C_1>0$ such that $\bar{\sigma}:=\max_{1\leq j\leq p}\barEp[x_{ij}^2]\leq C_1$ for all $1\leq j\leq p$. In addition, suppose that condition (SM) holds. Moreover, suppose that conditions (\ref{eq: statistic approximation}) and (\ref{eq: conditional quantiles}) are satisfied.
Then for every $\vartheta>0$,
\begin{align*}
\rho_\ominus&:=\sup_{\alpha\in(0,1)}\Pr(\{T\leq c_{W}(\alpha)\}\ominus\{T_0\leq c_{Z_0}(\alpha)\})\\
&\leq  2(\rho +  \pi(\vartheta) +  \Pr(\Delta>\vartheta)) + C(\zeta_1\vee p^{-c})\sqrt{1 \vee \log(p/\zeta_1)} +5\zeta_2,
\end{align*}
where
$$
\pi(\vartheta):=C\vartheta^{1/3}(1\vee\log(p/\vartheta))^{2/3}+Cp^{-c}\sqrt{1\vee\log(p/\vartheta)}
$$
and $c,C>0$ depend only on $\nu,\nu^\prime,c_{1}$ and $C_{1}$. In addition,
$$
\sup_{\alpha\in(0,1)}\left|\Pr(T\leq c_W(\alpha))-\alpha\right|\leq \rho_{\ominus}+\rho.
$$
\end{theorem}
\begin{corollary}\label{cor: multiplier bootstrap}
Suppose that there exist constants $c_2,C_2>0$ such that
conditions (\ref{eq: statistic approximation}) and (\ref{eq: conditional quantiles}) hold with $\zeta_1 \sqrt{\log p} + \zeta_2 \leq C_{2} n^{-c_{2}}$.
Moreover, suppose that  one of the following conditions is satisfied:
(i) (E.5) holds and $B_n^2 (\log (pn))^7/n\leq C_2 n^{-c_2}$ or
(ii) (E.6) holds and $B_n^4 (\log (pn))^7/n\leq C_2 n^{-c_2}$.
Finally, suppose that condition (SM) holds and $p\geq C_3n^{c_3}$ for some constants $c_3>0$ and $C_3>0$.
Then there exist constants $c > 0$ and $C > 0$ depending only on $\nu,\nu^\prime,c_{1},C_1,c_{2},C_2,c_3$, and $C_{3}$ such that
\begin{equation*}
\rho_{\ominus}=\sup_{\alpha\in(0,1)}\Pr(\{T\leq c_{W}(\alpha)\}\ominus\{T_0\leq c_{Z_0}(\alpha)\}) \leq Cn^{-c}.
\end{equation*}
In addition, $\sup_{\alpha \in (0,1)} | \Pr(T\leq c_W(\alpha)) - \alpha |\leq \rho_{\ominus}+\rho\leq Cn^{-c}$.
\end{corollary}

The proofs rely on the following auxiliary lemmas, whose proofs will be given below.

\begin{lemma}\label{lem: anticoncentration_extended}
(a) Let $Y_{1},\dots,Y_{p}$ be jointly Gaussian random variables with $\Ep[Y_{j}]=0$ and $\sigma_{j}^{2} := \Ep [Y_{j}^{2} ] $ for all $1 \leq j \leq p$. Let $b_{p} := \Ep [ \max_{1 \leq j \leq p} Y_{j} ]$ and $\bar{\sigma} = \max_{1 \leq j \leq p} \sigma_{j}>0$. Assume that $b_p\geq c_1\sqrt{\log p}$ for some $c_1>0$. Then for every $\varsigma > 0$,
\begin{equation}\label{eq: anticoncentration_extended}
\sup_{z \in \RR} \Pr \left( | \max_{1 \leq j \leq p} Y_{j} - z|  \leq   \varsigma \right) \leq C (\varsigma\vee p^{-c}) \left(b_p + \sqrt{1 \vee \log (\bar{\sigma}/\varsigma) }\right)
\end{equation}
where $c,C>0$ are some constants depending only on $c_1$ and $\bar{\sigma}$. (b) Furthermore, the worst case bound is obtained by bounding $b_p$ by $\bar{\sigma}\sqrt{2\log p}$.
\end{lemma}
\begin{lemma}\label{lem: gaussian comparison extended}
Let $V$ and $Y$ be centered Gaussian random vectors in $\RR^p$ with covariance matrices $\Sigma^V$ and $\Sigma^Y$, respectively. Let $\Delta_0:=\max_{1\leq j,k\leq p}|\Sigma^V_{jk}-\Sigma^Y_{jk}|$. Suppose that there are some constants $0<c_1<C_1$ such that $\bar{\sigma}:=\max_{1\leq j\leq p}\Ep[Y_j^2]\leq C_1$ for all $1\leq j\leq p$ and $b_p:=\Ep[\max_{1\leq j\leq p}Y_j]\geq c_1\sqrt{\log p}$. Then there exist constants $c>0$ and $C>0$ depending only on $c_1$ and $C_1$ such that
\begin{align*}
\sup_{t\in\RR}\left|\Pr\left(\max_{1 \leq j \leq p} V_{j}\leq t\right)-\Pr\left(\max_{1 \leq j \leq p} Y_{j}\leq t\right)\right| \leq  &C\Delta_{0}^{1/3}(1 \vee \log (p/\Delta_0))^{2/3}\\
&+Cp^{-c}\sqrt{1\vee\log(p/\Delta_0)}.
\end{align*}
\end{lemma}

\begin{proof}[Proof of Theorem \ref{thm: main_result1_extended}]
It follows from Theorem 2.3.16 in \cite{Dudley1999} that condition (SM) implies that $\Ep[Z_0]\geq c\sqrt{\log p}$ for some $c>0$ that depends only on $\nu,\nu^\prime$, and $c_1$. Therefore, using the argument like that in the proof of Theorem \ref{cor: Gaussian to nonGaussian KS 2} with an application of Lemma \ref{lem: anticoncentration_extended} instead of Lemma \ref{lem: anticoncentration}, we obtain
\begin{eqnarray}
& \rho  \leq
 C \big [ n^{-1/2} (\psi^3+\psi^2\beta+\psi\beta^2)M_3^3
+(\psi^2+\psi\beta) \bar \varphi(u) \nonumber \\
& \ \ \ \ \ \ \ \ \qquad  + \psi \bar \varphi(u)\sqrt{\log(p/\gamma)}+( \beta^{-1} \log p + \psi^{-1} + p^{-c}) \sqrt{1 \vee \log (p \psi)} \big ]
\label{key to}
\end{eqnarray}
where all notation is taken from the proof of Theorem \ref{cor: Gaussian to nonGaussian KS 2}. Recall that $\bar{\varphi}(\cdot)$ is any function satisfying $\bar{\varphi}(u)\geq \varphi(u)$ for all $u>0$ and $\varphi(u)=\varphi_x(u)\vee\varphi_y(u)$. To bound $\varphi_x(u)$, we have
\begin{align*}
\barEp[ x_{ij}^{2} 1\{  | x_{ij} | > u (\bar \Ep [ x_{ij}^{2} ])^{1/2} \} ]&\leq \barEp[x_{ij}^4]/(u^2\barEp[x_{ij}^2])\\
&=\barEp[x_{ij}^4]/(u^2\barEp[x_{ij}^2]^2)\barEp[x_{ij}^2]\leq (M_{4,2}^4/u^2)\barEp[x_{ij}^2].
\end{align*}
This implies that $\varphi_x(u)\leq M_{4,2}^2/u$. To bound $\varphi_y(u)$, note that $\barEp[y_{ij}^4]\leq 3\barEp[x_{ij}^4]$, which was shown in the proof of Theorem \ref{cor: Gaussian to nonGaussian KS 2}. In addition, $\barEp[y_{ij}^2]=\barEp[x_{ij}^2]$. Therefore, $\varphi_y(u)\leq C\varphi_x(u)$. Hence, we can set $\bar{\varphi}(u):=CM_{4,2}^2/u$ for all $u>0$. The rest of the proof is the same as that for Theorem \ref{cor: Gaussian to nonGaussian KS 2} with $M_4$ replaced by $M_{4,2}$.
\end{proof}
\begin{proof}[Proof of Corollary \ref{cor: central limit theorem extended}]
Note that in both cases, $M_{4,2}^2\leq CB_n$ and
$$
M_3^3=\max_{1\leq j\leq p}\barEp[|x_{ij}|^3]\leq M_{3,2}^3\max_{1\leq j\leq p}\barEp[x_{ij}^2]^{3/2}\leq CM_{3,2}^3\leq CB_n.
$$
Therefore, the claim of the corollary follows from Theorem \ref{cor: central limit theorem extended} by the same argument as that leading to Corollary \ref{cor: central limit theorem} from Theorem \ref{cor: Gaussian to nonGaussian KS 2}.
\end{proof}
\begin{proof}[Proof of Theorem \ref{thm: multiplier bootstrap}]
The proof is the same as that for Theorem \ref{thm: multiplier bootrstrap II} with Lemmas \ref{lem: anticoncentration_extended} and \ref{lem: gaussian comparison extended} replacing Lemmas \ref{lem: anticoncentration} and \ref{lemma: distances Gaussian to Gaussian}.
\end{proof}
\begin{proof}[Proof of Corollary \ref{cor: multiplier bootstrap}]
Since $B_n\geq 1$, both under (E.5) and under (E.6) we have $(\log(pn))^7/n\leq C_2n^{-c_2}$. Let $\tilde{\zeta}_1:=\zeta_1\vee n^{-1}$. Then conditions (\ref{eq: statistic approximation}) and (\ref{eq: conditional quantiles}) hold with $(\tilde{\zeta}_1,\zeta_2)$  replacing $(\zeta_1,\zeta_2)$ and $\tilde{\zeta}_1\sqrt{\log p}+\zeta_2\leq Cn^{-c}$. Further, since $p\geq C_3n^{c_3}$, we have $p^{-c}(1\vee \log(p/\tilde{\zeta}_1))^{1/2}\leq Cn^{-c}$.

Let $\vartheta=\vartheta_n:=((\Ep[\Delta])^{1/2}/\log p)\vee n^{-1}$. Then $p^{-c}(1\vee \log(p/\vartheta))^{1/2}\leq Cn^{-c}$. In addition, if $\vartheta=n^{-1}$, then $\vartheta^{1/3}(\log(p/\vartheta))^{2/3}\leq Cn^{-c}$.
Finally,
$$
M_4^2=\max_{1\leq j\leq p}\barEp[x_{ij}^4]^{1/2}\leq M_{4,2}^2\max_{1\leq j\leq p}\barEp[x_{ij}^2]\leq CM_{4,2}^2\leq CB_n.
$$
The rest of the proof is similar to that for Corollary \ref{cor: multiplier bootstrap examples}.
\end{proof}
\begin{proof}[Proof of Lemma \ref{lem: anticoncentration_extended}]
In the proof, several constants will be introduced. All of these constants are implicitly assumed to depend only on $c_1$ and $\bar{\sigma}$.

We choose $c>0$ such that $4\bar{\sigma}\sqrt{c}=c_1$.
Fix $\varsigma>0$. It suffices to consider the case $\varsigma\geq \bar{\sigma}p^{-c}$.
Let $\underline{\sigma}:=c_2b_p/\sqrt{\log p}$ for sufficiently small $c_2>0$ to be chosen below. Note that $\underline{\sigma}\geq c_1c_2$. So, if $\varsigma>\underline{\sigma}$, (\ref{eq: anticoncentration_extended}) holds trivially by selecting sufficiently large $C$.

Consider the case $\varsigma\leq \underline{\sigma}$. Assume that $z>b_p+\varsigma+\bar{\sigma}\sqrt{2\log(\underline{\sigma}/\varsigma)}$. Then
\begin{align*}
\Pr( | \max_{1 \leq j \leq p} Y_{j} - z|  \leq   \varsigma )&\leq \Pr(\max_{1\leq j\leq p}Y_j\geq z-\varsigma)\\
&\leq\Pr(\max_{1\leq j\leq p}Y_j\geq b_p+\bar{\sigma}\sqrt{2\log(\underline{\sigma}/\varsigma)})\leq \varsigma/\underline{\sigma}
\end{align*}
where the last inequality follows from Borell's inequality. So, (\ref{eq: anticoncentration_extended}) holds by selecting sufficiently large $C$.

Now assume that $z<b_p-\varsigma-\bar{\sigma}\sqrt{2\log(\underline{\sigma}/\varsigma)}$. Then
\begin{align*}
\Pr( | \max_{1 \leq j \leq p} Y_{j} - z|  \leq   \varsigma )&\leq \Pr(\max_{1\leq j\leq p}Y_j\leq z+\varsigma)\\
&\leq\Pr(\max_{1\leq j\leq p}Y_j\leq b_p-\bar{\sigma}\sqrt{2\log(\underline{\sigma}/\varsigma)})\leq \varsigma/\underline{\sigma}
\end{align*}
where the last inequality follows from Borell's inequality. So, (\ref{eq: anticoncentration_extended}) holds by selecting sufficiently large $C$.

Finally, assume that
$$
b_p-\varsigma-\bar{\sigma}\sqrt{2\log(\underline{\sigma}/\varsigma)}\leq z\leq b_p+\varsigma+\bar{\sigma}\sqrt{2\log(\underline{\sigma}/\varsigma)}.
$$
Then
$$
\Pr( | \max_{1 \leq j \leq p} Y_{j} - z|\leq\varsigma)  \leq \Pr( | \max_{j\in\tilde{J}} Y_{j} - z| \leq\varsigma) + \Pr( | \max_{j\in J\backslash\tilde{J}} Y_{j} - z|\leq\varsigma) =: I+II
$$
where $J:=\{1,\dots,p\}$ and $\tilde{J}:=\{j\in J:\sigma_j\leq \underline{\sigma}\}$.
Consider $I$.
We have
\begin{equation*}
b_p\geq_{(1)} c_1\sqrt{\log p}=_{(2)}4\bar{\sigma}\sqrt{c\log p}= 4\bar{\sigma}\sqrt{\log(p^c)}\geq_{(3)} 4\bar{\sigma}\sqrt{\log(\bar{\sigma}/\varsigma)}
\end{equation*}
where (1) holds by assumption, (2) follows from the definition of $c$, and (3) holds because $\varsigma\geq \bar{\sigma}p^{-c}$.
In addition, there exists $C_2>0$ such that $\Ep[\max_{j\in\tilde{J}}Y_j]\leq c_2C_2b_p$, and there exist $C_3>0$ such that $b_p\leq C_3\bar{\sigma}\sqrt{\log p}$, so that $\underline{\sigma}\leq c_2C_3\bar{\sigma}$. We choose $c_2$ so that $c_2C_2\leq 1/4$, $c_2/\sqrt{\log p}\leq 1/8$, and $c_2C_3\leq 1$.  Then $\underline{\sigma}\leq \bar{\sigma}$, $\underline{\sigma}\leq b_p/8$, and $\Ep[\max_{j\in\tilde{J}}Y_j]\leq b_p/4$. Also recall that $\varsigma\leq \underline{\sigma}$. Therefore,
\begin{align*}
b_p-2\varsigma-\bar{\sigma}\sqrt{2\log(\underline{\sigma}/\varsigma)}-\Ep[\max_{j\in\tilde{J}}Y_j]
&\geq b_p/2-\bar{\sigma}\sqrt{2\log(\bar{\sigma}/\varsigma)}\\
&\geq \bar{\sigma}\sqrt{2\log(\bar{\sigma}/\varsigma)}
\geq \underline{\sigma}\sqrt{2\log(\underline{\sigma}/\varsigma)}.
\end{align*}
So, Borell's inequality yields
$$
I\leq \Pr(\max_{j\in\tilde{J}}Y_j\geq z-\varsigma)\leq \Pr(\max_{j\in\tilde{J}}Y_j\geq b_p-2\varsigma-\bar{\sigma}\sqrt{2\log(\underline{\sigma}/\varsigma)})\leq \varsigma/\underline{\sigma}
$$
because $\sigma_j\leq \underline{\sigma}$ for all $j\in\tilde{J}$.

Consider $II$. It is proved in \cite{ChernozhukovChetverikovKato2012c}
that
\begin{equation}\label{eq: standard_anticoncentration}
II\leq 4\varsigma\{(1/\underline{\sigma}-1/\bar{\sigma})|z|+a_p+1\}/\underline{\sigma}
\end{equation}
where $a_p:=\Ep[\max_{j\in J\backslash\tilde{J}}Y_j/\sigma_j]$. See, in particular, equation (16) in that paper. Note that $a_p\leq b_p/\underline{\sigma}$. Therefore, (\ref{eq: standard_anticoncentration}) combined with our restriction on $z$ yields
\begin{align*}
II&\leq 4\varsigma\left(2b_p+\varsigma+\bar{\sigma}\sqrt{2\log(\underline{\sigma}/\varsigma)}
+\underline{\sigma}\right)/\underline{\sigma}^2\\
&\leq 4\varsigma\left(2b_p+2\underline{\sigma}+\bar{\sigma}\sqrt{2\log(\bar{\sigma}/\varsigma)}\right)/\underline{\sigma}^2
\end{align*}
where in the second line we used the facts that $\varsigma\leq \underline{\sigma}$ and $\underline{\sigma}\leq \bar{\sigma}$ by assumption and by construction, respectively.
Now (\ref{eq: anticoncentration_extended}) holds by selecting sufficiently large $C>0$, and using the fact that $\underline{\sigma}\geq c_1c_2>0$. This completes the proof.
\end{proof}
\begin{proof}[Proof of Lemma \ref{lem: gaussian comparison extended}]
The proof is the same as that for Theorem 2 in \cite{ChernozhukovChetverikovKato2012c} with Lemma \ref{lem: anticoncentration_extended} replacing Lemma \ref{lem: anticoncentration}.
\end{proof}

\section{Validity of Efron's Empirical bootstrap}
\label{sec: nonparametric bootstrap}

In this section, we study the validity of the empirical (or Efron's) bootstrap in approximating the distribution of $T_{0}$ in the simple case where $x_{ij}$'s are uniformly bounded (the bound can increase with $n$).
Moreover, we consider here the asymptotics where $n \to \infty$ and possibly $p=p_{n} \to \infty$.
Recall the setup in Section \ref{sec: Gaus vs NonGaus}: let $x_{1},\dots,x_{n}$ be independent centered random vectors in $\RR^{p}$ and define
\[
T_{0}= \max_{1 \leq j \leq p} \frac{1}{\sqrt{n}} \sum_{i=1}^{n} x_{ij}.
\]
The empirical bootstrap procedure is described as follows.
Let $x_{1}^{*},\dots,x_{n}^{*}$ be i.i.d. draws from the empirical distribution of $x_{1},\dots,x_{n}$.
Conditional on $(x_{i})_{i=1}^{n}$, $x_{1}^{*},\dots,x_{n}^{*}$ are i.i.d. with mean $\En [ x_{i} ]$ and covariance matrix $\En [ (x_{i}-\En[x_{i}])(x_{i}-\En[ x_{i}])' ]$. Define
\[
T_{0}^{*}= \max_{1 \leq j \leq p} \frac{1}{\sqrt{n}} \sum_{i=1}^{n} (x_{ij}^{*}-\En[ x_{ij} ]).
\]
The empirical bootstrap approximates the distribution of $T_{0}$ by the conditional distribution $T_{0}^{*}$ given $(x_{i})_{i=1}^{n}$.

Recall
\[
W_{0} = \max_{1 \leq j \leq p} \frac{1}{\sqrt{n}} \sum_{i=1}^{n}e_{i} x_{ij},
\]
where $e_{1},\dots,e_{n}$ are i.i.d. $N(0,1)$ random variables independent of $(x_{i})_{i=1}^{n}$.
We shall here compare the conditional distribution of $T_{0}^{*}$ to that of $W_{0}$.

\begin{theorem}
\label{prop: efron}
Suppose that there exists constants $C_{1} > c_{1} > 0$ and a sequence $B_{n} \geq 1$ of constants such that
$c_{1} \leq \barEp [ x_{ij}^{2} ] \leq C_{1}$ for all $1 \leq j \leq p$ and $| x_{ij} | \leq B_{n}$ for all $1 \leq i \leq n$ and $1 \leq j \leq p$.
Then provided that $B_{n}^{2}( \log (pn) )^{7} = o(n)$, with probability $1-o(1)$,
\[
\sup_{t \in \RR} | \Pr \{ T_{0}^{*} \leq t \mid (x_{i})_{i=1}^{n} \} - \Pr \{ W_{0} \leq t \mid (x_{i})_{i=1}^{n} \} | = o(1).
\]
\end{theorem}

This theorem shows the asymptotic equivalence of the empirical and Gaussian multiplier bootstraps.
The validity of the empirical bootstrap (in the form similar to that in Theorem \ref{thm: multiplier bootrstrap I}) follows relatively directly from the validity of the Gaussian multiplier bootstrap.

\begin{proof}[Proof of Theorem \ref{prop: efron}]
The proof consists of three steps.

\textbf{Step 1}. We first show that with probability $1-o(1)$, $c_{1}/2 \leq \En[( x_{ij}-\En[x_{ij}])^{2} ] \leq 2 C_{1}$ for all $1 \leq j \leq p$.
By Lemma \ref{lem: symmetrization inequality 1},
\begin{align*}
&\Ep \left [ \max_{1 \leq j \leq p} |\En [x_{ij}] | \right ] \lesssim \sqrt{C_{1} (\log p)/n} + B_{n}(\log p)/n = o((\log p)^{-1/2}), \\
&\Ep \left [ \max_{1 \leq j \leq p} |\En [x_{ij}^{2}] - \barEp[ x_{ij}^{2} ]| \right ] \lesssim \sqrt{C_{1}B_{n}^{2} (\log p)/n} + B_{n}^{2}(\log p)/n = o(1),
\end{align*}
so that uniformly in $1 \leq j \leq p$, $| \En[( x_{ij}-\En[x_{ij}])^{2} ] - \barEp [ x_{ij}^{2} ] | = o_{\Pr}(1)$, which implies the desired assertion.

\textbf{Step 2}. Define
\[
\check{W}_{0} = \max_{1 \leq j \leq p} \frac{1}{\sqrt{n}} \sum_{i=1}^{n}e_{i} (x_{ij} - \En [x_{ij}]).
\]
We show that with probability $1-o(1)$,
\begin{equation}
\sup_{t \in \RR} | \Pr \{ T_{0}^{*} \leq t \mid (x_{i})_{i=1}^{n} \} - \Pr \{ \check{W}_{0} \leq t \mid (x_{i})_{i=1}^{n} \} | = o(1). \label{conditionalGAR}
\end{equation}
Conditional on $(x_{i})_{i=1}^{n}$, $x_{i}^{*}-\En[x_{i}]$ are independent centered random vector in $\RR^{p}$ with covariance matrix $\En [ (x_{i} - \En [x_{i}])(x_{i} - \En [x_{i}])' ]$.
Hence conditional on $(x_{i})_{i=1}^{n}$, we can apply Corollary \ref{cor: central limit theorem} to $T_{0}^{*}$ to deduce (\ref{conditionalGAR}).

\textbf{Step 3}. We show that
with probability $1-o(1)$,
\[
\sup_{t \in \RR} | \Pr \{ \check{W}_{0} \leq t \mid (x_{i})_{i=1}^{n} \} - \Pr \{ W_{0} \leq t \mid (x_{i})_{i=1}^{n} \} | = o(1).
\]
By definition, we have
\[
| \check{W}_{0} - W_{0} | \leq \max_{1 \leq j \leq p} | \En[ x_{ij} ] | \times  \left | \frac{1}{\sqrt{n}} \sum_{i=1}^{n} e_{i}  \right | = o_{\Pr}((\log p)^{-1/2}).
\]
Hence using the anti-concentration inequality together with Step 1, we deduce the desired assertion.

The conclusion of Theorem \ref{prop: efron} follows from combining Steps 1-3.
\end{proof}

\section{Comparison of our Gaussian approximation results to other ones}\label{sec: literature review}

We first point out that our Gaussian approximation result (\ref{eq: main result}) can be viewed as a version of multivariate central limit theorem, which is concerned with conditions under which
\begin{equation}\label{eq: MCLT}
\left | \Pr\left( X \in A \right) - \Pr\left( Y  \in A \right) \right | \to 0,
 \end{equation}
uniformly in a collection of sets $A$, typically \textit{all} convex sets.
 Recall that $$X=\frac{1}{\sqrt{n}} \sum_{i=1}^{n} x_{i},$$ where $x_{1},\dots,x_{n}$ are independent centered random vectors in $\RR^{p}$ with possibly $p=p_{n} \to \infty$, and
 $$
 Y  =\frac{1}{\sqrt{n}} \sum_{i=1}^{n} y_{i},
 $$
 $y_{1},\dots, y_{n}$ independent random vectors with $y_{i} \sim N(0, \Ep [x_{i}x_{i}'])$.  In fact, the result (\ref{eq: main result}) in the main text can be rewritten as
\begin{equation}\label{eq: MCLT2}
\sup_{t \in \RR} \left | \Pr\left \{ X \in A_{\max}(t)  \right \}  - \Pr\left \{ Y  \in A_{\max}(t) \right \} \right | \to 0,
\end{equation}
where $A_{\max}(t)=\{ a \in \RR^p: \max_{1 \leq j \leq p} a_{j} \leq t\}$.

Hence, our paper contributes to the literature on multivariate central limit theorems with growing number of dimensions (see, among others, \cite{Nagaev1976,Portnoy1986,AsrievRotar89,Gotze1991,Bentkus2003}).
These papers are concerned with results of the form (\ref{eq: MCLT}), but either explicitly or implicitly require the condition that $p^c/n \to 0$ for some $c > 0$ (when specialized to a setting like our setup). Results in these papers rely on the anti-concentration results for Gaussian random vectors on the $\delta$-expansions of boundaries of
arbitrary convex sets $A$ (see \cite{Ball1993}). We restrict our attention to the class of sets of the form $A_{\max}(t)$ in (\ref{eq: MCLT2}). These sets have a special structure that allows us to deal with the case where $p \gg n$: in particular, concentration of measure on the $\delta$-expansion of boundary of $A_{\max}(t)$ is at most of order $$\delta \Ep[\max_{j \leq p} Y_j]$$ for Gaussian random vectors with unit variance (and separable Gaussian processes more generally), as shown in \cite{ChernozhukovChetverikovKato2012c} (see also Lemma \ref{lem: anticoncentration}).

There is  large literature on bounding the difference:
\begin{equation}\label{eq: MCLT3}
| \Ep[ H(X) ] - \Ep[ H(Y) ] |,
\end{equation}
for various smooth functions $H(\cdot) : \Bbb{R}^p \to \Bbb{R}$, in particular the recent work
includes \cite{GoldsteinRinott1996,ChatterjeeMeckes2008,ReinertRollin2009,ChenFang2011}. Any such bounds lead to Gaussian approximations, though the structure of $H$'s plays an important role in limiting the scope of this approximation.
Two methods in the literature that turned out most fruitful for deriving gaussian approximation results in high dimensional settings ($p\rightarrow\infty$ as $n\rightarrow\infty$ in our context) are those of Lindeberg and Stein. 
The history of the Lindeberg method dates back to Lindeberg's original proof of the central limit theorem (\cite{Lindeberg1922}), which has been revived in the recent literature.
 We refer to the introduction of \cite{Chatterjee2006} for a brief history on the Lindeberg's method; see also \cite{Chatterjee2005a}. 
The recent development on Stein's method when $x_{i}$'s are multivariate can be found in \cite{GoldsteinRinott1996,ChatterjeeMeckes2008,ReinertRollin2009,ChenFang2011}. See also \cite{BentkusReview} for a comprehensive overview of different methods.  In contrast to these papers, our paper analyzes a rather particular, yet important case $H(\cdot)= g(F_{\beta}(\cdot))$, with $g: \RR \to \RR$ (progressively less) smooth function and $0< \beta \to \infty$, and in our case,
self-normalized truncation, some fine properties of the smooth potential $F_\beta$, maximal fourth order moments
 of variables and of their envelops, play a critical role (as we comment further below), and so our main results can not be (immediately) deduced from these prior results (nor do we attempt to follow this route).

Using the Lindeberg's method and the smoothing technique of Bentkus \cite{Bentkus90}, \cite{NorvaisaPaulauskas91} derived in their Theorem 5 a Gaussian approximation result that is of similar form to our {\em simple} (non-main) Gaussian approximation result  presented in Section \ref{sec: elementary GAR} of the SM. However, in \cite{NorvaisaPaulauskas91}, the anti-concentration property is {\em assumed} (see equation (1.4) in their paper); in our notation, their assumption on anti-concentration says that there exists a constant $C$ independent of $n$ and $p$ such that
$$
\Pr(r < \max_{1\leq j\leq p} | Y_{j} | \leq r + \epsilon) \leq C \epsilon (1+r)^{-3}.
$$
This assumption is useful for the analysis of the Donsker case, but does not apply in our (non-Donsker) cases. In fact, it rules out the simple case where $Y_{1},\dots,Y_{p}$ are independent (i.e., the coordinates in $x_{i}$ are uncorrelated) or $Y_{1},\dots,Y_{p}$ are weakly dependent (i.e., the coordinates in $x_{i}$ are weakly correlated).  In addition, it is worth pointing out that the use of Bentkus's \cite{Bentkus90}  smoothing\footnote{Bentkus gave a proof of existence of a smooth function that approximates a supremum norm, with certain properties. The use of these properties alone does not lead to sharp results in our case. We rely instead on smoothing by potentials from spin glasses, and their detailed properties, most importantly the stability property, noted in Lemma \ref{lemma: switching property}, which is readily established for this smoothing method.} instead of the smoothing by potentials from spin glasses used here, does not lead to optimal results in our case, since very subtle properties (stability property noted in Lemma \ref{lemma: switching property}) of potentials play an important role in our proofs, and in particular, is crucial for getting a reasonable exponent in the dependence on $\log p$.

Chatterjee \cite{Chatterjee2005a}, who also used the Lindeberg method, analyzed a spin-glass example like  our example (E.4) and also derived a result similar to our {\em simple} (non-main) Gaussian approximation result presented in Section \ref{sec: elementary GAR} of the SM, where  $x_i = z_{ij} \epsilon_i$, where $(z_{ij})_{i=1}^p$ are fixed and $(\epsilon_i)_{i=1}^n$ are i.i.d
$\Bbb{R}$-valued and centered with bounded third moment.  We note \cite{Chatterjee2005a} only provided a result for smooth functionals, but the extension to non-smooth cases follows from our Lemma \ref{lem: anticoncentration} along with standard kernel smoothing. In fact, all of our paper is inspired by Chatterjee's work, and an early version employed (combinatorial) Lindeberg's method.    We discuss the limitations of Lindeberg's approach (in the canonical form) in the main text, where we motivate the use of a combination of Slepian-Stein method in conjunction with self-normalized truncation and subtle properties of the potential function that approximates the maximum function. Generalization of results of Chatterjee to the case where $\Bbb{R}$-valued $(\epsilon_i)_{i=1}^n$ are locally dependent are given in \cite{Rollin2011}, who uses Slepian-Stein methods for proofs and
gave a result similar to our {\em simple} (non-main) Gaussian approximation result presented in Section \ref{sec: elementary GAR} of the SM. Like \cite{Rollin2011}, we also use a version of Slepian-Stein methods, but the proof details (as well as results and applications) are quite  different, since we instead analyze the case where the data $(x_i)_{i=1}^n$ are $\RR^p$-valued (instead of $\RR$-valued) independent vectors, and since we have to perform truncation (to get good dependencies on the size of the envelopes, $\max_{1 \leq j \leq p}|x_{ij}|$)
and use subtle properties of the potential function to get our main results (to get good dependencies on  $\log p$).

Using an interesting modification of the Lindeberg method, \cite{HKM2012} obtained an invariance principle for a sequence of sub-Gaussian $\Bbb{R}$-valued random variables $(\epsilon_i)_{i=1}^n$ (instead of $\Bbb{R}^p$-valued case consider here). Specifically, they looked at the large-sample probability of $(\epsilon_i)_{i=1}^n$ hitting a polytope formed by $p$ half-spaces, which is on the whole a different problem than studied here (though tools are insightful, e.g. the novel use of results developed by Nazarov \cite{Nazarov03}). These results have no intersection with our results,  except for a special case of ``sub-exponential regression/spin-glass example" (E.3), if we further require in that example, that $\varepsilon_{ij}=\varepsilon_{i}$ for all $1\leq j\leq p$, that $\varepsilon_i$'s are sub-Gaussian, that $\Ep[\varepsilon_i^3]=0$,  and that $B_n^2(\log p)^{16}/n\rightarrow 0$. All of these conditions and especially the last one are substantively more restrictive than what is obtained for the example (E.3) in our Corollary \ref{cor: central limit theorem}.


Finally, we note that when $x_i$'s are identically distributed in addition to being independent, the theory of strong approximations and, in particular, Hungarian coupling can also be used to obtain results like that in (\ref{eq: main result}) in the main text under conditions permitting $p \gg n$; see, for example, Theorem 3.1 in Koltchinskii \cite{Koltchinskii1994} and Rio \cite{Rio1994}. However, in order for this theory to work, $x_i$ have to be well approximable in a Haar basis when considered as functions on the underlying probability space -- e.g., $x_{ij} = f_{j,n}(u_i)$, where $f_{j,n}$ should have a total variation norm with respect to $u_i \sim U(0,1)^d$ (where $d$ is fixed) that does not grow too quickly to enable the expansion in the Haar basis.  This technique has been proven fruitful in many applications,  but this requires a radically different structure than what our leading applications impose; instead our results, based upon Slepian-Stein methods, are more readily applicable in these settings (instead of controlling total variation bounds, they rely on control of maxima moments and moments of envelopes of $\{x_{ij}, j \leq p\}$).  For further theoretical comparisons of the two methods in the context of strong approximations of suprema of non-Donsker emprical processes by those of Gaussian processes,  in the classical kernel and series smoothing examples,  we refer to our companion work \cite{ChernozhukovChetverikovKato2012b} (there is no winner in terms of guaranteed rates of approximation, though side conditions seem to be weaker for the Slepian-Stein type methods; in particular Hungarian couplings often impose the boundedness conditions, e.g. $\|x_i\|_\infty \leq B_n$).


\newpage

\begin{center}
{\Large  Supplementary Material III} \\
\text{} \\
{\textbf{Additional Application}} \\
\end{center}


\section{Adaptive Specification Testing}
\label{sub: AST}

In this section, we study the problem of adaptive specification
testing. Let $(v_{i},y_{i})_{i=1}^{n}$ be a sample of independent
random pairs where $y_{i}$ is a scalar dependent random variable,
and $v_{i} \in \RR^{d}$ is a vector of non-stochastic covariates.
The null hypothesis, $H_{0}$, is that there exists $\beta\in\RR^{d}$
such that
\begin{equation}
\Ep[y_{i}]=v_{i}'\beta;\, i=1,\dots,n.\label{eq: null}
\end{equation}
The alternative hypothesis, $H_{a}$, is that there is no $\beta$
satisfying (\ref{eq: null}). We allow for triangular array asymptotics
so that everything in the model may depend on $n$. For brevity, however,
we omit index $n$.

Let $\eps_{i}=y_{i}-\Ep[y_{i}]$, $i=1,\dots,n$.
Then $\Ep[\eps_{i}]=0$, and under $H_{0}$, $y_{i}=v_{i}'\beta+\eps_{i}$.
To test $H_{0}$, consider a set of test functions $P_{j}(v_{i})$,
$j=1,\dots,p$. Let $z_{ij}=P_{j}(v_{i})$. We choose test functions
so that $\En[z_{ij}v_{i}]=0$ and $\En[z_{ij}^{2}]=1$ for all
$j=1,\dots,p$. In our analysis, $p$ may be higher or even much higher
than $n$. Let $\hat{\beta}=(\En[v_{i}v_{i}^{\prime}])^{-1}(\En[v_{i}y_{i}])$
be an OLS estimator of $\beta$, and let $\hat{\eps}_{i}=y_{i}-z_{i}^{\prime}\hat{\beta};\, i=1,\dots,n$
be corresponding residuals. Our test statistic is
\begin{equation*}
T:=\max_{1\leq j\leq p}\frac{\left|\sum_{i=1}^{n}z_{ij}\hat{\eps}_{i}/\sqrt{n}\right|}{\sqrt{\En[z_{ij}^{2}\hat{\eps}_{i}^{2}]}}.
\end{equation*}
The test rejects $H_{0}$ if $T$ is significantly large.

Note that since $\En[z_{ij}v_{i}]=0$,
we have
\begin{equation*}
\sum_{i=1}^{n}z_{ij}\hat{\eps}_{i}/\sqrt{n}=\sum_{i=1}^{n}z_{ij}(\eps_{i}+v_{i}^{\prime}(\beta-\hat{\beta}))/\sqrt{n}=\sum_{i=1}^{n}z_{ij}\eps_{i}/\sqrt{n}.
\end{equation*}
Therefore, under $H_{0}$,
\begin{equation*}
T=\max_{1\leq j\leq p}\frac{\left|\sum_{i=1}^{n}z_{ij}\eps_{i}/\sqrt{n}\right|}{\sqrt{\En[z_{ij}^{2}\hat{\eps}_{i}^{2}]}}.
\end{equation*}
This suggests that we can use the multiplier bootstrap to obtain a critical value for the test. More precisely, let $(e_{i})_{i=1}^{n}$ be a sequence of  independent $N(0,1)$ random variables that are independent of the data, and let
\begin{equation*}
W:=\max_{1\leq j\leq p}\frac{\left|\sum_{i=1}^{n}z_{ij}\hat\eps_{i}e_{i}/\sqrt{n}\right|}{\sqrt{\En[z_{ij}^{2}\hat{\eps}_{i}^{2}]}}.
\end{equation*}
The multiplier bootstrap critical value $c_W(1-\alpha)$ is the conditional $(1-\alpha)$-quantile of $W$ given the data.
To prove the validity of multiplier bootstrap, we will impose the following condition:

\begin{itemize}
\item[(S)] There are some  constants $c_{1} > 0, C_{1} > 0,\bar \sigma^2 > 0, \underline{\sigma}^2>0$, and a sequence $B_{n} \geq 1$ of constants such that for
all $1 \leq i \leq n$, $1 \leq j \leq p$, $1\leq k\leq d$:
(i) $|z_{ij}|\leq B_n$;
(ii) $\En[z_{ij}^2]= 1$;
(iii) $ \underline \sigma^2 \leq \Ep[\eps_{i}^2] \leq \bar \sigma^2$;
(iv) $|v_{ik}|\leq C_{1}$;
(v) $d\leq C_{1}$;
 and (vi) the minimum eigenvalue of $\En[v_{i}v_{i}^\prime]$ is bounded from below by $c_{1}$.
\end{itemize}

\begin{theorem}[Size Control of Adaptive Specification Test]\label{thm: AST}
Let $c_{2} > 0$ be some constant.
Suppose that condition (S) is satisfied.  Moreover, suppose that either
\begin{enumerate}
\item[(a)] $\Ep[\eps_{i}^4]\leq C_{1}$ for all $1 \leq i \leq n$ and $B_n^4 (\log (pn))^7/n\leq C_{1} n^{-c_{2}}$; or
\item[(b)] $\Ep [ \exp ( | \eps_{i} | /C_{1} ) ]  \leq 2$ for all $1 \leq i \leq n$ and $B_n^2 (\log (pn))^7/n\leq C_{1} n^{-c_{2}}$.
\end{enumerate}
Then there exist constants $c >0$ and $C>0$, depending only on $c_{1},c_{2},C_{1},\underline{\sigma}^2$ and $\bar{\sigma}^2$, such that
under $H_0$, $| \Pr(T\leq c_W(1-\alpha)) - (1-\alpha) | \leq C n^{-c}$.
\end{theorem}
\begin{remark}The literature on specification testing is large. In particular, \cite{HorowitzSpokoiny2001} and \cite{GuerreLavergne2005} developed adaptive tests that are suitable for inference in $L_2$-norm. In contrast, our test is most suitable for inference in $\sup$-norm. An advantage of our procedure is that selecting a wide class of test functions leads to a test that can effectively adapt to a wide range of alternatives, including those that can not be well-approximated by H\"{o}lder-continuous functions. \qed
\end{remark}

\begin{proof}[Proof of Theorem \ref{thm: AST}]
We only consider case (a). The proof for case (b) is similar and hence omitted.
In this proof, let $c,c',C,C'$ denote generic positive constants depending only on $c_{1},c_2,C_{1},\underline{\sigma}^{2},\bar{\sigma}^{2}$ and their values may change from place to place.
Let
\begin{equation*}
T_0:=\max_{1\leq j\leq p}\frac{|\sum_{i=1}^nz_{ij}\eps_{i}/\sqrt{n}|}{\sqrt{\En[z_{ij}^2\sigma^2_{i}]}}  \ \text{and} \ W_0:=\max_{1\leq j\leq p}\frac{|\sum_{i=1}^nz_{ij}\eps_{i}e_{i}/\sqrt{n}|}{\sqrt{\En[z_{ij}^2\sigma^2_{i}]}}.
\end{equation*}
We make use of Corollary \ref{cor: multiplier bootstrap examples}-(ii). To this end, we shall verify conditions (\ref{eq: statistic approximation}) and (\ref{eq: conditional quantiles}) in Section \ref{sec: multiplier bootstrap} of the main text, which will be separately done in Steps 1 and 2, respectively.

\textbf{Step 1.} We show that $\Pr(|T-T_0|>\zeta_1)<\zeta_2$ for some $\zeta_1$ and $\zeta_2$ satisfying $\zeta_1\sqrt{\log p}+\zeta_2 \leq C n^{-c}$.

By Corollary \ref{cor: central limit theorem}-(ii), we have
\begin{align*}
&\Pr\left(\max_{1\leq j\leq p} | {\textstyle \sum}_{i=1}^nz_{ij}\eps_{i}/\sqrt{n} |> t \right) \\
&\quad \leq \Pr\left(\max_{1\leq j\leq p} | {\textstyle \sum}_{i=1}^nz_{ij}\sigma_ie_{i}/\sqrt{n} | >  t \right) + C n^{-c},
\end{align*}
uniformly in $t \in \RR$.  By the Gaussian concentration inequality \citep[Proposition A.2.1][]{VW96}, for every $t > 0$, we have
\begin{equation*}
\Pr\left(\max_{1\leq j\leq p} | {\textstyle \sum}_{i=1}^nz_{ij}\sigma_ie_{i}/\sqrt{n}  | >  \Ep [\max_{1\leq j\leq p} | {\textstyle \sum}_{i=1}^nz_{ij}\sigma_ie_{i}/\sqrt{n} | ] + C t  \right) \leq e^{-t^{2}}.
\end{equation*}
Since $\Ep [\max_{1\leq j\leq p}|  {\textstyle \sum}_{i=1}^nz_{ij}\sigma_ie_{i}/\sqrt{n} |] \leq C \sqrt{\log p}$, we conclude that
\begin{equation}
\Pr\left( \max_{1\leq j\leq p} | {\textstyle \sum}_{i=1}^nz_{ij}\eps_{i}/\sqrt{n} | > C \sqrt{\log (pn)} \right) \leq C' n^{-c}. \label{eq: step1-1}
\end{equation}
Moreover,
\begin{align*}
\En[z_{ij}^2(\hat{\eps}_{i}^2-\sigma_{i}^2)] &= \En[z_{ij}^2(\hat{\eps}_{i}-\eps_{i})^2]+\En[z_{ij}^2(\eps_{i}^2-\sigma_{i}^2)]+2\En[z_{ij}^2\eps_{i}(\hat{\eps}_{i}-\eps_{i})] \\
&=: I_{j} + II_{j} + III_{j}.
\end{align*}

Consider $I_{j}$. We have
\begin{equation*}
I_{j} \leq_{(1)} \max_{1\leq i\leq n}(\hat{\eps}_{i}-\eps_{i})^2 \leq_{(2)} C \Vert\hat{\beta}-\beta \Vert^2 \leq_{(3)}  C' \Vert\En[v_{i}\eps_{i}]\Vert^2,
\end{equation*}
where (1) follows from assumption S-(ii), (2) from S-(iv) and S-(v), and (3) from S-(vi).
Since $\Ep[\Vert\En[v_{i}\eps_{i}]\Vert^2] \leq C/n$, by Markov's inequality, for every $t > 0$,
\begin{equation}\label{eq: step1-2}
\Pr\left(\max_{1\leq j\leq p}\En[z_{ij}^2(\hat{\eps}_{i}-\eps_{i})^2]> t \right) \leq C/(nt).
\end{equation}

Consider $II_{j}$. By Lemma \ref{lem: symmetrization inequality 1} and Markov's inequality, we have
\begin{equation}\label{eq: step1-3}
\Pr\left(\max_{1\leq j\leq p}|\En[z_{ij}^2(\eps_{i}^2-\sigma_{i}^2)]|> t \right) \leq C B_{n}^2(\log p)/(\sqrt{n}t).
\end{equation}

Consider $III_{j}$.  We have $| III_{j} | \leq2|\En[z_{ij}^2 v_{i}^\prime(\beta-\hat{\beta})\eps_{i}]|\leq 2\Vert\En[z_{ij}^2\eps_{i}v_{i}]\Vert\Vert\hat{\beta}-\beta\Vert$.
Hence
\begin{align}
&\Pr\left(\max_{1\leq j\leq p}|\En[z_{ij}^2\eps_{i}(\hat\eps_{i}-\eps_{i})]|> t \right) \notag \\
&\leq \Pr\left(\max_{1\leq j\leq p}\Vert \En[z_{ij}^2\eps_{i}v_{i}]\Vert> t \right)+ \Pr ( \| \hat{\beta} - \beta \| > 1 ) \notag \\
&\leq C [ B_{n}^2(\log p)/(\sqrt{n}t) + 1/n ].\label{eq: step1-4}
\end{align}

By (\ref{eq: step1-2})-(\ref{eq: step1-4}), we have
\begin{equation}
\Pr \left ( \max_{1 \leq j \leq p} | \En[z_{ij}^2(\hat{\eps}_{i}^2-\sigma_{i}^2)] | > t \right ) \leq C[  B_{n}^2(\log p)/(\sqrt{n}t) + 1/(nt) + 1/n ]. \label{eq: step1-5}
\end{equation}
In particular,
\begin{equation*}
\Pr \left ( \max_{1 \leq j \leq p} | \En[z_{ij}^2(\hat{\eps}_{i}^2-\sigma_{i}^2)] | > \underline{\sigma}^{2}/2 \right ) \leq Cn^{-c}.
\end{equation*}
Since $\En [ z_{ij}^{2} \sigma_{i}^{2} ] \geq \underline{\sigma}^{2} > 0$  (which is guaranteed by S-(iii) and S-(ii)),  on the event $\max_{1 \leq j \leq p} | \En[z_{ij}^2(\hat{\eps}_{i}^2-\sigma_{i}^2)] | \leq \underline{\sigma}^{2}/2$, we have
\begin{equation*}
\min_{1 \leq j \leq p} \En [ z_{ij}^{2} \hat{\eps}_{i}^{2} ] \geq \min_{1 \leq j \leq p} \En [ z_{ij}^{2} \sigma_{i}^{2} ] - \underline{\sigma}^{2}/2 \geq \underline{\sigma}^{2}/2,
\end{equation*}
and hence
\begin{align*}
| T - T_{0} | &= \max_{1 \leq j \leq p} \left | \frac{\sqrt{\En [ z_{ij}^{2} \sigma_{i}^{2} ]}-\sqrt{\En [ z_{ij}^{2} \hat{\eps}_{i}^{2} ]}}{\sqrt{\En [ z_{ij}^{2} \hat{\eps}_{i}^{2} ]}} \right | \times T_{0} \\
&\leq C \max_{1 \leq j \leq p}  \left | \sqrt{\En [ z_{ij}^{2} \sigma_{i}^{2} ]}-\sqrt{\En [ z_{ij}^{2} \hat{\eps}_{i}^{2} ]} \right | \times T_{0} \\
&\leq C \max_{1 \leq j \leq p}  | \En [ z_{ij}^{2} \sigma_{i}^{2} ]-\En [ z_{ij}^{2} \hat{\eps}_{i}^{2} ] | \times T_{0},
\end{align*}
where the last step uses the simple fact that
\begin{align*}
| \sqrt{a} - \sqrt{b} | &= \frac{| a-b | }{\sqrt{a} + \sqrt{b}} \leq \frac{|a-b|}{\sqrt{a}}.
\end{align*}
By (\ref{eq: step1-1}) and (\ref{eq: step1-5}), for every $t > 0$,
\begin{equation*}
\Pr\left( |T-T_0|> Ct \sqrt{\log (pn)} \right) \leq C' [n^{-c}+B_{n}^2(\log p)/(\sqrt{n}t) + 1/(nt)].
\end{equation*}
By choosing $t=(\log (pn))^{-1}n^{-c'}$ with sufficiently small $c' >0$, we obtain the claim of this step.

\textbf{Step 2.}
We show that $\Pr(\Pr_e(|W-W_0|>\zeta_1)>\zeta_2)<\zeta_2$ for some $\zeta_1$ and $\zeta_2$ satisfying $\zeta_1\sqrt{\log p}+\zeta_2 \leq C n^{-c}$.

For $0 < t \leq \underline{\sigma}^{2}/2$, consider the event
\begin{equation*}
\mathcal{E} = \left \{ (\eps_{i})_{i=1}^{n} : \max_{1\leq j\leq p}|\En[z_{ij}^2(\hat{\eps}_{i}^2-\sigma_{i}^2)]|\leq t, \max_{1\leq i\leq p}(\hat{\eps}_{i}-\eps_{i})^2\leq t^{2} \right \}.
\end{equation*}
By calculations in Step 1, $\Pr (\mathcal{E}) \geq 1-C[B_{n}^2(\log p)/(\sqrt{n}t)+1/(nt^2)+1/n]$. We shall show that, on this event,
\begin{align}
&\Pr_e\left(\max_{1\leq j\leq p}| {\textstyle \sum}_{i=1}^nz_{ij}\hat{\eps}_{i}e_{i}/\sqrt{n}|>C\sqrt{\log (pn)}\right)\leq n^{-1}, \label{step2-1} \\
&\Pr_e\left( \max_{1\leq j\leq p} | {\textstyle \sum}_{i=1}^nz_{ij}(\hat{\eps}_{i}-\eps_{i}) e_{i}/\sqrt{n} | >Ct\sqrt{\log (pn)}\right) \leq n^{-1}. \label{step2-2}
\end{align}
For (\ref{step2-1}), by the Gaussian concentration inequality, for every $s > 0$,
\begin{equation*}
\Pr_e\left(\max_{1\leq j\leq p}| {\textstyle \sum}_{i=1}^nz_{ij}\hat{\eps}_{i}e_{i}/\sqrt{n}| > \Ep_{e} [\max_{1\leq j\leq p}| {\textstyle \sum}_{i=1}^nz_{ij}\hat{\eps}_{i}e_{i}/\sqrt{n}|] + Cs  \right) \leq e^{-s^{2}}.
\end{equation*}
where we have used the fact $\En[z_{ij}^2\hat{\eps}_{i}^2]=\En[z_{ij}^2\sigma_{i}^2]+\En[z_{ij}^2(\hat{\eps}_{i}^2-\sigma_{i}^2)]\leq \bar{\sigma}^{2}+t \leq \bar{\sigma}^{2} + \underline{\sigma}^{2}/2$ on the event $\mathcal{E}$.
Here $\Ep_{e} [\cdot]$ means the expectation with respect to $(e_{i})_{i=1}^{n}$ conditional on $(\eps_{i})_{i=1}^{n}$. Moreover, on the event $\mathcal{E}$,
\begin{equation*}
\Ep_{e} [\max_{1\leq j\leq p}| {\textstyle \sum}_{i=1}^nz_{ij}\hat{\eps}_{i}e_{i}/\sqrt{n}|] \leq C \sqrt{\log p}.
\end{equation*}
Hence by choosing $s=\sqrt{\log n}$, we obtain (\ref{step2-1}).
Inequality (\ref{step2-2}) follows similarly, by noting that  $(\En[z_{ij}^2(\hat{\eps}_{i}-\eps_{i})^2])^{1/2}\leq \max_{1\leq i\leq n}|\hat{\eps}_{i}-\eps_{i}|\leq t$ on the event $\mathcal{E}$.

Define
\begin{equation*}
W_{1} := \max_{1 \leq j \leq p} \frac{| \sum_{i=1}^{n} z_{ij} \hat{\eps}_{i} e_{i}/\sqrt{n}|}{\sqrt{\En[ z_{ij}^{2} \sigma_{i}^{2} ]}}.
\end{equation*}
Note that $\En[ z_{ij}^{2} \sigma_{i}^{2} ] \geq \underline{\sigma}^{2}$.
Since on the event $\mathcal{E}$, $\max_{1 \leq j \leq p} | \En [ z_{ij}^{2} (\hat{\eps}_{i}^{2} - \sigma_{i}^{2}) ]| \leq t \leq \underline{\sigma}^{2}/2$, in view of Step 1,
on this event, we have
\begin{align*}
| W - W_{0} | &\leq | W - W_{1} | + | W_{1} - W_{0} | \\
&\leq C t W_{1} + | W_{1} - W_{0} | \\
&\leq C t \max_{1\leq j\leq p}| {\textstyle \sum}_{i=1}^nz_{ij}\hat{\eps}_{i}e_{i}/\sqrt{n}| + C \max_{1\leq j\leq p} | {\textstyle \sum}_{i=1}^nz_{ij}(\hat{\eps}_{i}-\eps_{i}) e_{i}/\sqrt{n} |.
\end{align*}
Therefore, by (\ref{step2-1}) and (\ref{step2-2}), on the event $\mathcal{E}$, we have
\begin{equation*}
\Pr_e\left(|W-W_0|>Ct \sqrt{\log (pn)}\right) \leq  2n^{-1}.
\end{equation*}
By choosing $t=(\log (pn))^{-1}n^{-c}$ with sufficiently small $c > 0$, we obtain the claim of this step.

\textbf{Step 3.} Steps 1 and 2 verified conditions (\ref{eq: statistic approximation}) and (\ref{eq: conditional quantiles}) in Section \ref{sec: multiplier bootstrap} of the main text.
Theorem \ref{thm: AST} case (a) follows from Corollary \ref{cor: multiplier bootstrap examples}-(ii).
\end{proof}

\end{document}